\documentclass{article}

\usepackage[utf8]{inputenc}
\usepackage[T1]{fontenc}
\usepackage{geometry}
\usepackage{amsmath,amssymb,amsthm,amsfonts}

\usepackage{indentfirst}
\usepackage{graphicx}
\usepackage{verbatim}
\usepackage{mathrsfs}
\usepackage{mathbbol}
\usepackage{hyperref}
\usepackage{authblk}

\usepackage{tikz}
\usetikzlibrary{arrows}

\usepackage[english]{babel}
\newtheorem{thm}{Theorem}[section]
\newtheorem{teo}[thm]{Proposition}

\newtheorem{lem}[thm]{Lemma}
\newtheorem{defn}[thm]{Definition}
\newtheorem{condition}[thm]{Condition}
\numberwithin{equation}{section}

\newtheoremstyle{named}{}{}{\itshape}{}{\bfseries}{.}{.5em}{#3}
\theoremstyle{named}

\newcommand{\activationenergy}{\Gamma^{\text{PCA}}}
\newcommand{\activationenergym}{\Gamma_m}
\newcommand{\stab}{x_0}
\newcommand{\metauno}{x_1}

\newcommand{\metadue}{x_2}
\newcommand{\rr}[1]{{\normalfont\textrm{#1}}}
\newcommand{\newatop}[2]{\genfrac{}{}{0pt}{}{#1}{#2}}

\begin{document}
 \title{Effect of energy degeneracy on the transition time \\ for a series of metastable states: \\ application to Probabilistic Cellular Automata}

\author[
         {}\hspace{0.5pt}\protect\hyperlink{hyp:email1}{1},\protect\hyperlink{hyp:affil1}{a}
        ]
        {\protect\hypertarget{hyp:author1}{Gianmarco Bet}}

\author[
         {}\hspace{0.5pt}\protect\hyperlink{hyp:email2}{2},\protect\hyperlink{hyp:affil1}{a},\protect\hyperlink{hyp:corresponding}{$\dagger$}
        ]
        {\protect\hypertarget{hyp:author2}{Vanessa Jacquier}}

\author[
         {}\hspace{0.5pt}\protect\hyperlink{hyp:email3}{3},\protect\hyperlink{hyp:affil1}{a},\protect\hyperlink{hyp:affil2}{b}
        ]
        {\protect\hypertarget{hyp:author3}{Francesca R.~Nardi}}

\affil[ ]{
          \small\parbox{365pt}{
             \parbox{5pt}{\textsuperscript{\protect\hypertarget{hyp:affil2}{a}}}Università degli Studi di Firenze,
            \enspace
             \parbox{5pt}{\textsuperscript{\protect\hypertarget{hyp:affil1}{b}}}Eindhoven University of Technology
            }
          }

\affil[ ]{
          \small\parbox{365pt}{
             \parbox{5pt}{\textsuperscript{\protect\hypertarget{hyp:email1}{1}}}\texttt{\footnotesize\href{mailto:gianmarco.bet@unifi.it}{gianmarco.bet@unifi.it}},
             \parbox{5pt}{\textsuperscript{\protect\hypertarget{hyp:email2}{2}}}\texttt{\footnotesize\href{mailto:vanessa.jacquier@unifi.it}{vanessa.jacquier@unifi.it}},
             \parbox{5pt}{\textsuperscript{\protect\hypertarget{hyp:email3}{3}}}\texttt{\footnotesize\href{mailto:francescaromana.nardi@unifi.it}{francescaromana.nardi@unifi.it}}
            }
          }

\affil[ ]{
          \small\parbox{365pt}{
             \parbox{5pt}{\textsuperscript{\protect\hypertarget{hyp:corresponding}{$\dagger$}}}Corresponding author
            }
          }

\date{\today}

\maketitle

\begin{center}
 {\it To our friend and colleague Carlo Casolo}

\end{center}

\begin{abstract}

 We consider the problem of metastability for stochastic reversible dynamics with exponentially small transition probabilities. We generalize previous results in several directions. We give an estimate of the spectral gap of the transition matrix and of the mixing time of the associated dynamics in terms of the maximal stability level. These \textit{model-independent results} hold in particular for a large class of Probabilistic Cellular Automata (PCA), which we then focus on. We consider the PCA in a finite volume, at small and fixed magnetic field, and in the limit of vanishing temperature. This model is peculiar because of the presence of three metastable states, two of which are degenerate with respect to their energy. We identify rigorously the metastable states by giving explicit upper bounds on the stability level of every other configuration. We rely on these estimates to prove a recurrence property of the dynamics, which is a cornerstone of the \textit{pathwise approach} to metastability. Further, we also identify the metastable states according to the \textit{potential-theoretic approach} to metastability, and this allows us to give precise asymptotics for the expected transition time from any such metastable state to the stable state.


\medskip\noindent
 \emph{Keywords:} Stochastic dynamics, probabilistic cellular automata, metastability, potential theory, low temperature dynamics, mixing times. \\
\emph{MSC2020:}
60K35, 82C20, \emph{secondary}: 60J10, 60J45, 82C22, 82C26.
\\
 \emph{Acknowledgment:} The research of Francesca R.~Nardi was partially supported by the NWO Gravitation Grant 024.002.003--NETWORKS and by the PRIN Grant 20155PAWZB ``Large Scale Random Structures''.

\end{abstract}

\maketitle

\section{Introduction}

 Metastability is a phenomenon that occurs when a physical system is close to a first order phase transition. Among classical examples are super-saturated vapors and ferromagnetic materials in a hysteresis loop \cite{penrose1971rigorous}.
 The metastability phenomenon occurs only for some thermodynamical parameters when a system is trapped for a long time in a state different from the stable state. This is the so-called \emph{metastable state}. While the system is trapped, it behaves as if it was in equilibrium, except that at a certain time it makes a sudden transition from the metastable state to the stable state. Metastability occurs in several physical situations and this has led to the formulation of numerous models for metastable behavior. However, in each case, three interesting issues are typically investigated. The first is the study of the \emph{transition time} from any metastable state to any stable states. The fluctuations of the dynamics should facilitate the transition, but these are very unlikely, so the system is typically stuck in the metastable state for an exponentially long time. The second issue is the identification of certain configurations, the so-called \emph{critical configurations}, that trigger the transition. The system fluctuates in a neighborhood of the metastable state until it visits the set of critical configurations during the last excursion. After this, the system relaxes to equilibrium. The third and last issue is the study of the typical path that the system follows during the transition from the metastable state to the stable state, the so-called \emph{tube of typical trajectories}. This issue is especially interesting from a physics point of view.

 The goal of this paper is twofold. First we consider general dynamics with exponentially small transition probabilities and we give an estimate of the mixing time and of the spectral gap of the transition matrix in terms of the maximal stability level. Second, we focus on a specific Probabilistic Cellular Automata in a finite volume, at small and fixed magnetic field, in the limit of vanishing temperature and we prove some results describing the metastable behaviour of the system.

 Let us now discuss the two goals in detail, starting with a comparison between our estimates for the mixing time and the spectral gap and the literature on the topic.
 Similar results on the estimate of the spectral gap have been proved for the model of simulated annealing in \cite{holley1988simulated}. The authors use Sobolev inequalities to study the simulated annealing algorithm and they demonstrate that this approach gives detailed information about the rate at which the process is tending to its ground state. Thanks to this result the mixing time is estimated for Metropolis dynamics. Our model-independent theorems are a generalization of the result in \cite[Proposition 3.24]{nardi2016hitting} to reversible dynamics with exponentially small transition probabilities in finite volume.
 The analysis of the spectral gap between the zero eigenvalue and the next-smallest eigenvalue of the generator is very interesting for Markov processes, since it is useful to control convergence to equilibrium. In \cite{bovier2002metastability} the authors focus on the connection between metastability and spectral theory for the so-called \emph{generic Markov chains} under the assumption of non-degeneracy. In particular, they use spectral information to derive sharp estimates on the transition times. We refer also to \cite[Chapter 8 and 16]{bovier2016metastability}, where the authors incorporate all the previous results about the study of metastability through spectral data. In particular, they show that the spectrum of the generator decomposes into a cluster of very small real eigenvalues that are separated by a gap from the rest of the spectrum. In order to study our PCA, we extend their estimates of the spectral gap to the case of degenerate in energy metastable states. The states $\sigma$ and $\eta$ are degenerate metastable states if they have the same energy and the energy barrier between them is smaller then the energy barrier between a metastable state and the stable state (see Condition \ref{t:series00} for a precise formulation and see \cite[Chapter 16.5 point 3]{bovier2016metastability} for a discussion). To suit our purposes, we express these estimates as functions of the virtual energy instead of the Hamiltonian function, see Equation \eqref{Hamgeneral} for the specific definition and \cite{catoni1999simulated}, \cite{cirillo2015metastability}.

 Regarding the expected transition time, in \cite{cirillo2017sum} the authors consider series of two metastable states with decreasing energy in the framework of reversible finite state space Markov chains with exponentially small transition probabilities. Under certain assumptions, not only they find the (exponential) order of magnitude of the transition time from the first metastable state to the stable state, they also give an addition rule to compute the prefactor. We generalize their results on the mean transition time and their addition rule to a setting with several degenerate metastable states, see Section \ref{s:series} for details. 

 The second goal concerns a particular Probabilistic Cellular Automata (PCA). Cellular Automata (CA) are discrete–time dynamical systems on a spatially extended discrete space and are used in a wide range of applications, for example to model natural and social phenomena. Probabilistic Cellular Automata (PCA) are the stochastic version of Cellular Automata, where the updating rules are random, i.e., the configurations are chosen according to probability distributions determined by the neighborhood of each site. Mathematically, we consider PCA with parallel (synchronous) dynamics, i.e., systems of finite-states Markov chains whose distribution at time $n$ depends only on the states in a neighboring set at time $n-1$. PCA are characterized by a matrix of transition probabilities from any configuration $\sigma$ to any other configuration $\eta$ defined as a product of local transition probabilities as
\begin{equation*}
\begin{split}
     p(\sigma,\eta) :=\prod_{i\in\Lambda}p_{i,\sigma}(\eta(i)), \qquad \sigma, \eta \in \mathcal{X},
\end{split}
\end{equation*}
 where $\Lambda \subset \mathbb{Z}^2$ is a finite box with periodic boundary conditions and $\mathcal{X}=\{-1,+1\}^{\Lambda}$ is the set of all configurations. Here we consider a specific PCA in the class introduced by Derrida \cite{derrida1989dynamical}, where the local transition probability is a certain function of the sum of neighboring spins $S_{\sigma}(\cdot)$ \eqref{ESSE} and the external magnetic field $h$
\begin{equation*}
\begin{split}
 p_{i,\sigma}(a):=\frac{1}{1+\exp{\{-2\beta a(S_{\sigma}(i) +h)}\}}= \frac{1}{2}[1 +a \tanh\beta(S_{\sigma}(i) +h)].
\end{split}
\end{equation*}
 We obtain our PCA by summing only over the nearest neighbor sites, see \eqref{KAPPA} and Figure \ref{neighbors}.
 When the sum is carried out over a symmetric set, the resulting dynamics is reversible with respect to a suitable Gibbs–like measure $\mu$ defined via a translation invariant multi–body potential, see \eqref{hamiltonianFunction}.
 This measure depends on a parameter $\beta$ which can be thought of as the inverse of the temperature of the system. For small values of the temperature, the PCA is likely to be found in the local minima of the Hamiltonian associated to $\mu$. The metastable behavior of this model has been investigated on heuristic and numerical grounds in  \cite{bigelis1999critical}.
 A key quantity in the study of metastability is the \emph{energy barrier} from one of the metastable states to the stable state. This is the minimum, over all paths connecting the metastable to the stable state, of the maximal transition energy along each path, minus the energy of the starting configuration (see \eqref{Phimax}-\eqref{Phimin}). Intuitively, the energy barrier from $\eta$ to $\sigma$ is the energy that the system must overcome to reach $\eta$ starting from $\sigma$.

 For our choice of parameters, our PCA has one stable state $\underline{+1}$ and peculiarly \emph{three} metastable states, which we identify rigorously as $\{\underline{-1},\underline{c}^e, \underline{c}^o\}$. To prove this, we will construct for each configuration $\sigma \notin \{\underline{-1},\underline{c}^e, \underline{c}^o, \underline{+1} \}$ a path starting from $\sigma$ and ending in a lower energy state, such that the maximal energy, along the path, is lower than the \emph{energy barrier from $\underline{-1}$ to $\underline{+1}$}.  This leads to an explicit upper-bound $V^*$ for the stability level of every configuration except $\{\underline{-1},\underline{c}^e,\underline{c}^o,\underline{+1}\}$, in Lemma \ref{EST}, which we will refer to as our main technical tool.
 We rely on this estimate to prove two recurrence properties. The first is that, starting from any configuration, the system reaches the set $\{\underline{-1},\underline{c}^e, \underline{c}^o,\underline{+1}\}$ in a time smaller than $e^{\beta V^*}$ with probability exponentially close to one. The second is that starting from any configuration the system reaches $\underline{+1}$ in a time smaller than $e^{\beta \activationenergy}$. To prove this, we combine our main tool with the computation of the energy barrier $\activationenergy$ in \cite{cirillo2003metastability} to prove the second recurrence property. We remark that $\underline{c}^e$ and $\underline{c}^o$ are two degenerate metastable states, since they have the same energy and the energy barrier between them is zero. Hence, we will use the shorthand $\underline{c}=\{\underline{c}^e,\underline{c}^o\}$.

 In order to find sharp estimates of the transition time from $\underline{-1}$ to $\underline{+1}$, we extend in Section \ref{s:series}, and then verify, the three model-dependent conditions given in \cite{cirillo2017sum}. These are, respectively, our main technical tool, the property that starting from $\underline{-1}$ the system visits the chessboard $\underline{c}$ before reaching $\underline{+1}$ with high probability \cite{cirillo2003metastability}, and the computation of the constants $k_1$ and $k_2$ in \cite{cirillo2016sum}. In fact, the sharp estimates on the transition time which we give here were already stated in \cite{cirillo2016sum}, but the proof there missed some key steps, which we provide here. First, our Lemma \ref{EST} was assumed to hold without proof and the generalization given in theorems \ref{t:ptaset}, \ref{t:meant}, \ref{t:addition01}, \ref{t:addition03}, \ref{t:addition02} were not done explicitly. To prove these last statements, we use model-independent theorems discussed earlier and model-dependent inputs such as the energy barrier.

 Regarding the model-dependent results, \cite{cirillo2003metastability} focuses on the transition from the metastable states to the stable state. In particular, the authors describe the tube of typical trajectories and they also estimate the transition time. To do this, they analyze the geometrical conditions for the shrinking or the growing of a cluster. Furthermore, they characterize the local minima of the energy and the so-called \emph{traps} for the PCA dynamics. Building on this, we construct a specific path from any cluster to the stable state that the system follows with probability tending to one. Our estimates of the stability levels in Lemma \ref{EST} are based on these characterizations.

 The authors in \cite{cirillo2008metastability} consider a reversible PCA model with self-interactions. In particular they prove the recurrence to the set $\{\underline{-1},\underline{+1}\}$ and that $\underline{-1}$ is the unique metastable state. They estimate the transition time in probability, in $L^1$ and in law. Moreover, they characterize the critical droplet that is visited by the system with probability tending to one during its excursion from the metastable to the stable state. Furthermore, in \cite{nardi2012sharp} they prove sharp estimates for expected transition time by computing the prefactor explicitly.

 \paragraph{State of the art.} A first mathematical description of metastability \cite{penrose1971rigorous} was inspired by Gibbsian Equilibrium Statistical Mechanics and was based on the computation of the expected values with respect to restricted equilibrium states. The first dynamical approach, known as pathwise approach, was initiated in \cite{cassandro1984metastable} and developed in \cite{olivieri1995markov, olivieri1996markov, scoppola1994metastability}, see also \cite{olivieri2005large}. This approach derives large deviation estimates of the first hitting time and of the tube of typical trajectories. It is based on the notions of cycles and cycle paths and it hinges on a detailed knowledge of the energy landscape. Independently, similar results based on a graphical definition of cycles were derived in \cite{catoni1997exit, catoni1999simulated} and applied to reversible Metropolis dynamics and to simulated annealing in \cite{catoni1992parallel, trouve1996rough}. The pathwise approach was further developed in \cite{manzo2004essential, cirillo2013relaxation, cirillo2015metastability} to disentangle the study of transition time from the one of typical trajectories. This method was applied in \cite{arous1996metastability, cirillo1998metastability, cirillo1996metastability, hollander2000metastability, den2003droplet, gaudilliere2005nucleation, kotecky1994shapes, nardi1996low, neves1991critical, neves1992behavior, olivieri2005large} for Metropolis dynamics and in \cite{cirillo2003metastability, cirillo2008metastability, cirillo2008competitive} for parallel dynamics.

 The \emph{potential-theoretical approach} is based on the study of the hitting time through the use of the Dirichlet form and spectral properties of the transition matrix. One of the advantages of this method is that it provides an estimate of the expected value of the transition time including the prefactor, by exploiting a detailed knowledge of the critical configurations, see \cite{bovier2004metastability, bovier2016metastability}. This method was applied in \cite{bashiri2017note, boviermanzo2002metastability, cirillo2017sum, bovier2006sharp, den2012metastability} for Metropolis dynamics and in \cite{nardi2012sharp} for parallel dynamics.

 Recently other approaches are described in \cite{beltran2010tunneling, beltran2012tunneling, gaudillierelandim2014} and in \cite{bianchi2016metastable}.

 The more involved infinite volume limit, at low temperature or vanishing magnetic field, was studied for Metropolis dynamics via large deviation techniques in \cite{cerf2013nucleation, dehghanpour1997metropolis, manzo1998relaxation, manzo2001dynamical, schonmann1994slow, schonmann1998wulff} and via the potential-theoretical approach in \cite{bovier2010homogeneous, gaudilliere2009ideal, gaudilliere2010upper, hollander2000metastability, gaudilliere2020asymptotic}.

 \paragraph{Outline.} The paper is organized as follows, in Section \ref{mir} we define a general setup and we present the main model-independent results with some applications to concrete models. In Section \ref{mdr} we describe the reversible PCA model that we consider and we present the main model-dependent results. In Section \ref{pmir} we carry out the proof of the model-independent results, and in Section \ref{pmdr} we carry out the proof of the model-dependent results. Finally in Appendix \ref{appendixA} we recall some results and give explicit computation that are used in the paper, and in Appendix \ref{AppendixB} we prove theorems stated in Section \ref{s:series}.

\section{Model-independent results}\label{mir}

\subsection{General setup and definitions}\label{gs}
 Let $\mathcal{X}$ be a finite set, which we refer to as \emph{state space}, and let $\Delta:\mathcal{X} \times \mathcal{X} \longrightarrow \mathbb{R}^+ \cup \{ \infty \}$ be a function, which we call \emph{rate function}. $\Delta$ is said to be \emph{irreducible} if for every $x,y \in \mathcal{X}$ there exist a path $\omega=(\omega_1,...,\omega_n) \in \mathcal{X}^n$ with $\omega_1=x$, $\omega_n=y$ and $\Delta(\omega_i,\omega_{i+1}) < \infty$ for every $1 \leq i \leq n-1$, where $n$ is a positive integer. A family of time-homogeneous Markov chains $(X_n)_{n \in \mathbb{N}}$ on $\mathcal{X}$ with transition probabilities $\mathcal{P}_\beta$ indexed by a positive parameter $\beta$ is said to have \emph{rare transitions with rate function $\Delta$} when
\begin{equation}\label{Deltageneral1}
     \lim_{\beta\to\infty} -\frac{\log \mathcal{P}_\beta(x,y)}{\beta}=:\Delta(x,y),
\end{equation}
 for any $x,y \in \mathcal{X}$. Intuitively, $\Delta(x,y)= + \infty$ should be understood as the fact that, when $\beta$ is large, there is no possible transition between states $x$ and $y$. We also note that condition \eqref{Deltageneral1} is sometimes written more explicitly as \cite[Equation (2.2)]{cirillo2015metastability}: for any $\gamma>0$, there exists $\beta_0>0$ such that
\begin{equation}\label{Deltageneral}
 e^{- \beta [\Delta(x,y)+\gamma]} \leq \mathcal{P}_{\beta}(x,y) \leq e^{- \beta [\Delta(x,y)-\gamma]},
\end{equation}
 for any $\beta>\beta_0$ and any $x,y \in \mathcal{X}$, where the parameter $\gamma$ is a function of $\beta$ that vanishes for $\beta\to\infty$. Because of this, we also refer to the function $\Delta(x,y)$ as the \emph{energy cost} of the transition from $x$ to $y$.

 We assume that the Markov chain $(X_n)_n$ satisfies the detailed balance property
\begin{equation}\label{detailedbalance}
 \mathcal{P}_\beta(x,y)\,e^{-\beta G(x)}=\mathcal{P}_\beta(y,x)\,e^{-\beta G(y)},
\end{equation}
 for any $x,y \in \mathcal{X}$, where $G: \mathcal{X} \longrightarrow \mathbb{R}$ is the so-called \emph{Hamiltonian function}. Equivalently, the Markov chain is reversible with respect to the Gibbs measure
\begin{equation}\label{GBMgeneral}
     \mu(x):=\frac{e^{-\beta G(x)}}{\sum_{y\in \mathcal{X}}e^{-\beta G(y)}}.
\end{equation}
 This implies that the measure $\mu$ is stationary, that is $\sum_{x\in \mathcal{X}}\mu(x)\mathcal{P}_\beta(x,y)=\mu(y)$. Next, we define the \emph{virtual energy} as
\begin{equation}\label{Hamgeneral}
H(x):=\lim_{\beta\rightarrow\infty}G(x).
\end{equation}
 Definition \eqref{Hamgeneral} is well-posed, since for large $\beta$, the Markov chain $(X_n)_{n}$ is irreducible and its invariant probability distribution $\mu$ in \eqref{GBMgeneral} is such that for any $x \in \mathcal{X}$ the limit $\lim_{\beta \to\infty}-\frac{1}{\beta}\log\mu(x)$ exists and is a positive real number \cite[Prop. 2.1]{cirillo2015metastability}. Taking the limit $\beta\to\infty$ in \eqref{detailedbalance} yields
\begin{equation}\label{proprev}
    H(x)+\Delta(x,y)=H(y)+\Delta(y,x).
\end{equation}
This motivates the following definition of \emph{transition energy}
\begin{equation}\label{transitionenergy}
    H(x,y):=H(x)+\Delta (x, y),
\end{equation}
 where $x,y$ are configurations in $\mathcal{X}$. The definition of transition energy is needed to define the height along a path $\omega$ in the general setting. Indeed, there may not exist a configuration whose energy is equal to the energy of the maximum along the path. The transition energy between two configurations is defined as the sum between the virtual energy of the first configuration and the energy cost of the transition between the two configurations. This is unlike the Metropolis dynamics case \cite{nardi2016hitting}, where the transition energy between two configurations is the virtual energy of some state along the path between the two.

 Let $\omega=\{\omega_1,...,\omega_n\}$ be a finite sequence of configurations. We call $\omega$ a path with starting configuration $\omega_1$ and final configuration $\omega_n$. We denote the length of $\omega$ as $|\omega|=n$. We define the height along $\omega$ as $\Phi_\omega=H(\omega_1)$ if $|\omega|= 1$, or if $|\omega|>1$
\begin{equation}\label{Phimax}
    \Phi_\omega:=\max_{i=1,...,|\omega|-1} H(\omega_i,\omega_{i+1}).
\end{equation}
 Let $x,y\in \mathcal{X}$ be two configurations. The \emph{communication height} between two configurations $x$, $y$ is defined as
\begin{equation}\label{Phimin}
\Phi(x,y):=\min_{\omega\in\Theta(x,y)}\Phi_w,
\end{equation}
 where $\Theta(x,y)$ the set of all the paths $\omega$ starting from $x$ and ending in $y$. Similarly, we also define the communication height between two sets $A, B \subset \mathcal{X}$ as
\begin{equation}
\Phi(A,B):=\min_{x \in A,y \in B} \Phi(x,y).
\end{equation}
\begin{figure}[!hbt]
\begin{center}
\begin{tikzpicture} [line width=1pt]
\path [red] (0,0) edge (1,0.6);
\path [red] (1,0.6) edge (2,0);
\path [red] (2,0) edge (3,0.6);
\path [red] (3,0.6) edge (4,0);
\path (1,0.6) edge (2,1.2);
\path (2,1.2) edge (3,0.6);
\path (0,0) edge (1,-0.6);
\path (1,-0.6) edge (2,0);
\path (2,0) edge (3,-0.6);
\path (3,-0.6) edge (4,0);
\path (1,-0.6) edge (2,-1.2);
\path (2,-1.2) edge (3,-0.6);

\node at (0,0.2) {$x$};
\node at (4,0.2) {$y$};
\fill (0,0) circle(2pt);
\fill [red] (1,0.6) circle(2pt);
\fill [red] (2,0) circle(2pt);
\fill [red] (3,0.6) circle(2pt);
\fill (4,0) circle(2pt);
\fill (2,1.2) circle(2pt);
\fill (2,-1.2) circle(2pt);
\fill (1,-0.6) circle(2pt);
\fill (3,-0.6) circle(2pt);

 \end{tikzpicture}
   \end{center}
   \caption{Example of a path $\omega$ between $x$ and $y$ with $|\omega|=5$.}
    \end{figure}
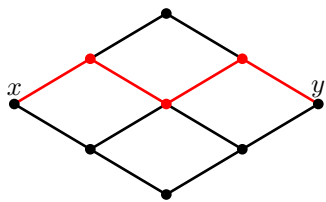
\begin{figure}[!hbt]
   \centering
    \includegraphics[scale=0.8]{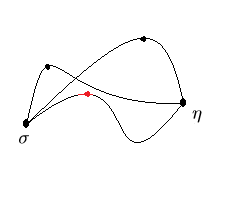}
    \label{fig:cammini}
       \caption{There are three paths in $\Theta(x,y)$. The red mark represents the communication height between $x$ and $y$. }
\end{figure}

 The \emph{first hitting time} of $A\subset \mathcal{X}$ starting from $x \in \mathcal{X}$ is defined as
\begin{equation}\label{fht}
    \tau^x_A:=\inf\{t>0 \,|\,X_t\in A\}.
\end{equation}
 Whenever possible we shall drop from the notation the superscript denoting the starting point. For any $x\in \mathcal{X}$, let $\mathcal{I}_x$ be the set of configurations with energy strictly lower than $H(x)$, i.e.,
\begin{equation}\label{I}
\mathcal{I}_x:=\{y\in \mathcal{X} \,|\, H(y)<H(x)\}.
\end{equation}
 The \emph{stability level} $V_{x}$ of $x$ is the  energy barrier that, starting from $x$, must be overcome to reach the set $\mathcal{I}_x$, i.e.,
\begin{equation}\label{stablev}
V_{x}:=\Phi(x,\mathcal{I}_x)-H(x).
\end{equation}
 If $\mathcal{I}_x$ is empty, then we let $V_x=\infty$. We denote by $\mathcal{X}^s$ the set of global minima of the energy, and we refer to these as ground states. The metastable states are those states that attain the maximal stability level $\activationenergym< \infty$, that is
\begin{align}\label{Gammam}
     & \activationenergym:=\max_{x\in \mathcal{X}\setminus \mathcal{X}^s}V_{x}, \\
    & \mathcal{X}^m:=\{y\in \mathcal{X}| \, V_{y}=\activationenergym \}.
\end{align}
 Since the metastable states are defined in terms of their stability level, a crucial role in our proofs is played by the set of all configurations with stability level strictly greater than $V$, that is
\begin{equation}\label{Xv}
\mathcal{X}_V:=\{x\in \mathcal{X} \,\, | \,\, V_{x}>V\}.
\end{equation}
 We frame the problem of metastability as the identification of metastable states and the computation of transition times from the metastable states to the stable configurations.
 In summary, from the mathematical point of view, the metastability phenomenon for a given system is described in terms of $\mathcal{X}^s$, $\activationenergym$ and $\mathcal{X}^m$. Now we define formally the \emph{energy barrier} $\Gamma$ as
\begin{equation}
\Gamma:=\Phi(y_m,y_s)-H(y_m),
\end{equation}
 where $y_m\in \mathcal{X}^m$ and $y_s\in \mathcal{X}^s$. Note that $\Gamma$ does not depend on the specific choice of $y_m, y_s$. The energy barrier is the minimum energy necessary to trigger the nucleation. The energy $\Gamma$ turns out to be equal to $\activationenergym$ under specific assumptions \cite[Theorem 2.4]{cirillo2013relaxation}.

 A different notion of metastable states is given in \cite{bovier2002metastability}, within the framework of the potential-theoretic approach. The \textit{Dirichlet form} associated with our reversible Markov chain is the functional
\begin{equation}\label{dirichlet}
\mathscr{D}_\beta[f]
:=
\frac{1}{2}\sum_{y,z\in \mathcal{X}}
 \mu_\beta(y)p_\beta(y,z) [f(y)-f(z)]^2,
\end{equation}
 where $f:\mathcal{X}\to \mathbb{R}$ is a function. Thus, given two not empty disjoint sets $Y,Z\subset \mathcal{X}$ the \emph{capacity} of the pair $Y$ and $Z$ defined as
\begin{equation}
\label{capac}
\rr{cap}_\beta(Y,Z)
:=
\min_{\newatop{f:\mathcal{X}\to[0,1]}{f\vert_Y=1,f\vert_Z=0}}
\mathscr{D}_\beta[f].
\end{equation}
 Note that the capacity is a \emph{symmetric} function of the sets $Y$ and $Z$. It can be proven that the right hand side of (\ref{capac}) has a unique minimizer called \emph{equilibrium potential} of the pair $Y$ and $Z$. There is a nice interpretation of the equilibrium potential in terms of hitting times. For any $x \in \mathcal{X}$, we denote by $\mathbb{P}_x(\cdot)$ and $\mathbb{E}_x[\cdot]$ respectively the probability and the
 average along the trajectories of the process started at $x$. Then, it can be proven that the \emph{equilibrium potential} of the pair $Y$ and $Z$ is equal to the function $h_{Y,Z}$ defined as follows
\begin{equation}
\label{eqpot}
h_{Y,Z}(x):=
\left\{
\begin{array}{ll}
 \mathbb{P}_x(\tau_Y<\tau_Z) & \;\;\textrm{ for } x\in \mathcal{X}\setminus(Y\cup Z)\\
1& \;\;\textrm{ for } x\in Y\\
0& \;\;\textrm{ for } x\in Z\\
\end{array}
\right.
\end{equation}
 where $\tau_Y$ and $\tau_Z$ are, respectively, the first hitting time to $Y$ and $Z$ for the chain started at $x$. It can be also proven that, for any $Y\subset \mathcal{X}$ and $z\in \mathcal{X}\setminus Y$,
\begin{equation}
\label{cap-prop}
\rr{cap}_\beta(z,Y)=\mu_\beta(z)\mathbb{P}_z(\tau_Y<\tau_z),
\end{equation}
see \cite[equation (7.1.16)]{bovier2016metastability}.
\begin{defn}
\label{pta}
 According to the potential-theoretic approach, a set $M\subset \mathcal{X}$ is said to be \emph{metastable} if
\begin{equation}
\label{metadef}
\lim_{\beta\to\infty}
\frac{\max_{x\notin{M}}\mu_\beta(x){[\rr{cap}_\beta(x,M)]}^{-1}}
      {\min_{x\in{M}}\mu_\beta(x){[\rr{cap}_\beta(x,M\setminus\{x\})]}^{-1}}
=0.
\end{equation}
\end{defn}
 In order to avoid confusion, we will denote the states that satisfy \eqref{metadef} as \emph{p.t.a.-metastable}. The physical meaning of the above definition can be understood once one remarks that the quantity $\mu_\beta(x)/\textrm{cap}_\beta(x,y)$, for any $x,y\in \mathcal{X}$, is strictly related to the communication cost between the states $x$ and $y$, see Proposition~\ref{t:apriori} for details. Thus, condition \eqref{metadef} ensures that the communication cost between any state outside $M$ and $M$ itself is smaller than the communication cost between any two states in $M$.

\subsection{Main model-independent results}
 The following theorems give estimates of the mixing time and the spectral gap in the general setting.
\begin{thm}\label{GM}
 Let $(P_\beta(x,y))_{x,y\in\mathcal{X}}$ be the transition matrix of a Markov chain. Assume there exists at least a stable state $s$ such that
\begin{equation}\label{P0}
\lim_{\beta \to \infty}- \frac{1}{\beta}\log \mathcal{P}_{\beta}(s,s)=0.
\end{equation}
Then, for any $0<\epsilon<1$ we have
\begin{equation}\label{lim2}
 \lim_{\beta \rightarrow \infty}{\frac{1}{\beta}\log{ t^{mix}_\beta(\epsilon)}}=\activationenergym,
\end{equation}
 where $t^{mix}_{\beta}:=\min\{n \geq 0 \, | \, \max_{x\in\mathcal{X}}||\mathcal{P}^n_\beta(x,\, \cdot \,)-\mu (\, \cdot \,)||_{TV}\leq \epsilon\}$ and $||\nu-\nu'||_{TV}=\frac{1}{2}\sum_{x\in\mathcal{X}}{|\nu(x)-\nu'(x)|}$ for every $\nu,\nu'$ probability distribution on $\mathcal{X}$.
\end{thm}
\begin{thm}\label{GMM}
 Let $(P_\beta(x,y))_{x,y\in\mathcal{X}}$ be a reversible transition matrix. Let $\rho_{\beta}=1-a^{(2)}_{\beta}$ be the spectral gap, with $a^{(2)}_\beta$ is the second eigenvalue of the transition matrix such $1=a^{(1)}_\beta>a^{(2)}_\beta\geq...\geq a^{(|\mathcal{X}|)}_\beta\geq -1$. Then there exist two constants $0<c_1<c_2<\infty$ independent of $\beta$ such that for every $\beta>0$,
\begin{equation}\label{rocompreso}
 c_1e^{-\beta(\activationenergym+\gamma_1)} \leq \rho_{\beta} \leq c_2e^{-\beta(\activationenergym-\gamma_2)},
\end{equation}
 where $\gamma_1,\gamma_2$ are functions of $\beta$ that vanish for $\beta\to\infty$.
\end{thm}
%

\subsection{Results for some concrete models}\label{examples}
 In this section we show that several well-known models in statistical mechanics satisfy the assumption \eqref{P0} of Theorem \ref{GM}. In particular we are able to get precise asymptotics for the mixing time of these models. 
 Throughout this section we denote by $\Lambda$ a finite subset of $\mathbb{Z}^2$, by $\mathcal{X}$ the configuration space and by $s$ a stable state.

 \paragraph{Metropolis algorithm.} The Hamiltonian function for this model coincides with the virtual energy and is given by
\begin{equation}
 H(\sigma):=-\frac{J}{2}\sum_{\substack{i,j \in \Lambda\\ |i-j|=1}} \sigma (i) \sigma (j) -\frac{h}{2} \sum_{i \in \Lambda} \sigma (i), \qquad \sigma \in \mathcal{X}.
\end{equation}
The transition probabilities are given by
\begin{equation}
 \mathcal{P}_\beta(\sigma,\eta):=q(\sigma,\eta) \exp\{-\beta[H(\eta)-H(\sigma)]\}, \qquad \sigma, \eta \in \mathcal{X},
\end{equation}
 where
\begin{equation*}
q(\sigma, \eta):=
\bigg \{
\begin{array}{rl}
\frac{1}{|\Lambda|} & \text{if } \exists i\in \Lambda: \sigma^i=\eta, \\
0 & \text{otherwise}. \\
\end{array}
\end{equation*}
 and
\begin{equation*}
\sigma^i(j):=
\bigg \{
\begin{array}{rl}
\sigma(j) & \text{if} \qquad j \neq i, \\
-\sigma(j) & \text{if} \qquad j = i. \\
\end{array}
\end{equation*}
 In this case the assumption \eqref{P0} is shown to hold in \cite[Prop. 3.24]{cirillo2015metastability}. Note that Kawasaki dynamics is a type of Metropolis dynamics, so it falls into this case.

 \paragraph{Reversible PCA model for Spin Systems.}\label{PCAmodel} For this model, the Hamiltonian function is given by
\begin{equation}\label{hamiltonianFunction}
 G(\sigma):=-h\sum_{i\in\Lambda}\sigma(i)-\frac{1}{\beta}\sum_{i\in\Lambda}\log\cosh [\beta(S_{\sigma}(i)+h)],
\end{equation}
and the virtual energy is obtained by \eqref{Hamgeneral}
\begin{equation}\label{Ham}
 H(\sigma)= -h\sum_{i\in\Lambda}\sigma(i)-\sum_{i\in\Lambda}|S_{\sigma}(i)+h|.
\end{equation}
 Here
\begin{equation}\label{ESSE}
S_{\sigma}(i):=\sum_{j\in U_i}K(i-j)\sigma(j),
\end{equation}
 where $K(i-j) \ne 0$ for $j \in U_i$ a neighborhood of $i$. Different choices of $K(\cdot)$ and $U_i$ yield different PCA. It can be shown that, if $U_i$ is symmetric, then the Markov chain is reversible.
The transition probabilities are given by
\begin{equation}\label{eqmark}
\begin{split}
     p(\sigma,\eta) :=\prod_{i\in\Lambda}p_{i,\sigma}(\eta(i)), \qquad \sigma, \eta \in \mathcal{X},
\end{split}
\end{equation}
 where, for $i\in\Lambda$ and $\sigma \in \mathcal{X}$, $p_{i,\sigma}(\cdot)$ is the probability measure on $\{-1,+1\}$ defined as
\begin{equation}\label{eqprob}
\begin{split}
 p_{i,\sigma}(a):=\frac{1}{1+\exp{\{-2\beta a(S_{\sigma}(i) +h)}\}}= \frac{1}{2}[1 +a \tanh\beta(S_{\sigma}(i) +h)],
\end{split}
\end{equation}
with $a \in \{-1,+1\}$. We have
\begin{align}
     \lim_{\beta \rightarrow \infty} - \frac{1}{\beta}\log p(s,s) &= \lim_{\beta \rightarrow \infty} - \frac{1}{\beta}\log \prod_{i \in \Lambda} \frac{1}{1+\exp{\{-2\beta s(i)(S_{s}(i) +h)}\}}  \notag\\
    & = \lim_{\beta \rightarrow \infty} \sum_{i \in \Lambda}\log((1+\exp{\{ -2\beta s(i)(S_{s}(i) +h)\}})^{\frac{1}{\beta}})  \notag\\
    & \le \lim_{\beta \rightarrow \infty}\sum_{i \in \Lambda}\log\Big( 1+\frac{1}{\beta}\exp{\{ -2\beta s(i)(S_{s}(i) +h)\}} \Big),
\end{align}
 where we used the inequality $(1+x)^{\alpha} \le 1+\alpha x$ with $\alpha \in (0,1)$. In this model the unique stable state is $s=\underline{+1}$, so we conclude in the following way
\begin{align}
 \lim_{\beta \rightarrow \infty}\sum_{i \in \Lambda}\log\Big( 1+\frac{1}{\beta}\exp{\{ -2\beta (S_{s}(i) +h)\}} \Big)
 & = \lim_{\beta \rightarrow \infty}\sum_{i \in \Lambda}\log\Big(\ 1+\frac{1}{\beta}\exp{\{ -2\beta (|U_i|+h)\}} \Big) \notag\\
 & = \lim_{\beta \rightarrow \infty} |\Lambda| \log\Big(\ 1+\frac{1}{\beta}\exp{\{ -2\beta (|U_i|+h)\}} \Big) \notag \\
&=0,
\end{align}
 where in the last equality we used that $h \ge 0$ and $|U_i|$ is the same for all $ i \in \Lambda$.

 \paragraph{Irreversible PCA model.} The Hamiltonian function of the Irreversible PCA model is given by
\begin{equation}
 G(\sigma,\tau):=-\sum\limits_{k\in \Lambda^2_{N}} [\sigma_k(\tau_{k^u}+\tau_{k^r})+h \sigma_k \tau_k],\qquad \sigma, \tau\in\mathcal X,
\end{equation}
 with $k^u:=(i,j+1)$, $k^r:=(i+1,j)$ for $k=(i,j)\in \Lambda^2_N$. The transition probabilities are given by
\begin{equation}
 \mathcal{P}_{\beta}(\sigma,\eta):=\frac{e^{-\beta G(\sigma,\eta)}}{\sum\limits_{\tau\in\mathcal{X}} e^{-\beta G(\sigma,\tau)}}.
\end{equation}
 Note that the subset $\mathcal{X} \setminus \mathcal{X}^s$ is not empty since $G$ is not constant. We compute
\begin{align}
      \lim_{\beta \rightarrow \infty} - \frac{1}{\beta}\log \mathcal{P}_{\beta}(s,s) & = \lim_{\beta \rightarrow \infty} - \frac{1}{\beta}\log \Big(\frac{e^{-\beta G(s,s)}}{\sum\limits_{\tau\in\mathcal{X}} e^{-\beta G(s,\tau)}} \Big) \notag \\
     & = H(s,s)+\lim_{\beta \rightarrow \infty} \frac{1}{\beta}\log\Big(\sum\limits_{\tau\in\mathcal{X}} e^{-\beta G(s,\tau)}\Big).
\end{align}
 Take $\overline{\tau}\in\mathcal{X}$ such that \( \displaystyle G(s,\overline{\tau})=\min_\tau G(s,\tau) \). We get
\begin{align}
  H(s,s)+\lim_{\beta \rightarrow \infty} \frac{1}{\beta}\log\Big(\sum\limits_{\tau\in\mathcal{X}} e^{-\beta G(s,\tau)}\Big) & \leq H(s,s)+\lim_{\beta \rightarrow \infty} \frac{1}{\beta}\log \Big(2^{N^2} e^{- \beta G(s,\overline{\tau})}\Big) \notag \\
 & =H(s,s)-H(s,\overline{\tau})+\lim_{\beta \rightarrow \infty} \frac{1}{\beta}\log(2^{N^2}).
\end{align}
 The last term goes to zero since $N$ is finite. Since in this model $s=\underline{+1}$, we have
\begin{align*}
     H(\underline{+1},\underline{+1})=-N^4(2+h), \qquad H(\underline{+1},\overline{\tau})=-N^4(2+h)
\end{align*}
and the conclusion follows.
\subsection{Series of metastable states}
\label{s:series}
\par\noindent
 The structure of the energy landscape that we analyze for our reversible PCA model in Section \ref{model} is such that the system has three metastable states with one non-degenerate-in-energy metastable state and two degenerate metastable states. Moreover, the system started at the metastable state with higher energy, must necessarily visit the second one before relaxing to the stable state. In this Section we generalize the results in \cite[Section 2.5, 2.6]{cirillo2017sum} to this degenerate context. In particular, we shall prove the addition rule for the exit times from the metastable states.

\begin{condition}
\label{t:series00}
 We assume that the energy landscape $(\mathcal{X},Q,H,\Delta)$ is such that there exist four or more states $\stab$, $\metauno^1, \metauno^2,..., \metauno^n$ and $\metadue$ such that $\mathcal{X}^s=\{\stab\}$, $\mathcal{X}^m=\{\metauno^1,...,\metauno^n,\metadue\}$, and $H(\metadue)>H(\metauno^r)$, $H(\metauno^r)=H(\metauno^q)$, $\Phi(\metauno^r,\metauno^q)-H(\metauno^r)<\activationenergym$ for every $r,q=1,...,n$, with $n \in \mathbb{N}$.
\end{condition}

 Recalling the definition of the set of ground states $\mathcal{X}^s$, we immediately have
\begin{equation}
\label{msl0-1}
H(\metauno^r)>H(x_0) \qquad \text{for every } r=1,...,n.
\end{equation}
 Moreover, from the definition \eqref{stablev} of maximal stability level it follows that (see \cite[Theorem~2.3]{cirillo2013relaxation}) the communication cost from $\metadue$ to $\stab$ is equal to the communication cost from $\metauno^r$ to $\stab$ for every $r=1,...,n$, that is
\begin{equation}
\label{msl00}
\Phi(x_2,x_0)-H(x_2)=\Phi(\metauno^r,\stab)-H(\metauno^r)=
\Gamma_m.
\end{equation}
 Note that, since $x_2$ is a metastable state, its stability level cannot be lower than $\Gamma_m$. Then, recalling that $H(x_2)>H(\metauno^r)$ for every $r=1,...,n$, one has that
 $\Phi(x_2,\metauno^r)-H(x_2)\ge\Gamma_m$. On the other hand, \eqref{msl00} implies that there exists a path $\omega\in\Theta(x_2,x_1^r)$ such that $\Phi_\omega=H(x_2)+\Gamma_m$ and, hence,
 $\Phi(x_2,x_1^r)-H(x_2)\le\Gamma_m$ for every $r=1,...,n$. The two bounds finally imply that
\begin{equation}
\label{series01new}
\Phi(x_2,\metauno^r)-H(x_2)=\Gamma_m.
\end{equation}
 Note that the communication cost from $\stab$ to $\metadue$ and that from $\metauno^r$ to $\metadue$ are larger than $\Gamma_m$, i.e.,
\begin{equation}
\label{indietro}
\Phi(\stab,\metadue)-H(\stab)
\Gamma_m
\;\;\;\textrm{ and }\;\;\;
\Phi(\metauno^r,\metadue)-H(\metauno^r)
\Gamma_m, \qquad \text{ for every } r=1,...,n.
\end{equation}
Indeed, recalling the reversibility property \eqref{proprev}, we have
\begin{eqnarray*}
\Phi(\metauno^r,\metadue)-H(\metauno^r)
&=&
\Phi(\metadue,\metauno^r)
-H(\metadue)
+H(\metadue)-H(\metauno^r)\\
&=&
\Gamma_m
+H(\metadue)-H(\metauno^r)
\Gamma_m.
\end{eqnarray*}
where in the last two steps we have used
 \eqref{series01new} and Condition~\ref{t:series00}, which proves the second of the two equations \eqref{indietro}. The first of them can be proved similarly.
 When the system is started at $x_2$, with high probability it will visit $x_1^r$ before $x_0$ for every $r=1,...,n$. For this reason we shall assume the following condition.
\begin{condition}
\label{t:series01}
Condition~\ref{t:series00} is satisfied and
\begin{equation}
\label{series01}
 \lim_{\beta\to\infty}\mathbb{P}_{\metadue}(\tau_{\stab}<\tau_{\metauno^r})=0, \qquad \text{for every } r=1,...,n.
\end{equation}
\end{condition}
 We remark that the Condition~\ref{t:series01} is in fact a condition on the equilibrium potential $h_{x_0,\metauno^r}$ evaluated at $x_2$, for every $r=1,...,n$.

 One of important goals of this paper is to prove an additional rule for the mean hitting time of $\underline{+1}$ starting at $\underline{-1}$ using Theorem \ref{t:addition02} for the expectation of the transition time $\tau_{\stab}$ for the chain started at $\metadue$. Such an expectation, hence, will be of order $\exp(\beta\Gamma_m)$ and the prefactor will be that given in \eqref{addition02}.

 We can thus formulate the further assumptions that we shall need in the sequel.

\begin{condition}
\label{t:series02}
 Condition~\ref{t:series00} is satisfied and there exists two positive constants $k_1,k_2<\infty$ and such that
\begin{align}
\label{series02}
 & \frac{\mu_\beta(\metadue)}{\rr{cap}_\beta(\metadue,\{\metauno^1,...,\metauno^n,\stab\})}
=
\frac{1}{k_1}
e^{\beta\Gamma_\rr{m}}[1+o(1)],,\,\,\,\,\,
 & \frac{\mu_\beta(\{\metauno^1,...,\metauno^n\})}{\rr{cap}_\beta(\{\metauno^1,...,\metauno^n\},\stab)}
=
\frac{1}{k_2}
e^{\beta\Gamma_\rr{m}}[1+o(1)],
\end{align}
where $o(1)$ denotes a function tending to zero in the limit
$\beta\to\infty$.
\end{condition}

\begin{condition}
\label{t:series03}
Condition~\ref{t:series00} is satisfied and
there exists $n$ positive constants $c_1, c_2,..., c_n<\infty$ such that
\begin{equation}
\label{series03}
\frac{\mu_\beta(x_1^r)}{\rr{cap}_\beta(\metauno^r,\stab)}
=
\frac{1}{c_i}
e^{\beta\Gamma_\rr{m}}[1+o(1)], \qquad \text{for every } r=1,...,n,
\end{equation}
where $o(1)$ denotes a function tending to zero in the limit
$\beta\to\infty$.
\end{condition}

 The following theorems generalize respectively Theorem 1, Theorem 2, Theorem 3, Theorem 4 in \cite{cirillo2017sum}. We prove them in Appendix B.

\begin{thm}\label{t:ptaset}
 Assume Condition~\ref{t:series00} is satisfied. Then for every $r=1,...,n$ we have $\{x_0,\metauno^r,x_2\} \subset \mathcal{X}$ is a p.t.a.-metastable set.
\end{thm}

\begin{thm}\label{t:meant}
Assume Condition~\ref{t:series00} is satisfied. Then
\begin{align}
&\mathbb{E}_{\metadue}[\tau_{\{\metauno^1,...,\metauno^n,\stab\}}]\!=\!
\frac{\mu_\beta(\metadue)}
      {\rr{cap}_\beta(\metadue,\{\metauno^1,...,\metauno^n,\stab\})}[1+o(1)],
\label{Egen1}\\
&\mathbb{E}_{\{\metauno^1,...,\metauno^n\}}[\tau_{\stab}]\!=\!
\frac{\mu_\beta(\{\metauno^1,...,\metauno^n\})}
      {\rr{cap}_\beta(\{\metauno^1,...,\metauno^n\},\stab)}[1+o(1)],
\label{Egen2} \\
&\mathbb{E}_{\metauno^r}[\tau_{\stab}]\!=\!
 \frac{n\mu_\beta(\metauno^r)}{\rr{cap}_\beta(\metauno^r,\stab)}[1+o(1)], \qquad \text{for every } r=1,...,n.
\label{valattsing}
\end{align}
\end{thm}

\begin{thm} \label{t:addition01}
 Assume Condition~\ref{t:series00} and Condition~\ref{t:series02} are satisfied. Then
\begin{align}\label{addition01}
& \mathbb{E}_{\metadue}[\tau_{\{\metauno^1,...,\metauno^n,\stab\}}]
=
e^{\beta\Gamma_\rr{m}}\frac{1}{k_1}[1+o(1)], \\
& \mathbb{E}_{\{\metauno^1,...,\metauno^n\}}[\tau_{\stab}]
=
e^{\beta\Gamma_m}\frac{1}{k_2}[1+o(1)],
\end{align}
\end{thm}

\begin{thm} \label{t:addition03}
 Assume Condition~\ref{t:series00} and Condition~\ref{t:series03} are satisfied. Then
\begin{align}\label{addition01}
& \mathbb{E}_{\metauno^r}[\tau_{\stab}]
=
 e^{\beta\Gamma_m} \frac{n}{c_i}[1+o(1)], \qquad \text{for every } i=1,...,n.
\end{align}
\end{thm}

\begin{thm}\label{t:addition02}
 Assume Condition~\ref{t:series00}, Condition~\ref{t:series01}, and Condition~\ref{t:series02} are satisfied. Then
\begin{equation}
\label{addition02}
\mathbb{E}_{\metadue}[\tau_{\stab}]
=
e^{\beta\Gamma_m}\Big(\frac{1}{k_1}+\frac{1}{k_2}\Big)[1+o(1)]
\end{equation}
\end{thm}

 We remark that Theorem~\ref{t:addition02} gives an addition formula for the mean hitting time of $\stab$ starting at $\metadue$. Neglecting terms of order $o(1)$, such a mean time can be written as the sum of the mean hitting time of the subset $\{\metauno^1,...,\metauno^n,\stab\}$ starting at $\metadue$ and of the mean hitting time of $\stab$ starting from any state in $\{\metauno^1,...,\metauno^n\}$. It is very interesting to note that in this decomposition no role is played by the mean hitting time of $\{\metauno^1,...,\metauno^n\}$ starting at $\metadue$. 

\section{Model-dependent results}\label{mdr}

\subsection{The model}\label{model}

 We consider the \textsl{reversible PCA model for Spin Systems} introduced by Derrida in \cite{derrida1989dynamical}, see also \cite{cirillo2003metastability}. In the second example of Section \ref{examples}, we considered a general PCA, but from now on we restrict ourselves to a specific nearest-neighbor interaction, see figure \ref{neighbors}. Consider the two–dimensional torus with $L$ even $\Lambda^2_{L}:=\{0,...,L-1\}^{2}$, endowed with the Euclidean metric. To each site $i\in\Lambda$ we associate a variable $\sigma(i)\in \{-1,+1\}$. $\Lambda^2_{L}$ represents an interacting particles system characterized by their spin and we interpret $\sigma(i)=+1$ (respectively $\sigma(i)=-1$) as indicating that the spin at site $i$ is pointing upwards (respectively downwards). Let $\mathcal{X}:=\{-1,+1\}^{\Lambda}$ be the \emph{configuration space}, let $\beta:=\frac{1}{T} >0$  where $T$ is thought of as the temperature. Let $h\in (0,1)$ be a parameter representing the \emph{external ferromagnetic field}. We do not consider the case $h>1$, because in that case there is no metastable behavior. The dynamics of the system are modelled as a Markov chain $(\sigma_n)_{n \in \mathbb{N}}$ on $\mathcal{X}$ with transition matrix defined in \eqref{ESSE}, \eqref{eqmark}. In the rest of the paper, we will choose
\begin{equation}\label{KAPPA}
K(i-j):=
\bigg \{
\begin{array}{rl}
1 & \text{if $|i-j|= 1$}, \\
0 & \text{otherwise}.\\
\end{array}
\end{equation}
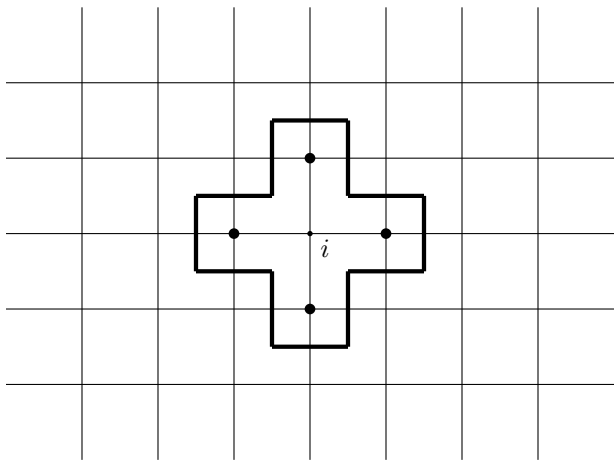
\begin{figure}\label{neighbors}
\begin{center}
    \begin{tikzpicture}
    \fill (0,0) circle(1pt);
    \node at (0.2,-0.2) {$i$};
    \fill (1,0) circle(2pt);
    \fill (-1,0) circle(2pt);
    \fill (0,1) circle(2pt);
    \fill (0,-1) circle(2pt);

    \path (-4,0) edge (4,0);
    \path (-4,1) edge (4,1);
    \path (-4,2) edge (4,2);
    \path (-4,-1) edge (4,-1);
    \path (-4,-2) edge (4,-2);

    \path (-3,-3) edge (-3,3);
    \path (-2,-3) edge (-2,3);
    \path (-1,-3) edge (-1,3);
    \path (0,-3) edge (0,3);
    \path (1,-3) edge (1,3);
    \path (2,-3) edge (2,3);
    \path (3,-3) edge (3,3);

    \path [ultra thick] (-1.5,-0.5) edge (-1.5,0.5);
    \path [ultra thick] (1.5,-0.5) edge (1.5,0.5);
    \path [ultra thick] (-0.5,0.5) edge (-0.5,1.5);
    \path [ultra thick] (0.5,0.5) edge (0.5,1.5);
    \path [ultra thick] (-0.5,-0.5) edge (-0.5,-1.5);
    \path [ultra thick] (0.5,-0.5) edge (0.5,-1.5);
    \path [ultra thick] (-1.5,-0.5) edge (-0.5,-0.5);
    \path [ultra thick] (0.5,0.5) edge (1.5,0.5);
    \path [ultra thick] (-1.5,0.5) edge (-0.5,0.5);
    \path [ultra thick] (0.5,-0.5) edge (1.5,-0.5);
    \path [ultra thick] (-0.5,1.5) edge (0.5,1.5);
    \path [ultra thick] (-0.5,-1.5) edge (0.5,-1.5);

    \end{tikzpicture}
\end{center}
   \caption{In black are highlighted the sites $j$ such that $K(i-j)\neq 0$ in the reversible PCA model for spin systems.}
    \end{figure}
 Note that the transition probability $p_{i,\sigma}(s)$ for the spin $\sigma(i)$ given in \eqref{eqprob} depends only on the values of the adjacent spins.

 The system evolves in discrete time steps, where at each step, all the spins are updated simultaneously according to the probability distribution \eqref{eqprob}. Intuitively, the value of the spin is likely to align with the local effective field $S_\sigma(i)+h$. Here $S_\sigma(i)$ represents a ferromagnetic interaction among spins.

 The Markov chain $\sigma_n$ satisfies the detailed balance property \eqref{detailedbalance}, where $G(\cdot)$ in \eqref{hamiltonianFunction} is the \emph{Hamiltonian function}. Equivalently, the Markov chain is reversible with respect to the Gibbs measure \eqref{GBMgeneral} and this implies that the measure $\mu$ is stationary. Finally, given $\sigma,\eta$ $\in \mathcal{X}$, we define the energy cost of the transition from $\sigma$ to $\eta$ for our specific PCA, as
\begin{equation}\label{Delta}
 \Delta(\sigma,\eta):=-\lim_{\beta \rightarrow \infty } \frac{\log p(\sigma,\eta)}{\beta}=\sum_{\substack{i\in \Lambda : \\ \eta (i)|S_{\sigma}(i)+h|<0}} 2|S_{\sigma}(i)+h|.
\end{equation}
 Note that $\Delta(\sigma,\eta)\geq 0$ and, perhaps surprisingly, $\Delta(\sigma,\eta)$ is not necessarily equal to $\Delta(\eta,\sigma)$. We also note that condition \eqref{Delta} is sometimes written more explicitly as in \eqref{Deltageneral}. The last equality in \eqref{Delta} is obtained as follows (for more details, see Appendix \ref{appendixA}),
\begin{align*}
      -\lim_{\beta \rightarrow \infty } \frac{\log p(\sigma,\eta)}{\beta} & = \sum_{i \in \Lambda:\eta(i)(S_{\sigma}(i)+h)<0} \lim_{\beta \rightarrow \infty } \frac{\log ({1+\exp\{2\beta |S_{\sigma}(i)+h|\}})}{\beta} \\
     & =\sum_{i \in \Lambda:\eta(i)(S_{\sigma}(i)+h)<0} 2|S_{\sigma}(i)+h|.
\end{align*}
Let us fix the notation of some important states as follows:
\begin{itemize}
     \item $\underline{+1}$ is the configuration such that $\underline{+1}(i)=+1$ for every $i\in\Lambda$;
     \item $\underline{-1}$ is the configuration such that $\underline{-1}(i)=-1$ for every $i\in\Lambda$;
     \item $\underline{c}^e$ and $\underline{c}^o$ are the configurations such that $\underline{c}^e(i)=(-1)^{i_1+i_2}$ and $\underline{c}^o(i)=(-1)^{i_1+i_2+1}$  for every $i=(i_1,i_2)\in\Lambda$. These configuration are called \emph{chessboard configurations}.
\end{itemize}
 Next we define the \emph{virtual energy} as the limit 
\begin{equation}\label{Ham}
 \lim_{\beta\to\infty}G(\sigma):=H(\sigma)= -h\sum_{i\in\Lambda}\sigma(i)-\sum_{i\in\Lambda}|S_{\sigma}(i)+h|,
\end{equation}
We distinguish two cases.
\begin{itemize}
     \item Case $h=0$. In this case $H(\sigma)=-\sum_{i \in \Lambda}|S_{\sigma}(i)|$, so there exist four minima of $H$ given by the configurations  $\underline{+1}, \underline{-1}$ and the chessboard configurations. The configurations \underline{+1}, $\underline{-1}$ and $\underline{c}$ are ground states and each site of them contributes $-4$ to the total energy.
     \item Case $h >0$. In this case \underline{+1} is the unique ground state. The energy of this state is $(-h-(4+h)) |\Lambda|$, so each site contributes $-h-(4+h)$ to the total energy.
\end{itemize}
 From now on we assume $h>0$, fixed and small. Under periodic boundary conditions, the energy of these configurations is, respectively
\begin{itemize}
    \item $H(\underline{+1}) =-L^2(4 + 2h)$,
    \item $H(\underline{-1}) =-L^2(4-2h)$,
    \item $H(\underline{c}^e)=H(\underline{c}^0) =-4L^2$.
\end{itemize}
 Since $H(\underline{c}^e)=H(\underline{c}^o)$ and $\Delta(\underline{c}^e,\underline{c}^o)=\Delta(\underline{c}^o,\underline{c}^e)=0$, from now on we will indicate either element of the set $\{\underline{c}^e, \, \underline{c}^o\}$ as $\underline{c}$, this is an example of stable pair (see Definition \ref{stablepair}). Therefore, $H(\underline{-1})> H(\underline{c})> H(\underline{+1})$ for $0< h < 1$. Our first goal is to show that $\{\underline{-1},\underline{c}\}$ is the set of metastable states and $\underline{+1}$ is the global minimum (or ground state).

\subsection{Main model-dependent results}
 In the setup introduced in \cite{manzo2004essential}, the minimal description of the metastability phenomenon is given in terms of $\mathcal{X}^s$, $\mathcal{X}^m$ and $\activationenergym$, so we concentrate our attention on these. In particular we determine the metastable and stable stases and we show that the maximal stability level $\activationenergym$ is equal to the energy barrier $\activationenergy$, defined as \cite[(3.29)]{cirillo2003metastability}
\begin{equation}
\Gamma\equiv \activationenergy=-2h\lambda^2+2\lambda(4+h)-2h,
\end{equation}
 where $\lambda$ is the \emph{critical length} computed in \cite[(3.24)]{cirillo2003metastability} and defined as
\begin{equation}\label{critical length}
\lambda:=\Big[\frac{2}{h}\Big]+1,
\end{equation}
 where $[\cdot]$ is the integer part. Assuming that the system is prepared in the state $\sigma_0=\underline{-1}$, with probability tending to one as $\beta \to \infty$ the system visits the chessboard $\underline{c}$ before relaxing to the stable state $\underline{+1}$. Moreover, by \cite[Theorem 3.11, Theorem 3.13]{cirillo2003metastability} along the tube of paths from $\underline{-1}$ to $\underline{c}$ the system visits a certain set of configurations called \emph{critical droplets from $\underline{-1}$ to $\underline{c}$}. The critical droplets are all those configurations that have a single chessboard droplet of a specific size in a sea of minuses. Instead, along the tube of paths from $\underline{c}$ to $\underline{+1}$ the system visits a certain set of configurations, also called \emph{critical droplets from $\underline{c}$ to $\underline{+1}$}, but in this case these are all those configurations that have a single plus droplet of a specific size in a chessboard. The droplet size, in both cases, is the so-called critical length $\lambda$. We then say that a rectangle is \emph{supercritical} (resp.~\emph{subcritical}) if the side of the rectangle is greater than $\lambda$ (resp.~smaller than $\lambda$). Formally, the chessboard droplet is a supercritical rectangle with a one-by-one protuberance attached to one of the two longest sides and with the spin plus in this protuberance. Note that starting from different initial configurations yields different kinds of droplets.

 We are finally ready to present our model-dependent results. In Lemma \ref{EST} we show that all states different from ${\underline{+1},\underline{-1},\underline{c}}$ have a strictly lower stability level than $\activationenergy$. Using this lemma and \cite[Lemma 3.4, Lemma 4.1]{cirillo2003metastability}, we show that $\activationenergy=\activationenergym$, allowing us to conclude in Theorem \ref{IdMS} that the only metastable states are indeed $\underline{-1}$ and $\underline{c}$.
 \begin{lem}[Estimate of stability levels]\label{EST} For every $\eta\in\mathcal{X}\setminus\{\underline{-1},\underline{c},\underline{+1}\},$ there exists $V^*$ such that $V_\eta\leq V^*<\activationenergy$.
\end{lem}
 \begin{thm}[Identification of metastable states]\label{IdMS} For the reversible PCA model \eqref{KAPPA} we have $\activationenergym=\activationenergy$ and thus $\mathcal{X}^m=\{\underline{-1},\underline{c}\}$.
\end{thm}
 Theorem \ref{RP} below implies that the system visits a metastable state or a ground state in a time shorter than $e^{\beta V^*+\epsilon}$ and visits a stable state in a time shorter than $e^{\beta \activationenergym+\epsilon}$, uniformly in the starting state for any $\epsilon >0$. We say that a function $\beta \mapsto f(\beta)$ is super exponentially small (SES) if $$\lim_{\beta\to\infty}\log{f(\beta)}=-\infty.$$
 \begin{thm}[Recurrence property]\label{RP} For any $\epsilon >0$, the functions
\begin{equation}\label{eqRP}
 \beta \mapsto \sup_{\eta\in \mathcal{X}}\mathbb{P}_{\eta}(\tau_{\{\underline{+1}, \underline{c}, \underline{-1}\}}>e^{\beta (V^{*}+\epsilon)}), \qquad \beta \mapsto \sup_{\eta\in \mathcal{X}}\mathbb{P}_{\eta}(\tau_{\underline{+1}}>e^{\beta (\activationenergy +\epsilon)})
\end{equation}
are SES.
\end{thm}
%
 Equation \eqref{valoreatteso} in the next theorem already appeared in \cite[Theorem 3.1]{cirillo2016sum}, however the proof there was incomplete. Thanks to the previous theorems we are able to prove it rigorously here. The second part of the next theorem is an application of Theorem \ref{GM} to the reversible PCA model by Derrida.
\begin{thm}\label{ThE} For $\beta$ large enough, we have
\begin{equation}\label{valoreatteso}
     \mathbb{E}_{\underline{-1}}[\tau_{\underline{+1}}]=\bigg(\frac{1}{k_1}+\frac{1}{k_2}\bigg)e^{\beta\activationenergy}(1+o(1)),
\end{equation}
 where $k_1=k_2=8\lambda|\Lambda|$. Moreover for any $0<\epsilon<1$ we have
\begin{equation}\label{lim2PCA}
 \lim_{\beta \rightarrow \infty}{\frac{1}{\beta}\log{ t^{mix}_\beta(\epsilon)}}=\activationenergy,
\end{equation}
 and there exist two constants $0<c_1<c_2<\infty$ independent of $\beta$ such that for every $\beta>0$
\begin{equation}\label{rocompresopca}
 c_1e^{-\beta(\activationenergy+\gamma_1)} \leq \rho_{\beta} \leq c_2e^{-\beta(\activationenergy-\gamma_2)},
\end{equation}
 where $\gamma_1,\gamma_2$ are functions of $\beta$ that vanish for $\beta\to\infty$, and $\rho_{\beta}$ is the spectral gap.
\end{thm}
 The first term $\frac{1}{k_1}e^{\beta \activationenergy}$ represents the contribution of the mean hitting time $\mathbb{E}_{\underline{-1}}[\tau_{\underline{c}}\textbf{1}_{\{\tau_{\underline{c}}< \tau_{\underline{+1}}\}}]$ while the second term $\frac{1}{k_2}e^{\beta \activationenergy}$ represents the contribution of $\mathbb{E}_{\underline{c}}[\tau_{\underline{+1}}]$.

\section{Proof of model-independent results}\label{pmir}

 Before we prove Theorem \ref{GM}, let us recall some important definitions.
 \begin{defn} [{Cycle, \cite[Def. 2.3]{cirillo2015metastability}, \cite[Def. 4.2]{catoni1999simulated}}] Let $(X_n)_n$ be a Markov chain. A nonempty set $C \subset \mathcal{X}$ is a \emph{cycle} if it is either a singleton or for any $x,y \in C$, such that $x\neq y$,
\begin{equation}
\lim_{\beta \to \infty}-\frac{1}{\beta} \log \mathcal{P}(
X_{\tau_
{(\mathcal{X} \setminus C)\cup\{y\}}}
\neq y \,\, | \,\, X_0=x)>0.
\end{equation}
\end{defn}
 In other words, a nonempty set $C \subset \mathcal{X}$ is a
\emph{cycle} if it is either a singleton or if for any $x \in C$,
 the probability for the process starting from $x$ to leave $C$ without first visiting all the other elements of $C$ is exponentially small. We denote by $\mathcal{C}(\mathcal{X})$ the set of cycles of $\mathcal{X}$.
 \begin{defn}[{Energy Cycle, \cite[(2.17)]{cirillo2015metastability}, \cite[Def. 3.5]{cirillo2015metastability}}]
 A nonempty set $A \subset \mathcal{X}$ is an \emph{energy-cycle} if and only if it is either a singleton or it verifies the relation
\begin{equation}
    \max_{x,y \in A} \Phi(x,y)< \Phi(A, \mathcal{X} \setminus A).
\end{equation}
\end{defn}
\begin{defn}
 Given a cycle $C \subset \mathcal{X}$, we denote by $\mathcal{F}(C)$ the set of the minima of the energy in $C$, namely
\begin{equation}
\mathcal{F}(C):= \{ x \in C \, | \, \min_{y \in C} H(y) = H(x) \}.
\end{equation}
\end{defn}
 The proposition \cite[Prop. 3.10]{cirillo2015metastability} establishes the equivalence between cycle and energy-cycle and allows us to use the equivalence between the approach in \cite{holley1988simulated, catoni1992parallel, catoni1997exit} and the path-wise approaches \cite{cirillo2003metastability, nardi2016hitting, cirillo2015metastability, manzo2004essential, olivieri1995markov, olivieri1996markov, olivieri2005large} that uses the energy-cycle. Next we define the collection of maximal cycles.
 \begin{defn} [{\cite[Def. 20]{nardi2016hitting}, \cite[Def. 2.4]{cirillo2015metastability}}]
 Given a nonempty subset $A \subset \mathcal{X}$, we denote by $\mathcal{M}(A)$ the \emph{collection of
maximal cycles} that partitions $A$, that is
\begin{equation}
   \mathcal{M}(A):=\{C \in \mathcal{C}(\mathcal{X}) \,\, | \,\, C \,\, \text{maximal by inclusion under the constraint} \,\, C \subseteq A\}.
\end{equation}
\end{defn}
 Moreover, we extend to the general setting the definition of the maximal depth given in \cite[Def. 21]{nardi2016hitting} for the setting of Metropolis dynamics.
\begin{defn}
 The \emph{maximal depth} $\Tilde{\Gamma}(A)$ of a nonempty subset $A \subset \mathcal{X}$ is the maximal depth of a cycle contained in $A$, that is
\begin{equation}
    \Tilde{\Gamma}(A):=\max_{C \in \mathcal{M}(A)} \Gamma(C).
\end{equation}
 Trivially $\Tilde{\Gamma}(C):=\Gamma(C)$ if $C \in \mathcal{C}(\mathcal{X})$.
\end{defn}
 \begin{proof}[Proof of Theorem \ref{GM}] We prove \eqref{lim2} by generalizing \cite[Prop. 3.24]{nardi2016hitting}. To do this, we show that $\Tilde{\Gamma}(\mathcal{X} \setminus \{s\})$ is equal to $\activationenergym$. Recall definition \eqref{Gammam}
\begin{align*}
 \activationenergym:=\max_{x\in\mathcal{X}\setminus \{s\}}(\Phi(x,\mathcal{I}_x)-H(x)).
\end{align*}
 Since $\Phi(x,\mathcal{I}_x)\le \Phi(x,s)$, we have that $\activationenergym \le \Tilde{\Gamma}(\mathcal{X}\setminus \{s\})$. To prove the reverse inequality $\activationenergym \ge \Tilde{\Gamma}(\mathcal{X}\setminus \{s\})$, we consider $R_D(x)$, the union of $\{x \}$ and of the points in $\mathcal{X}$ which can be reached by means of paths starting from $x$ with height smaller than the height that is necessary to escape from $D \subset \mathcal{X}$ starting from $x$ \cite[(3.58)]{cirillo2015metastability}. We consider
\begin{equation}
     R_{\mathcal{X} \setminus \{s\}}(x)=\{x\} \cup \{y \in \mathcal{X} \,\, | \,\, \Phi(x,y)< \Phi(x,s)\}.
\end{equation}
 We partition $\mathcal{X}$ into the set of local minima $\mathcal{X}_0$ (i.e., $\mathcal{X}_V$ with $V=0$) and its complement, as $\mathcal{X}=\mathcal{X}_0 \cup (\mathcal{X} \setminus \mathcal{X}_0)$, so that $\mathcal{X} \setminus \{s\}=( \mathcal{X}_0 \cup (\mathcal{X} \setminus \mathcal{X}_0)) \setminus \{s\}=(\mathcal{X}_0 \setminus \{s\}) \cup (\mathcal{X} \setminus \mathcal{X}_0)$. Then,
\begin{equation}
 \Tilde{\Gamma}(\mathcal{X}\setminus \{s\})=\max_{x \in \mathcal{X} \setminus \{s\}}\Gamma(R_{\mathcal{X}\setminus \{s\}}(x))= \max
 \bigg \{ \max_{x \in \mathcal{X} \setminus \mathcal{X}_0}\Gamma(R_{\mathcal{X}\setminus \{s\}}(x)),  \max_{x \in \mathcal{X}_0 \setminus \{s\}}\Gamma(R_{\mathcal{X}\setminus \{s\}}(x)) \bigg \}.
\end{equation}
Let us analyze the two terms on the right separately.
\begin{itemize}
     \item If $x\in \mathcal{X}_0 \setminus \{s\}$, then $R_{\mathcal{X}\setminus \{s\}}(x)=\{y \in \mathcal{X} \,\, | \,\, \Phi(x,y)< \Phi(x,s)\}$ is a non-trivial cycle. Using \cite[Prop. 3.17]{cirillo2015metastability},
    \begin{itemize}
         \item[i)] If $x\in \mathcal{F}(R_{\mathcal{X}\setminus \{s\}}(x))$, then $\Gamma(R_{\mathcal{X} \setminus \{s\}}(x)) \le V_x$, by \cite[Prop. 3.17 (3)]{cirillo2015metastability}.
         \item[ii)] Suppose that $x\not\in \mathcal{F}(R_{\mathcal{X}\setminus \{s\}}(x))$. Consider $\Tilde x=\text{argmin}_{x\in R_{\mathcal{X}\setminus \{s\}}(x)}H(x)$, then $\tilde x \in \mathcal{F}(R_{\mathcal{X}\setminus \{s\}}(x))$ and by \cite[Prop. 3.17 (2), (3)]{cirillo2015metastability} we have $V_x <\Gamma(R_{\mathcal{X} \setminus \{s\}}(x))=\Gamma(R_{\mathcal{X} \setminus \{s\}}(\Tilde x)) = V_{\Tilde x}$. So
        \begin{equation}
             \max_{y \in R_{\mathcal{X}\setminus \{s\}}(x)}V_y=V_{\tilde x}=\Gamma(R_{\mathcal{X}\setminus \{s\}}(x)).
        \end{equation}
    \end{itemize}
 From this follows that
 \begin{equation}
  \max_{x \in \mathcal{X}_0 \setminus \{s\}}\Gamma(R_{\mathcal{X} \setminus \{s\}}(x)) = \max_{x \in \mathcal{X}_0 \setminus \{s\}}\max_{y\in R_{\mathcal{X} \setminus \{s\}}(x)} V_y \le \activationenergym.
 \end{equation}
      \item If $x\in \mathcal{X}\setminus \mathcal{X}_0$, we proceed as follows
    \begin{itemize}
         \item[I)] If $\Phi(x,s)=H(x)$, then $R_{\mathcal{X} \setminus \{s\}}(x)=\{x\}$ because  $\{y \in \mathcal{X} \,\, | \,\, \Phi(x,y)<H(x)\}$ is empty. Indeed, $\Phi(x,y)$ is always greater than or equal to $H(x)$. So, $\Gamma(R_{\mathcal{X} \setminus \{s\}}(x))=\Gamma(\{x\})=0$.
         \item[II)] If $\Phi(x,s)>H(x)$, we choose $\Tilde x=\text{argmin}_{x\in R_{\mathcal{X}\setminus \{s\}}(x)}H(x)$, so $\Tilde x \in \mathcal{X}_0 \setminus \{s\}$ and $\Phi(x,s)=\Phi(\Tilde x, s)$. Then $\{y \in \mathcal{X} \,\, | \,\, \Phi(x,y)< \Phi(x,s)\} \subseteq R_{\mathcal{X} \setminus \{s\}}(\Tilde x)$ and we refer to the previous case $x \in \mathcal{X}_0$, since $\tilde x \in \mathcal{X}_0 \setminus \{s \}$.
    \end{itemize}
\end{itemize}
 This concludes the proof that $\activationenergym \ge \tilde\Gamma (\mathcal{X}\setminus \{s\})$ and hence that $\activationenergym = \tilde\Gamma (\mathcal{X}\setminus \{s\})$.

 The key step in \cite[Prop. 3.24]{nardi2016hitting} was to show that $H_2 = H_3$, $H_2$ is defined as \cite[Theorem 5.1]{catoni1999simulated}
\begin{equation}
 H_2:= \widetilde{\Gamma}(\mathcal{X} \setminus \{ x \}), \qquad x \in \text{argmin}_{x \in \mathcal{X}}G(x)
\end{equation}
 The critical depth $H_3$ is defined as \cite[Theorem 5.1]{catoni1999simulated}
\begin{equation}
H_3:=\widetilde{\Gamma}(\mathcal{X} \times \mathcal{X} \setminus F),
\end{equation}
 where $F=\{(x,x)| \, x \in \mathcal{X}\}$, $\widetilde{\Gamma}(\mathcal{X} \times \mathcal{X} \setminus F)=\max_{C \in \mathcal{M}(\mathcal{X} \times \mathcal{X} \setminus F)}\Gamma(C)$ and $\mathcal{M}(\mathcal{X} \times \mathcal{X}\setminus F)=\{C \in \mathcal{C}(\mathcal{X}) \,\,| \,\, C$ maximal cycle by inclusion under the constraint $C \subseteq \mathcal{X} \times \mathcal{X}\}$.
 Through the equivalence of two definitions of cycles, given by \cite[Prop. 3.10]{cirillo2015metastability}, the critical depth $H_2$ is equal to $\Tilde{\Gamma}(\mathcal{X} \setminus\{s\})$. This quantity is well defined because its value is independent of the choice of $s$ \cite[Theorem 5.1]{catoni1999simulated}. Now we consider two independent Markov chains, $X_t$ and $Y_t$, on the same energy landscape and with the same inverse temperature $\beta$. We define the two dimensional Markov chain $\{(X_t,Y_t)\}$ on $\mathcal{X} \times \mathcal{X}$ with transition probabilities $\mathcal{P}_{\beta}^{\otimes 2}$ given by
\begin{equation}
    \mathcal{P}_{\beta}^{\otimes 2}\Big( (x,y)(\tilde x,\tilde y)\Big)=
     \mathcal{P}_{\beta}(x,\tilde x)\mathcal{P}_{\beta}(y,\tilde y) \qquad \forall \, (x,y),(\tilde x,\tilde y) \in \mathcal{X}\times\mathcal{X}
\end{equation}
 So, using \cite[Theorem 5.1]{catoni1999simulated} and the assumption \eqref{P0}, the proof is concluded. \\
\end{proof}
Before proving the bounds \eqref{rocompreso}
\begin{equation*}
 c_1e^{-\beta(\activationenergym+\gamma_1)} \leq \rho_{\beta} \leq c_2e^{-\beta(\activationenergym-\gamma_2)},
\end{equation*}
 we recall the Definition \ref{dirichlet} and we define the \emph{generator} of a Markov process.
\begin{defn}
 For any function $f: \mathcal{X} \longrightarrow \mathbb{R}$, $\mathbb{L}_\beta f$ is the function defined as
\begin{equation}
     \mathbb{L}_\beta f(x):=\sum_{y \in \mathcal{X}}\mathcal{P}_\beta(x,y)[f(x)-f(y)]^2.
\end{equation}
\end{defn}
 The result \eqref{rocompreso} is an immediate consequence of the next two lemmas and it is obtained by generalizing \cite[Theorem 2.1, Lemma 2.3, Lemma 2.7]{holley1988simulated}.
\begin{lem}
 There exists a constant $C \leq \infty$ such that for all $\beta \geq 0$,
\begin{equation}
    \rho_\beta \leq C e^{- \beta (\activationenergym - \gamma)},
\end{equation}
 where $\gamma$ is a function of $\beta$ that vanishes for $\beta\to\infty$.
\end{lem}
 \proof We first observe that by assumption $\activationenergym>0$. Without loss of generality, we may assume that $x_0 \in \mathcal{X}^m, y_0 \in \mathcal{X}^s$ and $H(y_0)=0$. Therefore $\activationenergym=\Phi(x_0,y_0)-H(x_0)$ since $\mathcal{X}$ is finite. We write the spectral gap $\rho_\beta$ as
\begin{equation}
 \rho_\beta=\inf_{f \in L^2(\mu)} \frac{- \sum_{x \in \mathcal{X}} f(x) \mathbb{L}_\beta f(x) \mu(x)}{\text{Var}_\beta(f)},
\end{equation}
 where $\text{Var}_\beta(f):=\sum_{x \in \mathcal{X}} f^2(x)\mu(x)-(\sum_{x \in \mathcal{X}} f(x)\mu(x))^2$, and $L^2$ is the space of functions with finite second moment under the measure $\mu$. We will find a function $F$ and a constant $C<\infty$, such that
\begin{equation}
     \frac{- \sum_{x \in \mathcal{X}} F(x) \mathbb{L}_\beta F(x) \mu(x)}{\text{Var}_\beta(F)} \leq C e^{- \beta (\activationenergym - \gamma)}.
\end{equation}
 Let $x_0 \in \mathcal{X}$ and $y_0 \in \mathcal{I}_{x_0}$ be two points for which $\Phi(x_0,y_0)-H(x_0)=\activationenergym$ and let us consider the set $\mathcal{R}_{\mathcal{X} \setminus \{x_0\}}(y_0)=\{y_0\} \cup \{x \in \mathcal{X} \,\, | \,\, \Phi(y_0,x)< \Phi(y_0,x_0)\}$. Note that $x_0 \not\in \mathcal{R}_{\mathcal{X} \setminus \{x_0\}}(y_0)$ and $y_0 \in \mathcal{R}_{\mathcal{X} \setminus \{x_0\}}(y_0)$. Moreover if $x \in \mathcal{R}_{\mathcal{X} \setminus \{x_0\}}(y_0)$ and $y \not\in \mathcal{R}_{\mathcal{X} \setminus \{x_0\}}(y_0)$, then
\begin{equation}\label{HDP}
H(y)+\Delta(y,x) \geq \Phi(y_0,x_0).
\end{equation}
\begin{figure}[htb!]
    \centering
    \includegraphics{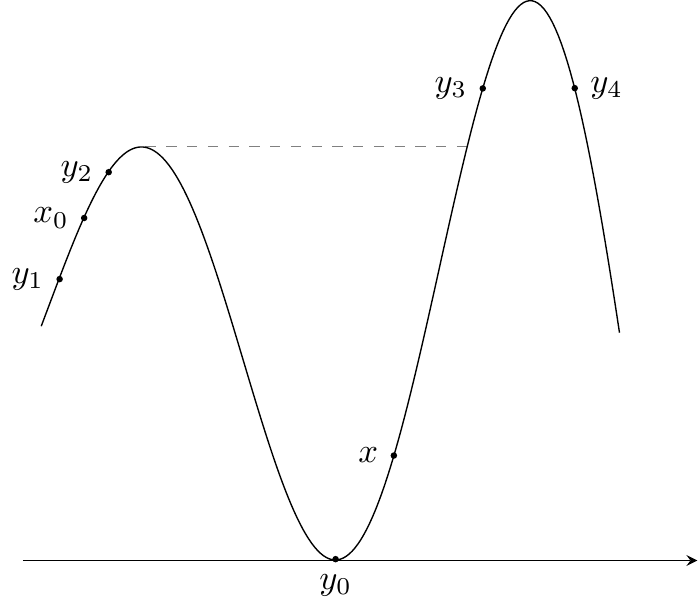}
     \caption{In this figure we draw an example energy-landscape, compatible with the assumptions on $x_0,y_0$ and $x$. We also draw four $y_i \not\in \mathcal{R}_{\mathcal{X} \setminus \{x_0\}}(y_0)$, $i=1,2,3,4$, for which \eqref{HDP} is valid.}
    \label{figure: exampleenergycycle}
\end{figure}
 For any $x \in  \mathcal{R}_{\mathcal{X} \setminus \{x_0\}}(y_0)$ and $y \not\in  \mathcal{R}_{\mathcal{X} \setminus \{x_0\}}(y_0)$, by reversibility we have
\begin{align}\label{eqdis}
   \mathcal{P}_\beta(x,y)\mu(x)=\mathcal{P}_\beta(y,x)\mu(y) & = e^{-\beta(-\frac{\log \mathcal{P}_\beta(y,x)}{\beta}-\frac{\log{\mu(y)}}{\beta})} \leq e^{-\beta(\Delta(y,x)+H(y)-\gamma^*_1)},
\end{align}
 where, to obtain the inequality, the first term is estimated by \eqref{Deltageneral1} and \cite[Equation (2.2)]{cirillo2015metastability}, i.e.,
\begin{equation}
     -\frac{\log \mathcal{P}_\beta(y,x)}{\beta} \geq \Delta(y,x)-\tilde\gamma_1.
\end{equation}
 The second term in \eqref{eqdis} is estimated by \eqref{Hamgeneral} and \eqref{GBMgeneral}, that is
\begin{equation}\label{gammino2}
-\frac{\log{\mu(y)}}{\beta} \geq H(y)-\tilde\gamma_2,
\end{equation}
 where $\tilde\gamma_1, \tilde\gamma_2$ and $\gamma^*_1=\tilde\gamma_1+\tilde\gamma_2$ are functions of $\beta$ that vanish for $\beta\to\infty$. Then using \eqref{HDP} we get
\begin{align}
   e^{-\beta(\Delta(y,x)+H(y)-\gamma^*_1)} \leq e^{-\beta \Phi(x_0,y_0)}e^{\beta \gamma^*_1}.
\end{align}
 Let $F(x)=\mathbb{1}_{\mathcal{R}_{\mathcal{X} \setminus \{x_0\}}(y_0)}(x)$, then
\begin{align}
     - \sum_{x \in \mathcal{X}} F(x) \mathbb{L}_\beta F(x) \mu(x) &= \frac{1}{2}\sum_{x,y \in \mathcal{X}} \mu(x)\mathcal{P}_\beta(x,y)[F(x)-F(y)]^2 \notag \\
     & \leq \sum_{\substack{x \in \mathcal{R}_{\mathcal{X} \setminus \{x_0\}}(y_0) \\ y \not \in \mathcal{R}_{\mathcal{X} \setminus \{x_0\}}(y_0)}} e^{-\beta(\Phi(x_0,y_0))}e^{\beta \gamma^*_1}.
\end{align}
On the other hand,
\begin{align}
     \text{Var}_\beta(f)=\mu(\mathcal{R}_{\mathcal{X} \setminus \{x_0\}}(y_0))\mu(\mathcal{R}_{\mathcal{X} \setminus \{x_0\}}(y_0)^c) & \geq \frac{e^{-\beta G(y_0)}}{Z}\frac{e^{-\beta G(x_0)}}{Z} \notag \\
     & \geq e^{-\beta (H(y_0)+\tilde\gamma_2)}e^{-\beta (H(x_0)+\tilde\gamma_2)} \notag \\
    & = e^{-\beta (H(x_0)+2\tilde\gamma_2)},
\end{align}
 where the last inequality is obtained by \eqref{gammino2}, and by our assumption $H(y_0)=0$. We conclude that
\begin{align*}
\rho_{\beta} \leq C e^{-\beta(\activationenergym-\gamma)}
\end{align*}
where $C$ is a constant and $\gamma=\gamma^*_1+2\tilde\gamma_2$.
\begin{lem}
There exists a constant $C>0$, such that for all $\beta \geq 0$,
\begin{equation}\label{ro2}
    \rho_\beta \geq C e^{- \beta (\activationenergym + \gamma)},
\end{equation}
 where $\gamma$ is a function of $\beta$ that vanishes for $\beta\to\infty$.
\end{lem}

 \proof It will be enough to find a constant $c>0$ such that for every $\beta \geq 0$ and every $f\in L^2(\mu)$,
\begin{equation}
     \frac{- \sum_{x \in \mathcal{X}} f(x) \mathbb{L}_\beta f(x) \mu(x)}{\text{Var}_\beta(F)} \geq C e^{- \beta (\activationenergym + \gamma)}.
\end{equation}
 We consider $x, y \in \mathcal{X}$ and $\omega \in \Theta(x,y)$ with length $|\omega|=n(x,y)$ and define
\begin{equation}\label{N}
    N:=\max_{x,y \in \mathcal{X}}n(x,y).
\end{equation}
 For $z \in \mathcal{X},w \in \mathcal{I}_z$, we define the function $\mathbb{F}_{(z,w)}: \Theta(x,y) \longrightarrow \{0,1\}$ as
\begin{equation}\label{F}
    \mathbb{F}_{(z,w)}(\omega):=
\bigg \{
\begin{array}{rl}
 1 & \text{if $\omega_i=z$ and $\omega_{i+1}=w$ for some $0 \leq i < n(x,y)$}, \\
0 & \text{otherwise}. \\
\end{array}
\end{equation}
Then,
\begin{align*}
     2\text{Var}_{\beta}(f) & =\sum_{x,y \in \mathcal{X}}(f(y)-f(x))^2\mu(y)\mu(x)= \sum_{x,y \in \mathcal{X}}\Bigg(\sum_{i=1}^{n(x,y)} f(\omega_i)-f(\omega_{i-1})\Bigg)^2\mu(y)\mu(x),
    \end{align*}
  where in the last equality we use that $\omega \in \Theta(x,y)$ with $|\omega|=n(x,y)$ and we wrote $f(y)-f(x)$ as a telescopic sum. Using \eqref{N} and \eqref{F}, we get the following inequalities
\begin{align}\label{2dis1}
     \sum_{x,y \in \mathcal{X}}\Bigg(\sum_{i=1}^{n(x,y)} f(\omega_i)-f(\omega_{i-1})\Bigg)^2\mu(x)\mu(y) & \leq \sum_{x,y \in \mathcal{X}} n(x,y)\sum_{i=1}^{n(x,y)} (f(\omega_i)-f(\omega_{i-1}))^2\mu(x)\mu(y) \notag \\
     & \leq N \sum_{x,y \in \mathcal{X}} \sum_{z,w \in \mathcal{X}} \mathbb{F}_{(z,w)}(\omega)(f(w)-f(z))^2\mu(x)\mu(y).
    \end{align}
    We estimate $\mu(x)\mu(y)$ as in \eqref{gammino2},
    \begin{equation}
     \mu(x)\mu(y)=e^{-\beta(-\frac{\log(\mu(x))}{\beta}-\frac{\log(\mu(y))}{\beta})} \leq e^{-\beta (H(x)+H(y)-2\tilde\gamma_2)}.
    \end{equation}
    Then we have
    \begin{align}\label{2dis2}
     & N \sum_{x,y \in \mathcal{X}} \sum_{z,w \in \mathcal{X}} \mathbb{F}_{(z,w)}(\omega)(f(w)-f(z))^2\mu(x)\mu(y) \notag \\
     & \leq N \sum_{x,y \in \mathcal{X}} \sum_{z,w \in \mathcal{X}} \mathbb{F}_{(z,w)}(\omega)(f(w)-f(z))^2 e^{-\beta \Phi(z,w)} \frac{e^{-\beta (H(x)+H(y)-2\tilde\gamma_2)}}{e^{-\beta \Phi(z,w)}} \notag \\
     & \leq N \Big(\max_{z,w} \sum_{x,y \in \mathcal{X}} \mathbb{F}_{(z,w)}(\omega) \frac{e^{-\beta (H(x)+H(y)-2\tilde\gamma_2)}}{e^{-\beta \Phi(z,w)}}\Big) \sum_{u,v \in \mathcal{X}} (f(v)-f(u))^2 e^{-\beta \Phi(u, v)}.
\end{align}
 Moreover
\begin{align}\label{2dis3}
     \mathbb{F}_{(z,w)}(\omega) \frac{e^{-\beta (H(x)+H(y)-2\tilde\gamma_2)}}{e^{-\beta \Phi(z,w)}} & = \mathbb{F}_{(z,w)}(\omega) e^{\beta(\Phi(z,w)-H(x)-H(y)+2\tilde\gamma_2)} \notag \\
     & \leq \mathbb{F}_{(z,w)}(\omega) e^{\beta(\Phi(x,y)-H(x)-H(y)+2\tilde\gamma_2)} \notag \\
     & \leq \mathbb{F}_{(z,w)}(\omega) e^{\beta{(\activationenergym+2\tilde\gamma_2)}}.
\end{align}
 The result \eqref{ro2} follows from \eqref{2dis1}, \eqref{2dis2}, \eqref{2dis3}.

\section{Proof of model-dependent results}\label{pmdr}
 In Section \ref{otherproofs} we prove the main model-dependent results except for Lemma \ref{EST}, which we postpone to Section \ref{Lemmamain}.

 \subsection{Proof of Theorem \ref{IdMS}, Theorem \ref{RP}, Theorem \ref{ThE}}\label{otherproofs}
 Note that our PCA verifies \cite[Definition 2.1]{cirillo2013relaxation}. In order to prove Theorem \ref{IdMS} we will lean on \cite[Theorem 2.4]{cirillo2013relaxation} (see Appendix \ref{appendixA}). Roughly speaking, if we have an ansatz for the set of metastable configurations and one for the communication height, and we show that these verify two conditions, then \cite[Theorem 2.4]{cirillo2013relaxation} guarantees that the anzatzes are correct.
 \begin{proof}[Proof of Theorem \ref{IdMS} (Identification of metastable states)] In \cite{cirillo2003metastability} the authors computed the value of $\Gamma$ to be $\activationenergy=-2h\lambda^2+2\lambda(4+h)-2h$. There, it was also proven that
\begin{align}
 \Phi(\underline{-1},\underline{+1})-H(\underline{-1})=\activationenergy, \\
\Phi(\underline{c},\underline{+1})-H(\underline{c})=\activationenergy.
\end{align}
 By \cite[Lemma 3.4, Lemma 4.1]{cirillo2003metastability} we have that $\Phi(-\underline{1},\underline{c})=\activationenergy+H(-\underline{1})$, that is $\activationenergy+H(-\underline{1})$ is the minmax between $-\underline{1}$ and $\underline{c}$.
 The first assumption of \cite[Theorem 2.4]{cirillo2013relaxation} is satisfied for $A=\{\underline{-1},\underline{c}\}$ and $a=\activationenergy$ thanks to \cite[Theorem 3.11, Lemma 3.4, Lemma 4.1]{cirillo2003metastability}, hence
\begin{equation}
     \Phi(\sigma, \mathcal{X}^s)-H(\sigma)=\activationenergy \text{ for all } \sigma\in \{\underline{-1},\underline{c}\}.
\end{equation}
 Moreover, the second assumption of \cite[Theorem 2.4]{cirillo2013relaxation} is satisfied because by Lemma \ref{EST} either $\mathcal{X}\setminus (\{\underline{-1},\underline{c}\} \cup \mathcal{X}^s)=\emptyset$ or
\begin{equation}
 V_\sigma<\activationenergy \text { for all } \sigma\in \mathcal{X}\setminus (\{\underline{-1},\underline{c}\} \cup \mathcal{X}^s).
\end{equation}
 Finally, by applying \cite[Theorem 2.4]{cirillo2013relaxation}, we conclude that $\activationenergym=\activationenergy$ and $\mathcal{X}^m=\{\underline{-1}, \underline{c}\}$.
\end{proof}
 \begin{proof}[Proof of Theorem \ref{RP} (Recurrence property)] In Lemma \ref{EST} we compute $V^*=2(2-h)$. Recall the definition of $\mathcal{X}_V$ in \eqref{Xv} and apply \cite[Prop. 2.8]{cirillo2015metastability} with $a=V^*$, $\mathcal{X}_{V^*}=\{\underline{-1},\underline{c}, \underline{+1}\}$. We get
\begin{equation}
 \beta \mapsto \sup_{\eta\in \mathcal{X}}\mathbb{P}_{\eta}(\tau_{\mathcal{X}^m\cup\mathcal{X}^s}>e^{\beta (V^{*}+\epsilon)}) \qquad \text{is SES.}
\end{equation}
 With a similar reasoning with $a=\activationenergym$, $\mathcal{X}_{\activationenergym}=\mathcal{X}^s$, we get
\begin{equation}
 \beta \mapsto \sup_{\eta\in \mathcal{X}}\mathbb{P}_{\eta}(\tau_{\mathcal{X}^s}>e^{\beta (\activationenergym +\epsilon)}) \qquad \text{is SES.}
\end{equation}
\end{proof}
 \begin{proof}[Proof of Theorem \ref{ThE}] In \cite{cirillo2016sum} the proof of \cite[Theorem 3.1]{cirillo2016sum} was only sketched in Section 4. Recall Theorem \ref{t:addition02}, then Condition~\ref{t:series00} is satisfied thanks to our Theorem \ref{IdMS}, , Condition~\ref{t:series01} is satisfied thanks to \cite[Lemma 3.3, Lemma 3.4]{cirillo2016sum} and Condition~\ref{t:series02} is satisfied thanks to \cite[Lemma 3.5]{cirillo2016sum}. Thus, applying Theorem \ref{t:addition02} concludes the rigorous proof of \eqref{valoreatteso}.
 In the second example of Section \ref{examples} we verify the assumptions of Theorem \ref{GM} and Theorem \ref{GMM} for general reversible PCA model in order to get \eqref{lim2PCA} and \eqref{rocompresopca}.
\end{proof}

\subsection{Proof of main Lemma \ref{EST}}\label{Lemmamain}

\begin{defn}\label{stablepair}
 We call \emph{stable configurations} those configurations $\sigma \in \mathcal{X}$ such that $p(\sigma,\sigma)\rightarrow 1$ in the limit $\beta \rightarrow \infty$. Equivalently, $\sigma \in \mathcal{X}$ is a stable configuration if and only if $p(\sigma,\eta)\rightarrow 0$ in the limit $\beta \rightarrow \infty$ for all $\eta \in \mathcal{X}\setminus \{ \sigma \}$.
\end{defn}

 For any $\sigma\in \mathcal{X}$ there exists a unique configuration $\eta\in \mathcal{X}$ such that the transition $\sigma \rightarrow \eta$ happens with high probability as $\beta \rightarrow \infty$, that is $p(\sigma, \eta) \overset{\beta \rightarrow \infty}{\longrightarrow} 1$. So let $\eta$ and $\sigma$ be two configurations in $\mathcal{X}$ such that $\eta=T\sigma$, where
\[
\begin{aligned}
T\colon \mathcal{X} &\to \mathcal{X}\\
\sigma &\mapsto T\sigma
\end{aligned}
\]
is the map such that for each $x\in \Lambda$
$$T\sigma (x) =
\bigg \{
\begin{array}{rl}
 \sigma^x(x) & \text{if} \, \, \, p_x(\sigma^x(x)|\sigma)\overset{\beta \rightarrow \infty}{\longrightarrow} 1 \\
 \sigma(x) & \text{if} \, \, \, p_x(\sigma^x(x)|\sigma)\overset{\beta \rightarrow \infty}{\longrightarrow} 0\\
\end{array}
$$
\begin{defn}
 Let $\sigma, \eta \in \mathcal{X}$ be two different configurations. We say that $\sigma$ and $\eta$ form a stable pair if and only if $\eta= T \sigma$ and $T \eta=\sigma$. Moreover, we say that $\sigma \in \mathcal{X}$ is a trap if either $\sigma$ is a stable configuration or the pair $(\sigma,T\sigma)$ is a stable pair. We denote by $\mathcal{T} \subset\mathcal{X}$ the collection of all traps.
\end{defn}

  We define two further maps, that will be useful later on. For any given $j\in \Lambda$, $T_j^{F}(\sigma)=T(\sigma)$ except in the site $j$, where $T_j^{F}(\sigma)=\sigma(j)$. Formally,
     \begin{equation}
    T_j^{F}\sigma (i) =
    \begin{cases}
     \sigma_{\mathcal{X} \setminus \{j\}}^i(i) & \text{if} \, \, \, p_i(\sigma^i(i)|\sigma)\overset{\beta \rightarrow \infty}{\longrightarrow} 1, \\
     \sigma_{\mathcal{X} \setminus \{j\}}(i) & \text{if} \, \, \, p_i(\sigma^i(i)|\sigma)\overset{\beta \rightarrow \infty}{\longrightarrow} 0,\\
    \sigma(j) & \text{if} \, \, \, i=j.
    \end{cases}
    \end{equation}
     For any given $j\in \Lambda$, $T_j^{C}(\sigma)=T(\sigma)$ except in the site $j$, where $T_j^{C}(\sigma)=-\sigma(j)$. Formally,
    \begin{equation}
    T_j^{C}\sigma (i) =
    \begin{cases}
     \sigma_{\mathcal{X} \setminus \{j\}}^i(i) & \text{if} \, \, \, p_i(\sigma^i(i)|\sigma)\overset{\beta \rightarrow \infty}{\longrightarrow} 1, \\
     \sigma_{\mathcal{X} \setminus \{j\}}(i) & \text{if} \, \, \, p_i(\sigma^i(i)|\sigma)\overset{\beta \rightarrow \infty}{\longrightarrow} 0,\\
    -\sigma(j) & \text{if} \, \, \, i=j.
    \end{cases}
    \end{equation}
 The two maps are similar to $T(\sigma)$, the only difference being that $T_j^{F}(\sigma)$ fixes the value of the spin in $j$ and $T_j^{C}(\sigma)$ changes the value of the spin in $j$.

 We say that $x,y \in \Lambda$ are \emph{nearest neighbors} if and only if the lattice distance $d$ between $x,y$ is one, i.e., $d(x,y)=1$. We indicate by $R_{l,m} \subseteq \Lambda$ the rectangle with sides $l$ and $m$, $2 \le l \le m$ and we call \emph{non-interacting rectangles} two rectangles $R_{l,m}$ and $R_{l',m'}$ such that any of the following conditions hold:
\begin{itemize}
     \item $d(R_{l,m},R_{l',m'})\geq 3$, if $\sigma_{R_{l,m}}=\underline{c}^o_{R_{l,m}}$ and $\sigma_{R_{l',m'}}=\underline{c}^o_{R_{l',m'}}$;
     \item $d(R_{l,m},R_{l',m'})\geq 3$, if $\sigma_{R_{l,m}}=\underline{c}^e_{R_{l,m}}$ and $\sigma_{R_{l',m'}}=\underline{c}^e_{R_{l',m'}}$;
     \item $d(R_{l,m},R_{l',m'})\geq 3$, if $\sigma_{R_{l,m}}=\underline{+1}_{R_{l,m}}$ and $\sigma_{R_{l',m'}}=\underline{+1}_{R_{l',m'}}$;
     \item $d(R_{l,m},R_{l',m'})=1$, if $\sigma_{R_{l,m}}=\underline{c}^o_{R_{l,m}}$ and $\sigma_{R_{l',m'}}=\underline{c}^e_{R_{l',m'}}$;
     \item $d(R_{l,m},R_{l',m'})=1$, if $\sigma_{R_{l,m}}=\underline{c}_{R_{l,m}}$, $\sigma_{R_{l',m'}}=\underline{+1}_{R_{l',m'}}$ and the sides on the interface are of the same length.
\end{itemize}
 Whenever two rectangles are not \emph{non-interacting}, we call them \emph{interacting}.
\begin{proof}[Proof of Lemma \ref{EST}]
 We begin by giving a rough sketch of the proof. Without loss of generality, we consider only configurations in $\mathcal{U}:=\mathcal{X}_0\setminus\{\underline{-1},\underline{c},\underline{+1}\}$, since the configurations in $\mathcal{X} \setminus \mathcal{X}_0$ have stability level zero. Indeed, if $\sigma \in \mathcal{X} \setminus \mathcal{X}_0$, we construct the path $\overline{\omega}=(\sigma, T(\sigma))$, so that $T(\sigma) \in \mathcal{I}_\sigma$ and $V_\sigma=0$, where $\mathcal{I_{\sigma}}$ was defined in \eqref{I}. We will partition $\mathcal{X}_0\setminus\{\underline{-1},\underline{c},\underline{+1}\}$ into several subsets $A,B,D,E$ and for each of these we will construct a path $\overline\omega \in \Theta(\sigma,\mathcal{I_{\sigma}} \cap \mathcal{X}_0)$. Denote with $\sigma_{\Lambda'}$ a configuration $\sigma \in \Lambda' \subseteq \Lambda$. We will find an explicit upper-bound $V^*_{\sigma}$ on the transition energy along $\overline{\omega}$ as
\begin{equation}\label{maxH}
 \max_{k=1,...,|\overline{\omega}|-1}H(\omega_k,\omega_{k+1})-H(\sigma) \le V^*_{\sigma}.
\end{equation}
We define
\begin{equation}\label{maxV}
V^*_S=\max_{\sigma \in S}V^*_{\sigma}, \qquad S \in \{A,B,D,E\},
\end{equation}
and since
\begin{equation}
    \max_{S \in \{A,B,D,E\}} V^*_S< \activationenergy,
\end{equation}
 from \eqref{maxH} and \eqref{maxV} follows that, for any $\sigma \in \mathcal{X}_0 \setminus \{\underline{-1}, \underline{c}, \underline{+1}\}$,
\begin{equation}
 \Phi(\sigma, \mathcal{I_{\sigma}})-H(\sigma)=\min_{\omega \in \Theta(\sigma, \eta)} \max_{i=1,...,|\omega| -1} H(\omega_i, \omega_{i+1}) -H(\sigma) < \activationenergy.
\end{equation}
 This means that all configurations in $\mathcal{X}_0\setminus\{\underline{-1},\underline{c},\underline{+1}\}$ have a lower stability level than $\activationenergy$. We now proceed with the detailed proof. We partition the set $\mathcal{X}_0\setminus\{\underline{-1},\underline{c},\underline{+1}\}$ into four subset as  $\mathcal{X}_0\setminus\{\underline{-1},\underline{c},\underline{+1}\}=A \cup B \cup D \cup E$ \cite[Prop. 3.3]{cirillo2003metastability}. For each set $A,B,D,E$, we first describe it in words and then give its formal definition.

 We define the set $A$ to be the set of configurations consisting of a single rectangle containing either $\underline{c}$ or $\underline{+1}$, and surrounded by either $\underline{c}$ or $\underline{-1}$, see Figure \ref{figure: Aa}. More precisely, $A=A_1 \cup A_2 \cup A_3 \cup A_4 \cup A_5 \cup A_6$, where:
    \begin{itemize}
     \item $A_1$ is the collection of configurations such that $\exists! \, R_{l,m} \subset \Lambda$ with $l<\lambda$, $\sigma_{R_{l,m}}=\underline{c}_{R_{l,m}}$ and $\sigma_{\Lambda \setminus R_{l,m}}=\underline{-1}_{\Lambda \setminus R_{l,m}}$;
     \item $A_2$ is the collection of configurations such that $\exists! \, R_{l,m}\subset \Lambda$ with $l \ge \lambda$, $\sigma_{R_{l,m}}=\underline{c}_{R_{l,m}}$ and $\sigma_{\Lambda \setminus R_{l,m}}=\underline{-1}_{\Lambda \setminus R_{l,m}}$;
     \item $A_3$ is the collection of configurations such that $\exists! \, R_{l,m}\subset \Lambda$ with $l<\lambda$, $\sigma_{R_{l,m}}=\underline{+1}_{R_{l,m}}$ and $\sigma_{\Lambda \setminus R_{l,m}}=\underline{c}_{\Lambda \setminus R_{l,m}}$;
     \item $A_4$ is the collection of configurations such that $\exists! \, R_{l,m}\subset \Lambda$ with $l \ge \lambda$, $\sigma_{R_{l,m}}=\underline{+1}_{R_{l,m}}$ and $\sigma_{\Lambda \setminus R_{l,m}}=\underline{c}_{\Lambda \setminus R_{l,m}}$;
     \item $A_5$ is the collection of configurations such that $\exists! \, R_{l,m}\subset \Lambda$ with $l<\lambda$, $\sigma_{R_{l,m}}=\underline{+1}_{R_{l,m}}$ and $\sigma_{\Lambda \setminus R_{l,m}}=\underline{-1}_{\Lambda \setminus R_{l,m}}$;
     \item $A_6$ is the collection of configurations such that $\exists! \, R_{l,m}\subset \Lambda$ with $l \ge \lambda$, $\sigma_{R_{l,m}}=\underline{+1}_{R_{l,m}}$ and $\sigma_{\Lambda \setminus R_{l,m}}=\underline{-1}_{\Lambda \setminus R_{l,m}}$.
    \end{itemize}
    \begin{figure}[!htb]
        \centering
        \includegraphics[scale=0.8]{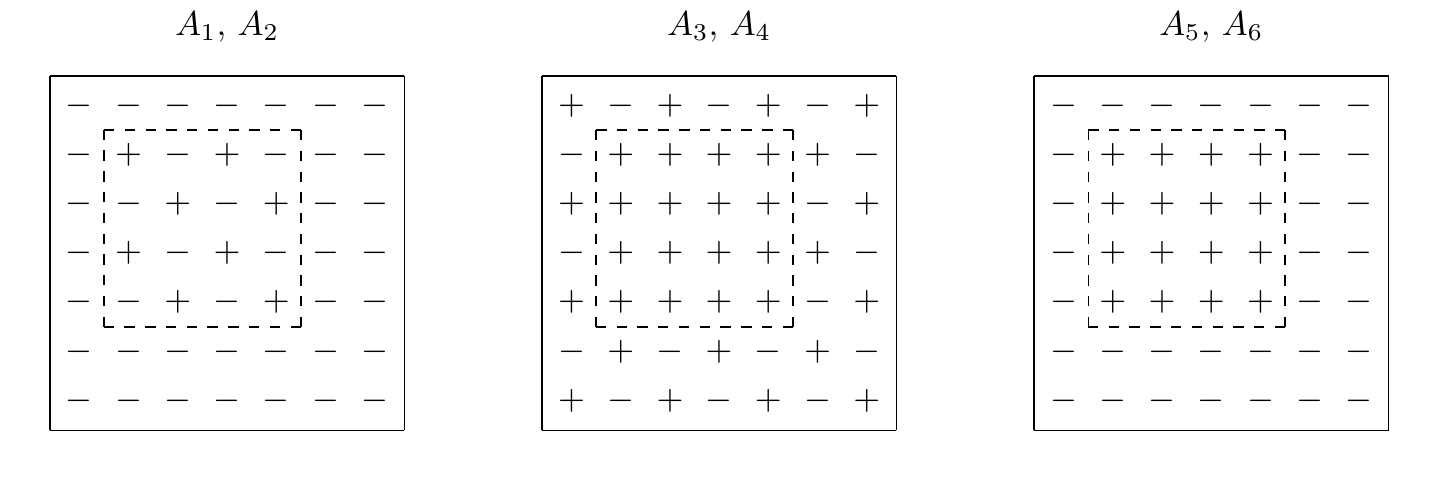}
        \caption{Examples of one configurations in $A$.}
        \label{figure: Aa}
    \end{figure}
 Configurations in the set $B$ consist of a single chessboard rectangle which may contain an island of $\underline{+1}$, surrounded by $\underline{-1}$, see Figure \ref{figure: Bb}. More precisely, $B=B_1 \cup B_2 \cup B_3$, where:
    \begin{itemize}
     \item $B_1$ is the collection of configurations such that $\exists! \, R_{l,m}$ with $\sigma_{R_{l,m}}=\underline{+1}_{R_{l,m}}$ and $\exists! \, R_{l',m'} \supsetneq R_{l,m}$ with $l'<\lambda$, $\sigma_{R_{l',m'} \setminus R_{l,m}}=\underline{c}_{R_{l',m'} \setminus R_{l,m}}, \,\,\, \sigma_{\Lambda \setminus R_{l',m'}}=\underline{-1}_{\Lambda \setminus R_{l',m'}}$;
     \item $B_2$ is the collection of configurations such that $\exists! \, R_{l,m}$ with $l \ge \lambda$, $\sigma_{R_{l,m}}=\underline{+1}_{R_{l,m}}$ and $\exists! \, R_{l',m'} \supsetneq R_{l,m}$ such that $\sigma_{R_{l',m'} \setminus R_{l,m}}=\underline{c}_{R_{l',m'} \setminus R_{l,m}}, \,\,\, \sigma_{\Lambda \setminus R_{l'-,m'}}=\underline{-1}_{\Lambda \setminus R_{l',m'}}$;
     \item $B_3$ is the collection of configurations such that $\exists! \, R_{l,m}$ with $l < \lambda$, $\sigma_{R_{l,m}}=\underline{+1}_{R_{l,m}}$ and $\exists! \, R_{l',m'} \supsetneq R_{l,m}$ with $l' \ge \lambda$ such that $\sigma_{R_{l',m'} \setminus R_{l,m}}=\underline{c}_{R_{l',m'} \setminus R_{l,m}}, \,\,\, \sigma_{\Lambda \setminus R_{l',m'}}=\underline{-1}_{\Lambda \setminus R_{l',m'}}$.
    \end{itemize}
    \begin{figure}[htb!]
        \centering
        \includegraphics[scale=0.8]{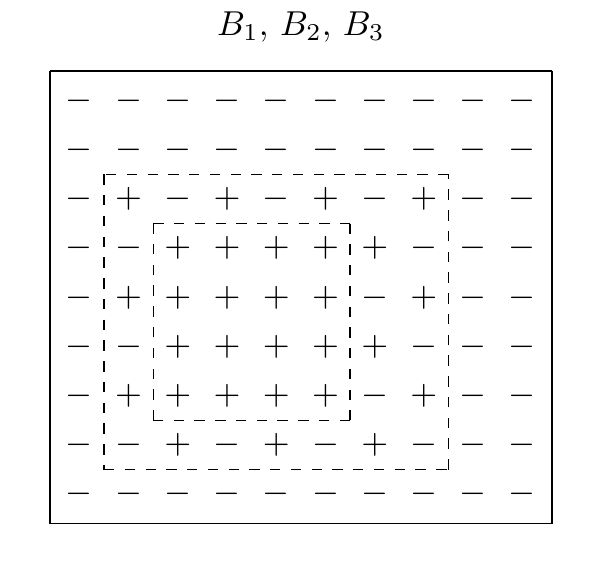}
        \caption{Examples of one configurations in $B$.}
        \label{figure: Bb}
    \end{figure}
 The set $D$ contains all configurations with more than one rectangle, see Figure \ref{figure: Dd}. More precisely, $D=D_1 \cup D_2 \cup D_3 \cup D_4 \cup D_5 \cup D_6$, where:
    \begin{itemize}
     \item $D_1$ is the collection of configurations such that there exist subcritical non-interacting rectangles $\mathcal{R}:=(R_{l,m})_{l,m}$ such that $\sigma_{\Lambda \setminus \mathcal{R}}=\underline{-1}_{\Lambda \setminus \mathcal{R}}$ and any rectangle of chessboard may contain one or more non-interacting rectangles of pluses;
     \item $D_2$ is the collection of configurations such that there exist non-interacting rectangles $\mathcal{R}:=(R_{l,m})_{l,m}$ where at least one of them is supercritical and such that $\sigma_{\Lambda \setminus \mathcal{R}}=\underline{-1}_{\Lambda \setminus \mathcal{R}}$. Moreover, any rectangle of chessboard may contain one or more non-interacting rectangles of pluses;
     \item $D_3$ is the collection of configurations consisting of interacting rectangles $\mathcal{R}:=(R_{l,m})_{l,m}$ with $l<\lambda$ and such that any rectangle of chessboard may contain one or more non-interacting rectangles of pluses; 
     \item $D_4$ is the collection of configurations consisting of non-interacting rectangles $\mathcal{R}:=(R_{l,m})_{l,m}$ with $l<\lambda$ such that $\sigma_{R_{l,m}}=\underline{+1}_{R_{l,m}}$ and $\sigma_{\Lambda \setminus \mathcal{R}}=\underline{c}_{\Lambda \setminus \mathcal{R}}$;
     \item $D_5$ is the collection of configurations consisting of rectangles $\mathcal{R}:=(R_{l,m})_{l,m}$ where at least one has $l\geq\lambda$ and such that $\sigma_{R_{l,m}}=\underline{+1}_{R_{l,m}}$ and $\sigma_{\Lambda \setminus \mathcal{R}}=\underline{c}_{\Lambda \setminus \mathcal{R}}$;
    \end{itemize}
    \begin{figure}[htb!]
        \centering
        \includegraphics[scale=0.8]{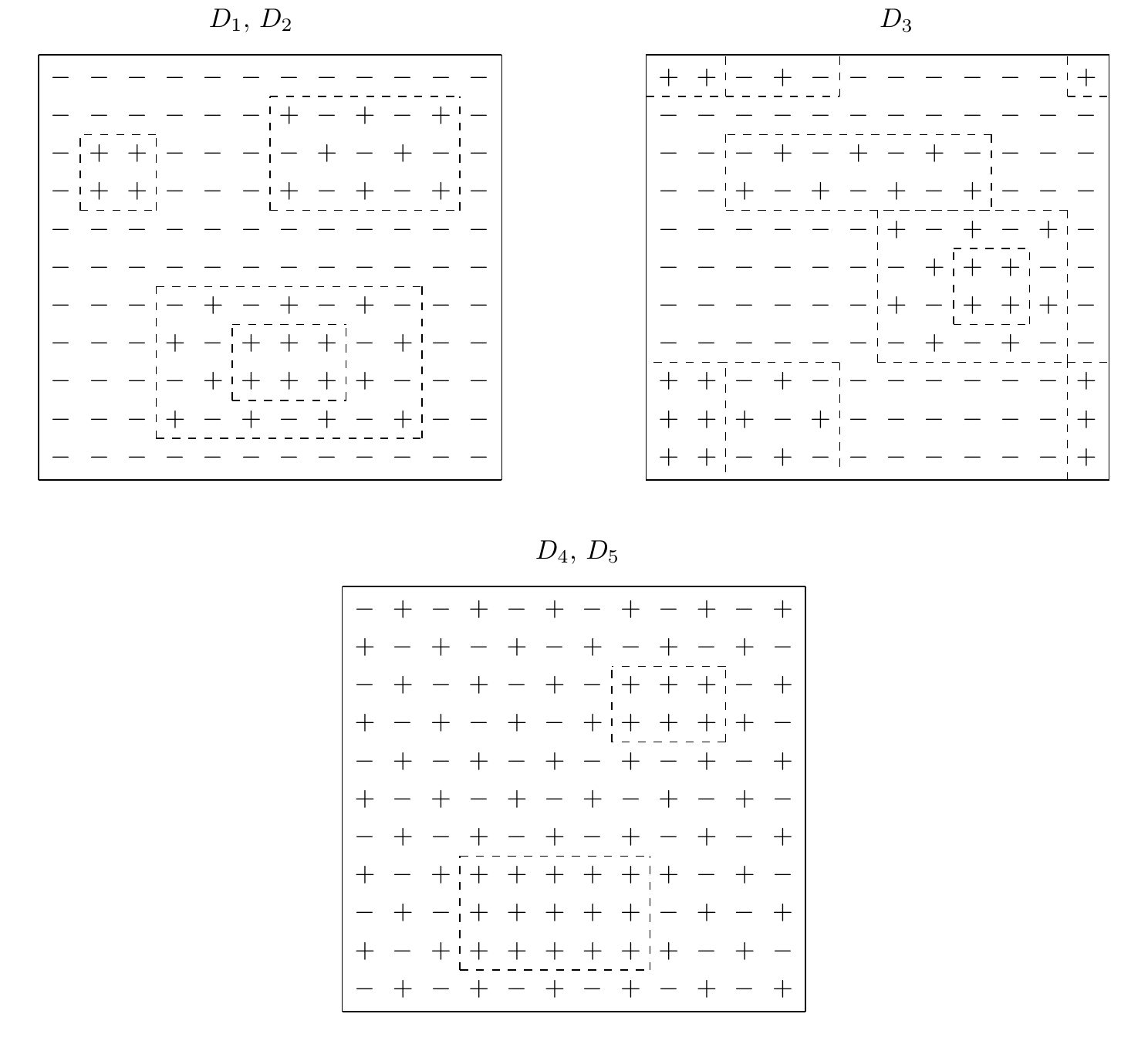}
        \caption{Examples of configurations in $D$.}
        \label{figure: Dd}
    \end{figure}
     The set $E$ contains all possible strips, that is, rectangles winding around the torus, see Figure \ref{figure: Ee}. More precisely, $E=E_1 \cup E_2 \cup E_3 \cup E_4 \cup E_5 \cup E_6 \cup E_7$, where:
    \begin{itemize}
     \item $E_1$ is the collection of configurations containing strips of $\underline{c}$ of width one surrounded by $\underline{-1}$, and possibly rectangles of $\underline{+1}$ and $\underline{c}$;
     \item $E_2$ is the collection of configurations containing strips of $\underline{+1}$ of width one surrounded by $\underline{c}$, and possibly rectangles of $\underline{+1}$;
     \item $E_3$ is the collection of configurations containing strips of $\underline{+1}$ of width one surrounded by $\underline{-1}$, and possibly rectangles of $\underline{+1}$ and $\underline{c}$;
     \item $E_4$ is the collection of configurations containing pairs of adjacent strips of $\underline{c}$ and $\underline{-1}$. For at least one of these pairs, both strips have width greater than one. Furthermore, there may be rectangles of $\underline{c}$ and $\underline{+1}$ surrounded by $\underline{-1}$, and rectangles of $\underline{+1}$ surrounded by $\underline{c}$;
     \item $E_5$ is the collection of configurations containing pairs of adjacent strips of $\underline{c}$ and $\underline{+1}$. For at least one of these pairs, both strips have width greater than one. Furthermore, there may be rectangles of $\underline{+1}$ surrounded by $\underline{c}$;
     \item $E_6$ is the collection of configurations containing pair of adjacent strips of $\underline{+1}$ and $\underline{-1}$. For at least one these pairs, both strips have width greater than one. Furthermore, there may be rectangles of $\underline{c}$ and $\underline{+1}$ surrounded by $\underline{-1}$;
     \item $E_7$ is the collection of configurations containing strips of $\underline{c}$, $\underline{-1}$ and $\underline{+1}$ with at least one width greater than one, and possibly rectangles of $\underline{c}$ and $\underline{+1}$ in $\underline{-1}$, and possibly rectangles of $\underline{+1}$ in $\underline{c}$;
    \end{itemize}
    \begin{figure}[htb!]
        \centering
        \includegraphics[scale=0.8]{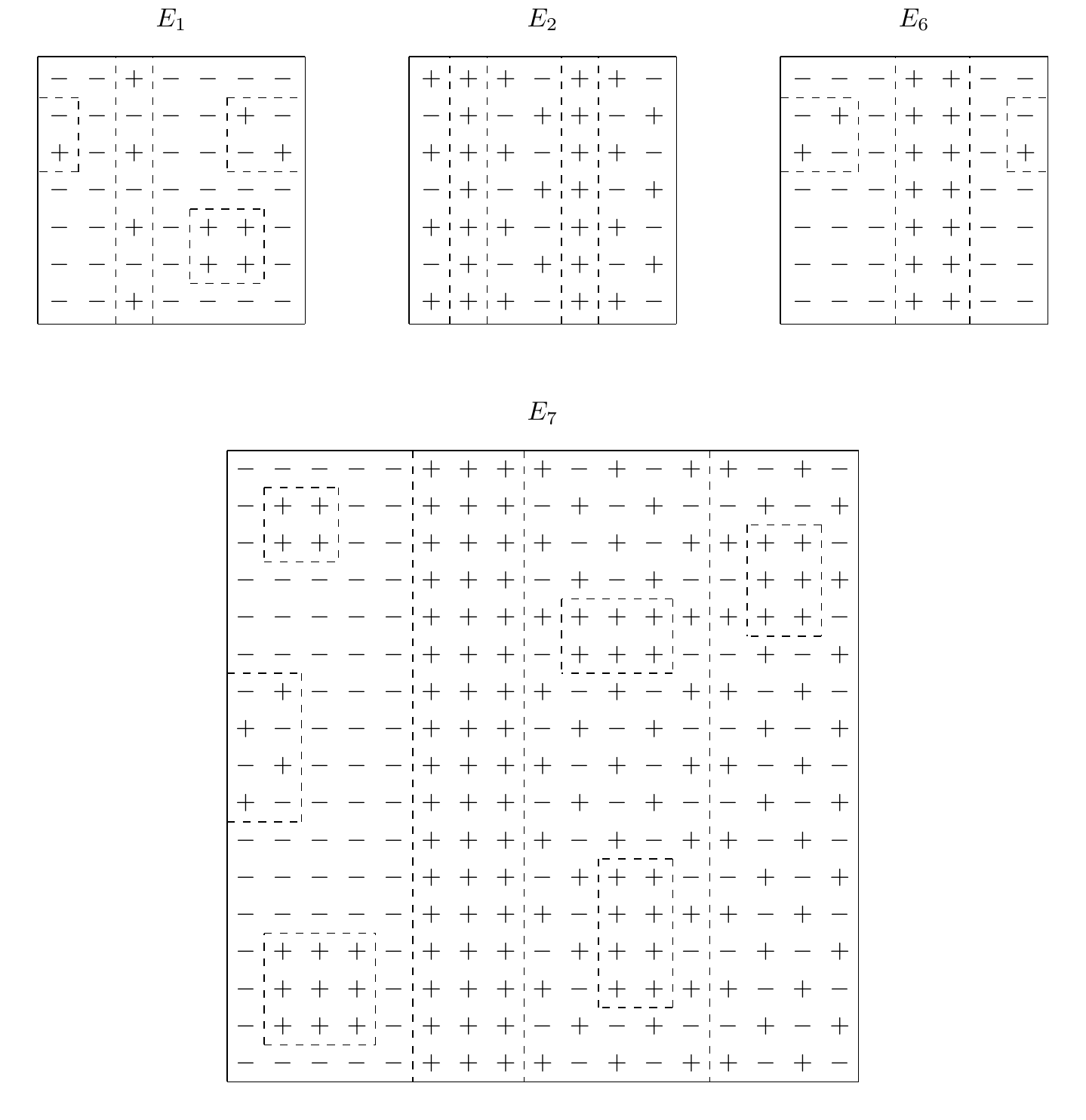}
        \caption{Examples of configurations in $E$.}
        \label{figure: Ee}
    \end{figure}

    We begin by considering the set $A$. Consider first the set $A_1$.
     \paragraph{Case $A_1$.} For any configuration $\sigma\in A_1$ we construct a path that begins in $\sigma$ and ends in a configuration in $A_1 \cup \{\underline{-1}\}$ with lower energy than $\sigma$, i.e., $\overline{\omega}\in \Theta(\sigma,\mathcal{I_{\sigma}} \cap (A_1 \cup \{\underline{-1}\}))$. We now fix $\sigma\equiv \omega_1\in A_1$ and we begin by defining $\omega_2$. If there is a minus corner in $\sigma_{R_{l,m}}$ , say in $j_1$, then $\sigma(j_1)$ is kept fixed and all other spins in the rectangle switch sign, i.e., $\omega_2:=T^{F}_{j_1}(\omega_1)$. On the other hand, if there is no minus corner in $\sigma_{R_{l,m}}$, then we call the next configuration in the path $\omega_1'$ and we define it as $\omega_1':=T(\omega_1)$, i.e., all the spins in the rectangle switch sign. After this step, $\omega_1'$ has a minus corner, so we can proceed as above and define $\omega_2:=T^{F}_{j_1}(\omega_1')$. Note that in $\omega_2$ there are two minus corners in the rectangle that are nearest neighbors of $j_1$. For the next step, keep fixed the minus corner that is contained in a side of length $l$, say in $j_2$, and define $\omega_3:=T^{F}_{j_2}(\omega_2)$. By iterating this procedure $l-2$ times, a full slice of the droplet is erased and we obtain the configuration $\eta \equiv \omega_l$ such that $\eta_{R_{l,m-1}}=\underline{c}$ and $\eta_{\Lambda \setminus R_{l,m-1}}=\underline{-1}$. In order to determine where the maximum of the transition energy is attained, we rewrite for $k=1, \ldots, l-1$
    \begin{align}\label{firstdifference}
       H(\omega_k,\omega_{k+1})-H(\omega_1) &
        =H(\omega_k) + \Delta(\omega_k,\omega_{k+1}) -H(\omega_1) \notag \\
        & =\sum_{m=1}^{k-1}(H(\omega_{m+1})-H(\omega_m))+\Delta(\omega_k,\omega_{k+1}),
    \end{align}
     with the convention that a sum over an empty set is equal to zero. From the reversibility property of the dynamics follows that
        \begin{equation}\label{reversibility}
             H(\omega_k)+\Delta(\omega_k,\omega_{k+1})=H(\omega_{k+1})+\Delta(\omega_{k+1},\omega_k),
        \end{equation}
         and since $\Delta(\omega_{k+1},\omega_k)=0$ for $k=1,\ldots,l-2$, for the path $\overline{\omega}$,
        \begin{align}\label{secdiff1}
            H(\omega_k,\omega_{k+1})-H(\omega_1)=
            \begin{cases}
             \sum_{m=1}^{k}(H(\omega_{m+1})-H(\omega_m)) & \text{if } k=1,\ldots,l-2,\\
             \sum_{m=1}^{l-2}(H(\omega_{m+1})-H(\omega_m))+\Delta(\omega_{l-1},\omega_l) & \text{if } k=l-1.
            \end{cases}
        \end{align}
     It can be shown that $H(\omega_{m+1})-H(\omega_m)=2h>0$ for $m=1,\ldots,l-2$ and $\Delta(\omega_{l-1},\omega_l)=2h$ \cite[Tab. 1]{cirillo2003metastability}, so the maximum is attained in the pair of configurations $(\omega_{l-1}, \omega_l)$. Hence,
    \begin{align}\label{VA1}
     \max_{\omega_k,\omega_{k+1} \in \overline{\omega}}H(\omega_k,\omega_{k+1})-H(\omega_1) &
     =\sum_{m=1}^{l-2}(H(\omega_{m+1})-H(\omega_m))+\Delta(\omega_{l-1},\omega_l) \notag \\
    & =2h(l-2)+2h=2h(l-1):=V_{\sigma}^*.
    \end{align}
     Since $V^*_{\sigma}$ depends only on the length $l$, we find $V^*_{A_1}=\max_{\sigma \in A_1} V^*_{\sigma}$ by taking the maximum over $l$. Since $l<\lambda$, we have
    \begin{equation}
        V^*_{A_1} < 2 (2-h).
    \end{equation}
     Finally, let us check that $\omega_l \in \mathcal{I_{\sigma}} \cap (A_1 \cup \{\underline{-1}\})$. Using \eqref{reversibility}, \eqref{VA1} and \cite[Tab. 1]{cirillo2003metastability}, we get
    \begin{align}
    H(\omega_1)+2h(l-1)=H(\omega_l)+2(2-h).
    \end{align}
     The rectangle $R_{l,m}$ is subcritical if and only if $l<2/h$, and so
    \begin{equation}
        H(\omega_1)-H(\omega_l)=4-2hl>0,
    \end{equation}
    which concludes the proof for $A_1$.
\paragraph{Case $A_2$.} For any configuration $\sigma \in A_2$ we construct
a path that begins in $\sigma$ and ends in a configuration in
$A_2 \cup \{\underline{c}\}$ with lower energy than
$\sigma$, i.e., $\overline{\omega}\in\Theta(\sigma,\mathcal{I_{\sigma}}
\cap (A_2 \cup \{\underline{c}\}))$.
We now fix $\sigma\equiv \omega_1\in A_2$ and we begin by
defining $\omega_2$. We call $j\in R_{l,m}$ a site in one of the sides
of length $l$ and such that $\sigma(j)=+1$. Furthermore, we call
$j_1 \in \Lambda \setminus R_{l,m}$ the nearest neighbor of $j$
such that (necessarily) $\sigma(j_1)=-1$ and we define
$\omega_2:=T^{C}_{j_1}(\omega_1)$, i.e., $\sigma(j_1)$
switches sign and the signs of all other sites in
$\sigma_{\Lambda \setminus R_{l,m}}$ remain fixed.
We define $\omega_3:=T(\omega_2)$, $\omega_4:=T(\omega_3)=T^2(\omega_2)$
and so on until a new slice is filled with chessboard.
We obtain the configuration $\eta$ such that
$\eta_{R_{l,m+1}}=\underline{c}$ and
$\eta_{\Lambda \setminus R_{l,m+1}}=\underline{-1}$.
Note that at the first step of the dynamics either one or two nearest
neighbors of $j_1$ in the external side of the rectangle switch sign
when $T$ is applied. Analogously, at each subsequent application
of $T$, either one of two further sites in the external side of the
rectangle switch sign. Therefore, the maximum number of iterations of
the map $T$ is $l-1$. In order to determine where the maximum of the
transition energy is attained, we rewrite the energy difference as
in \eqref{firstdifference}. Using \eqref{reversibility} and since
$\Delta(\omega_k,\omega_{k+1})=0$ for $k=2,\ldots,l-1$, for the
path $\overline{\omega}$,
        \begin{align}\label{secdiff2}
            H(\omega_k,\omega_{k+1})-H(\omega_1)=
            \begin{cases}
             \Delta(\omega_1,\omega_2)+\sum_{m=2}^{k}(H(\omega_{m+1})-H(\omega_m)) & \text{if } k=2,\ldots,l-1\\
            \Delta(\omega_1,\omega_2) & \text{if } k=1
            \end{cases}
        \end{align}
     It can be shown that $H(\omega_{m+1})-H(\omega_m)=-\Delta(\omega_{m+1},\omega_m)=-2h<0$ for $m=2,\ldots,l$ \cite[Tab. 1]{cirillo2003metastability}, so the maximum is attained in the pair of configurations in $(\omega_1, \omega_2)$, hence
    \begin{align}\label{VA2}
     \max_{\omega_k,\omega_{k+1} \in \overline{\omega}}H(\omega_k,\omega_{k+1})-H(\omega_1) &
    =\Delta(\omega_1,\omega_2)=2(2-h):=V_{\sigma}^*.
    \end{align}
     Since $V^*_{\sigma}$ is the same for all configurations in $A_2$, $V^*_{A_2}=\max_{\sigma \in A_2} V^*_{\sigma}=2(2-h)$.
     Finally, let us check that $\omega_l \in \mathcal{I_{\sigma}} \cap (A_2 \cup \{\underline{c}\})$. Using \eqref{reversibility}, \eqref{VA2} and \cite[Tab. 1]{cirillo2003metastability}, we get
    \begin{align}
        H(\omega_1)+2(2-h)=H(\omega_l)+2h(l-1).
    \end{align}
     The rectangle $R_{l,m}$ is supercritical if and only if $l>2/h$, and so
    \begin{equation}
        H(\omega_1)-H(\omega_l)=2hl-4>0,
    \end{equation}
    which concludes the proof for $A_2$.
     \paragraph{Case $A_3$.} For any configuration $\sigma\in A_3$ we construct a path that begins in $\sigma$ and ends in a configuration in $A_3 \cup \{\underline{c}\}$ with lower energy than $\sigma$, i.e., $\overline{\omega} \in \Theta(\sigma,\mathcal{I_{\sigma}} \cap (A_3 \cup \{\underline{c}\}))$. We now fix $\sigma\equiv\omega_1\in A_3$ and we begin by defining $\omega_2$. If in $\sigma_{R_{l,m}}$ there is a plus corner surrounded by two minuses, say in $j_1$, then $\sigma(j_1)$ switches sign and the signs of all other spins in the rectangle remain fixed, i.e., $\omega_2:=T^{C}_{j_1}(\omega_1)$. On the other hand, if in $\sigma_{R_{l,m}}$ there are no plus corners surrounded by minuses, then we call the next configuration in the path $\omega_1'$ and we define it as $\omega_1':=T(\omega_1)$, i.e., all the spins in $\sigma_{\Lambda \setminus R_{l,m}}$ switch sign. After this step, $\omega_1'$ has a plus corner surrounded by two minuses, so we can proceed as above and define $\omega_2:=T^{C}_{j_1}(\omega_1')$. Note that in $\omega_2$ there are two plus corners in the rectangle that are nearest neighbors of $j_1$. For the next step, the plus corner, say in $j_2$, that is contained in a side of length $l$, switches sign, i.e., $\omega_3:=T^{C}_{j_2}(\omega_2)$. By iterating this step $l-2$ times, a full slice of the droplet is erased and we obtain the configuration $\eta \equiv \omega_l$ such that $\eta_{R_{l,m-1}}=\underline{+1}$ and $\eta_{\Lambda \setminus R_{l,m-1}}=\underline{c}$. In order to determine where the maximum of the transition energy is attained, we rewrite the energy difference as in \eqref{firstdifference}. Using \eqref{reversibility}, we obtain the same result as in \eqref{secdiff1}. Hence,
    \begin{align}\label{VA3}
     \max_{\omega_k,\omega_{k+1} \in \overline{\omega}}H(\omega_k,\omega_{k+1})-H(\omega_1) &
     =\sum_{m=1}^{l-2}(H(\omega_{m+1})-H(\omega_m))+\Delta(\omega_{l-1},\omega_l) \notag \\
    & =2h(l-2)+2h=2h(l-1):=V_{\sigma}^*.
    \end{align}
     Since $V^*_{\sigma}$ depends only on the length $l$, we find $V^*_{A_3}=\max_{\sigma \in A_3} V^*_{\sigma}$ by taking the maximum over $l$. Since $l<\lambda$, we have
    \begin{equation}
        V^*_{A_3} < 2 (2-h).
    \end{equation}
     Finally, let us check that $\omega_l \in \mathcal{I_{\sigma}} \cap (A_3 \cup \{\underline{c}\})$. Using \eqref{reversibility}, \eqref{VA3} and \cite[Tab. 1]{cirillo2003metastability}, we get
    \begin{align}
    H(\omega_1)+2h(l-1)=H(\omega_l)+2(2-h).
    \end{align}
     The rectangle $R_{l,m}$ is subcritical if and only if $l<2/h$, and so
    \begin{equation}
        H(\omega_1)-H(\omega_l)=4-2hl>0,
    \end{equation}
    which concludes the proof for $A_3$.
     \paragraph{Case $A_4$.} For any configuration $\sigma \in A_4$ we construct a path that begins in $\sigma$ and ends in a configuration in $A_4 \cup \{\underline{+1}\}$ with lower energy than $\sigma$, i.e., $\overline{\omega} \in \Theta(\sigma,\mathcal{I_{\sigma}} \cap (A_4 \cup \{\underline{+1}\}))$. We now fix $\sigma\equiv \omega_1\in A_4$ and we begin by defining $\omega_2$. Pick any site $j \in R_{l,m}$ in one of the sides of length $l$, such that its nearest neighbor $j_1\in \Lambda \setminus R_{l,m}$ is such that $\sigma(j_1)=+1$. We define $\omega_2:=T^{F}_{j_1}(\omega_1)$, i.e., $\sigma(j_1)$ is kept fixed and all the spins in $\sigma_{\Lambda \setminus R_{l,m}}$ switch sign. We define $\omega_3:=T(\omega_2)$, $\omega_4:=T(\omega_3)=T^2(\omega_2)$ and so on until a new slice is filled with $+1$. We obtain the configuration $\eta$ such that $\eta_{R_{l,m+1}}=\underline{+1}$ and $\eta_{\Lambda \setminus R_{l,m+1}}=\underline{c}$. Note that at the first step of the dynamics either one or two nearest neighbors of $j_1$ in the external side of the rectangle switch sign when $T$ is applied. Analogously, at each subsequent application of $T$, either one of two further sites in the external side of the rectangle switch sign. Therefore, the maximum number of iterations of the map $T$ is $l-1$. In order to determine where the maximum of the transition energy is attained, we rewrite the energy difference as in \eqref{firstdifference}. Using \eqref{reversibility}, we obtain the same result as in \eqref{secdiff2}. Hence,
    \begin{align}\label{VA4}
     \max_{\omega_k,\omega_{k+1} \in \overline{\omega}}H(\omega_k,\omega_{k+1})-H(\omega_1)=\Delta(\omega_1,\omega_2)=2(2-h):=V_{\sigma}^*.
    \end{align}
     Since $V^*_{\sigma}$ is the same for all configurations in $A_4$, $V^*_{A_4}=\max_{\sigma \in A_4} V^*_{\sigma}=2(2-h)$.
     Finally, let us check that $\omega_l \in \mathcal{I_{\sigma}} \cap (A_4 \cup \{\underline{+1}\})$. Using \eqref{reversibility}, \eqref{VA4} and \cite[Tab. 1]{cirillo2003metastability}, we get
    \begin{align}
        H(\omega_1)+2(2-h)=H(\omega_l)+2h(l-1).
    \end{align}
     The rectangle $R_{l,m}$ is supercritical if and only if $l>2/h$, and so
    \begin{equation}
        H(\omega_1)-H(\omega_l)=2hl-4>0,
    \end{equation}
    which concludes the proof for $A_4$.
     \paragraph{Case $A_5$.} For any configuration $\sigma\in A_5$ we construct a path that begins in $\sigma$ and ends in a configuration in $D_1$ with lower energy than $\sigma$, i.e., $\overline{\omega}\in \Theta(\sigma,\mathcal{I_{\sigma}} \cap D_1)$. We now fix $\sigma\equiv \omega_1\in A_5$ and we begin by defining $\omega_2$. We call $j_1$ a corner in $R_{l,m}$ such that (necessarily) $\sigma(j_1)=+1$ and we define $\omega_2:=T^{C}_{j_1}(\omega_1)$, i.e., $\sigma(j_1)$ switches sign and the signs of all other spins in the rectangle remain fixed. Note that in $\omega_2$ there are two plus corners in the rectangle that are nearest neighbors of $j_1$. For the next step, the plus corner, say in $j_1$, that is contained in a side of length $l$ switches sign, i.e., $\omega_3:=T^{C}_{j_2}(\omega_2)$. After this, the spin of the nearest neighbor of $j_2$ along the same side of $R_{l,m}$ and different from $j_1$, say in $j_3$, switches spin, i.e., $\omega_4:=T^{C}_{j_3}(\omega_3)$. By iterating this step $l-3$ times, a full slice of the droplet is erased and we obtain the configuration $\omega_l \equiv \eta$ such that $\eta_{R_{l,m-1}}=\underline{+1}$, $\eta_{R_{l,1}}=\underline{c}$, $\eta_{\Lambda \setminus R_{l,m}}=\underline{-1}$. The configuration $\eta$ is a configuration in $D_1$. In order to determine where the maximum of the transition energy is attained, we rewrite the energy difference as in \eqref{firstdifference}. Using \eqref{reversibility}, we obtain the same result \eqref{secdiff1}. Hence,
    \begin{align}\label{VA5}
     \max_{\omega_k,\omega_{k+1} \in \overline{\omega}}H(\omega_k,\omega_{k+1})-H(\omega_1) &
     =\sum_{m=1}^{l-2}(H(\omega_{m+1})-H(\omega_m))+\Delta(\omega_{l-1},\omega_l) \notag \\
    & =2h(l-2)+2h=2h(l-1):=V_{\sigma}^*.
    \end{align}
     Since $V^*_{\sigma}$ depends only on the length $l$, we find $V^*_{A_5}=\max_{\sigma \in A_5} V^*_{\sigma}$ by taking the maximum over $l$. Since $l<\lambda$, we have
    \begin{equation}
        V^*_{A_5} < 2 (2-h).
    \end{equation}
     Finally, let us check that $\omega_l \in \mathcal{I_{\sigma}} \cap D_1$. Using \eqref{reversibility}, \eqref{VA5} and \cite[Tab. 1]{cirillo2003metastability}, we get
    \begin{align}
        H(\omega_1)+2h(l-1)=H(\omega_l)+2(2-h).
    \end{align}
     The rectangle $R_{l,m}$ is subcritical if and only if $l<2/h$, and so
    \begin{equation}
        H(\omega_1)-H(\omega_l)=4-2hl>0,
    \end{equation}
    which concludes the proof for $A_5$.
     \paragraph{Case $A_6$.} For any configuration $\sigma \in A_6$ we construct a path that begins in $\sigma$ and ends in a configuration in $D_3$ with lower energy than $\sigma$, i.e., $\overline{\omega}\in \Theta(\sigma,\mathcal{I_{\sigma}} \cap D_3)$. We now fix $\sigma\equiv\omega_1\in A_6$ and we begin by defining $\omega_2$. We call $j \in R_{l,m}$ a site in a side of $R_{l,m}$, and note that (necessarily) $\sigma(j)=+1$. Without loss of generality, we choose a side of length $l$. Furthermore, we call $j_1\in \Lambda \setminus R_{l,m}$ the nearest neighbor of $j$ contained in the external side with length $l$ such that (necessarily) $\sigma(j_1)=-1$. We define $\omega_2:=T^{C}_{j_1}(\omega_1)$, i.e., $\sigma(j_1)$ switches sign and the signs of all other spins in $\sigma_{\Lambda \setminus R_{l,m}}$ remain fixed. We define $\omega_3:=T(\omega_2)$, $\omega_4:=T(\omega_3)=T(\omega_2)$ and so on until a new slice is filled with $\underline{c}$, so we obtain the configuration $\eta$ such that $\eta_{R_{l,m}}=\underline{+1}$, $\eta_{R_{l,1}}=\underline{c}$ and $\eta_{\Lambda \setminus R_{l,m+1}}=\underline{-1}$. Note that at the first step of the dynamics either one or two nearest neighbors of $j_1$ in the external side of the rectangle switch sign when $T$ is applied. Analogously, at each subsequent application of $T$, either one of two further sites in the external side of the rectangle switch sign. Therefore, the maximum number of iterations of the map $T$ is $l-1$. The configuration $\eta$ is a configuration in $D_2$. In order to determine where the maximum of the transition energy is attained, we rewrite the energy difference as in \eqref{firstdifference}. Using \eqref{reversibility}, we obtain the same result as in \eqref{secdiff2}. Hence,
    \begin{align}\label{VA6}
     \max_{\omega_k,\omega_{k+1} \in \overline{\omega}}H(\omega_k,\omega_{k+1})-H(\omega_1)=\Delta(\omega_1,\omega_2)=2(2-h):=V_{\sigma}^*.
    \end{align}
     Since $V^*_{\sigma}$ is the same for all configurations in $A_6$, $V^*_{A_6}=\max_{\sigma \in A_6} V^*_{\sigma}=2(2-h)$.
     Finally, let us check that $\omega_l \in \mathcal{I_{\sigma}} \cap D_3$. Using \eqref{reversibility}, \eqref{VA6} and \cite[Tab. 1]{cirillo2003metastability}, we get
    \begin{align}
        H(\omega_1)+2(2-h)=H(\omega_l)+2h(l-1).
    \end{align}
     The rectangle $R_{l,m}$ is supercritical if and only if $l>2/h$, and so
    \begin{equation}
        H(\omega_1)-H(\omega_l)=2hl-4>0,
    \end{equation}
    which concludes the proof for $A_6$. In conclusion,
    \begin{equation}
    V^*_A:=\max_{i=1,...,6}{V^*_{A_i}}=2 (2-h)
    \end{equation}

    Next we consider the set $B$.
     \paragraph{Case $B_1$.} For every configuration in $B_1$, both rectangles are subcritical. Following a path that changes a slice of $\underline{+1}$ into a slice of $\underline{c}$, analogously as was done for $A_3$, we get a configuration in $\mathcal{I}_\sigma \cap (B_1 \cup A_1)$. We have
    \begin{equation}
        V^*_{B_1}=V^*_{A_3}<2(2-h).
    \end{equation}
     \paragraph{Case $B_2$.} For every configuration in $B_2$, both rectangles are supercritical. Following a path that adds a slice of $\underline{c}$, analogously as was done for $A_2$, we get a configuration in $\mathcal{I}_\sigma \cap (B_2 \cup A_4)$. We have
    \begin{equation}
        V^*_{B_2}=V^*_{A_2}=2(2-h).
    \end{equation}
     \paragraph{Case $B_3$.} For every configuration in $B_3$, the external  rectangle is supercritical and the internal rectangle is subcritical. Following a path that adds a slice of $\underline{c}$, analogously as was done for $A_2$, we get a configuration in $\mathcal{I}_\sigma \cap (B_3 \cup A_3)$. We have
    \begin{equation}
        V^*_{B_3}=V^*_{A_2}=2(2-h).
    \end{equation}
    We conclude that
    \begin{equation*}
    V^*_B=\max\{V^*_{B_1},V^*_{B_2},V^*_{B_3}\}=V^*_A.
    \end{equation*}

    Next we consider the set $D$.
     \paragraph{Case $D_1$.} For every configuration $\sigma$ in $D_1$, all rectangles are subcritical and non-interacting. If $\sigma$ contains at least one rectangle of $\underline{+1}$ surrounded by $\underline{c}$, we take our path to be the path that cuts a slice of $\underline{+1}$, analogously as was done for $A_3$. We get a configuration in $\mathcal{I}_\sigma \cap D_1$. Otherwise, if $\sigma$ contains at least one rectangle of $\underline{+1}$ surrounded by $\underline{-1}$, we take our path to be the path that changes a slice of $\underline{+1}$ into a slice of $\underline{c}$, analogously as was done for $A_5$. We get a configuration in $\mathcal{I}_\sigma \cap D_3$. Finally, we consider all remaining configurations, namely chessboard rectangles in a sea of minus. We take our path to be the path that cuts a slice of $\underline{c}$, analogous to the one described in $A_1$. We get a configuration in $\mathcal{I}_\sigma \cap (D_1 \cup A_1)$. So, we have
    \begin{equation}
        V^*_{D_1}=\max \{ V^*_{A_1}, V^*_{A_3}, V^*_{A_5} \} <2(2-h).
    \end{equation}
         \paragraph{Case $D_2$.} For every configuration $\sigma$ in $D_2$, there exists at least one supercritical rectangle. If this is a chessboard rectangle, then we take the path that makes the rectangle grow a slice of $\underline{c}$, analogously as was done for $A_2$.  We get a configuration in $\mathcal{I}_\sigma \cap (A_3 \cup A_4 \cup D_2 \cup D_4 \cup D_5 \cup E_4 \cup \{\underline{c}\})$. Otherwise, if this supercritical rectangle contains $\underline{+1}$, we take the path that makes the rectangle grow a slice of $\underline{c}$, analogously as was done for $A_6$. We get a configuration in $\mathcal{I}_\sigma \cap (D_2 \cup D_4 \cup D_5)$. So, we have
    \begin{equation}
        V^*_{D_2}=\max \{ V^*_{A_2}, V^*_{A_6} \} =2(2-h).
    \end{equation}
     \paragraph{Case $D_3$.} For every configuration $\sigma$ in $D_3$, all rectangles are subcritical and non-interacting. If $\sigma$ contains at least one rectangle of $\underline{+1}$ surrounded by $\underline{c}$, we take our path to be the path that cuts a slice of $\underline{+1}$, analogously as was done for $A_3$. We get a configuration in $\mathcal{I}_\sigma \cap D_3$. Otherwise, if $\sigma$ contains at least one rectangle of $\underline{+1}$ at lattice distance one from a rectangle of $\underline{c}$, we take the path that changes a slice of $\underline{+1}$ into a slice of $\underline{c}$ along the interface between the two rectangles, analogously as was done for $A_3$. We get a configuration in $\mathcal{I}_\sigma \cap (A_1 \cup D_1 \cup D_3)$. In the remaining cases, $\sigma$ contains at least two rectangles of different chessboard parity at lattice distance one. We take our path to be a path that changes a slice of $\underline{c}$, analogously as was done for $A_1$. We get a configuration in $\mathcal{I}_\sigma \cap (A_1 \cup D_1  \cup D_3)$. So, we have
    \begin{equation}
        V^*_{D_3}=\max \{ V^*_{A_1}, V^*_{A_3} \} <2(2-h).
    \end{equation}
     \paragraph{Case $D_4$.} For every configuration $\sigma$ in $D_4$, all rectangles of $\underline{+1}$ surrounded by $\underline{c}$ are subcritical and non-interacting. We take our path to be a path that cuts a slice of $\underline{+1}$, analogously as was done for $A_3$. We get a configuration in $\mathcal{I}_\sigma \cap (D_4 \cup A_3)$. So, we have
    \begin{equation}
        V^*_{D_4}=V^*_{A_3}<2(2-h).
    \end{equation}
     \paragraph{Case $D_5$.} For every configuration $\sigma$ in $D_5$, there exists at least a supercritical rectangle of $\underline{+1}$ surrounded $\underline{c}$. We consider this rectangle and we take the path that makes the rectangle grow a slice of $\underline{+1}$, analogously as was done for $A_4$. We get a configuration in $\mathcal{I}_\sigma \cap (D_5 \cup A_4 \cup E_5)$. So, we have
    \begin{equation}
        V^*_{D_5}=V^*_{A_4}=2(2-h).
    \end{equation}
    In conclusion,
    \begin{equation*}
     V^*_D=\max\{V^*_{D_1},V^*_{D_2},V^*_{D_3},V^*_{D_4},V^*_{D_5}\}=V^*_A.
    \end{equation*}
    %
    %
    %

The last set $E$ is composed of strips.

 \paragraph{Case $E_1$.} A configuration $\sigma \equiv \omega_1$ in $E_1$ has at least a strip of $\underline{c}$ of width one. Pick a site $j$ in the strip such that $\sigma(j)=-1$ and define $\omega_2=T_j^F(\omega_1)$, i.e., $\sigma(j)$ is kept fixed. The energy difference is $H(\omega_2)-H(\omega_1)=2h$ \cite[Tab.1]{cirillo2003metastability}. We define $\omega_3:=T(\omega_2)$, $\omega_4:=T(\omega_3)=T^2(\omega_2)$ and so on until we obtain a configuration in $\mathcal{I}_\sigma \cap (E_1 \cup D_1 \cup D_2 \cup D_3 \cup A_1 \cup A_2 \cup A_5 \cup A_6 \cup B \cup \{\underline{-1}\})$. So, we have
    \begin{equation}
        V^*_{E_1}=2h.
    \end{equation}
 \paragraph{Case $E_2$.} A configuration $\sigma \equiv \omega_1$ in $E_2$ contains at least a strip of $\underline{+1}$ of width one. Let $\sigma(j)$ be a plus in the strip surrounded by one or two minuses. We define $\omega_2=T_j^C(\omega_1)$, i.e., $\sigma(j)$ switches sign. The maximum energy difference is $H(\omega_2)-H(\omega_1)=2(2-h)$ \cite[Tab.1]{cirillo2003metastability}. We define $\omega_3:=T(\omega_2)$, $\omega_4:=T(\omega_3)=T^2(\omega_2)$ and so on until we obtain a configuration in $\mathcal{I}_{\sigma} \cap (E_2 \cup E_7 \cup \{\underline{c}\})$. So, we have
    \begin{equation}
        V^*_{E_2}=\max\{V^*_{E_7},2(2-h)\}=2(2-h).
    \end{equation}
 \paragraph{Case $E_3$.} A configuration $\sigma \equiv \omega_1$ in $E_3$ has at least a strip of $\underline{+1}$ of width one. If in $\sigma$ there is a strip of $\underline{+1}$ surrounded by two chessboards with the same parity, then pick a plus $\sigma(j)$ in the strip and define $\omega_2=T_j^C(\omega_1)$, i.e., $\sigma(j)$ switches sign. The energy difference is $H(\omega_2)-H(\omega_1)=2h$ \cite[Tab.1]{cirillo2003metastability}. We define $\omega_3:=T(\omega_2)$, $\omega_4:=T(\omega_3)=T^2(\omega_2)$ and so on until we obtain a configuration in $\mathcal{I}_\sigma \cap (E_1 \cup E_7)$. Instead, if in $\sigma$ there is a strip of $\underline{+1}$ surrounded by two chessboards with different parity, then pick a plus $\sigma(j)$ in a chessboard at lattice distance one from the strip and define $\omega_2=T_j^F(\omega_1)$, i.e., $\sigma(j)$ is kept fixed. The energy difference is $H(\omega_2)-H(\omega_1)=2(2-h)$ \cite[Tab.1]{cirillo2003metastability}. We define $\omega_3:=T(\omega_2)$, $\omega_4:=T(\omega_3)=T^2(\omega_2)$ and so on until we obtain a configuration in $\mathcal{I}_\sigma \cap E_5$. So, we have
    \begin{equation}
        V^*_{E_3}=\max \{2h,2(2-h)\}=2(2-h).
    \end{equation}
 \paragraph{Case $E_4$.} We consider a configuration $\sigma \equiv \omega_1$ in $E_4$ and pick a plus on the interface between $\underline{c}$ and $\underline{-1}$, and call $j$ the site of this plus. We call $j_1$ the nearest neighbor of $j$ in $\underline{-1}$ and we define $\omega_2=T_{j_1}^C(\omega_1)$, i.e., $\sigma(j_1)$ switches sign. The energy difference is $H(\omega_2)-H(\omega_1)=2(2-h)$ \cite[Tab.1]{cirillo2003metastability}. We define $\omega_3:=T(\omega_2)$, $\omega_4:=T(\omega_3)=T^2(\omega_2)$ and so on until we obtain a configuration in $\mathcal{I}_\sigma \cap (E_4 \cup D_4 \cup D_5 \cup E_7 \cup \{\underline{c}\})$. So, we have
    \begin{equation}
        V^*_{E_4}=2(2-h).
    \end{equation}
 \paragraph{Case $E_5$.} We consider a configuration $\sigma \equiv \omega_1$ in $E_5$ and pick a plus in $\underline{c}$ on the interface between $\underline{c}$ and $\underline{+1}$, and call $j$ the site of this plus. We define $\omega_2=T_{j}^F(\omega_1)$, i.e., $\sigma(j)$ is kept fixed. The energy difference is $H(\omega_2)-H(\omega_1)=2(2-h)$ \cite[Tab.1]{cirillo2003metastability}. We define $\omega_3:=T(\omega_2)$, $\omega_4:=T(\omega_3)=T^2(\omega_2)$ and so on until we obtain a configuration in $\mathcal{I}_\sigma \cap (E_5 \cup \{\underline{+1}\})$. So, we have
    \begin{equation}
        V^*_{E_5}=2(2-h).
    \end{equation}
 \paragraph{Case $E_6$.} We consider a configuration $\sigma \equiv \omega_1$ in $E_6$ and pick a minus on the interface between $\underline{-1}$ and $\underline{+1}$, and call $j$ the site of this minus. We define $\omega_2=T_{j}^C(\omega_1)$, i.e., $\sigma(j)$ switches sign. The energy difference is $H(\omega_2)-H(\omega_1)=2(2-h)$ \cite[Tab.1]{cirillo2003metastability}. We define $\omega_3:=T(\omega_2)$, $\omega_4:=T(\omega_3)=T^2(\omega_2)$ and so on until we obtain a configuration in $\mathcal{I}_\sigma \cap E_7$. So, we have
    \begin{equation}
        V^*_{E_6}=2(2-h).
    \end{equation}
 \paragraph{Case $E_7$.} If the configuration $\sigma \equiv \omega_1$ in $E_7$ contains a strip of $\underline{-1}$ adjacent to a strip of $\underline{+1}$ and both have width greater then one, then we pick a minus one on the interface between $\underline{-1}$ and $\underline{+1}$ and we take a path analogously as was done for $E_6$. We get a configuration in $\mathcal{I}_\sigma \cap (E_7 \cup E_5)$. Otherwise, $E_7$ contains a strip of $\underline{c}$ adjacent to a strip of $\underline{-1}$, both with width greater then one. Then, we pick a plus one, say in $j$, in the strip of $\underline{c}$. We call $j_1$ the nearest neighbor of $j$ in $\underline{-1}$ and we define $\omega_2=T_{j_1}^C(\omega_1)$, i.e., $\sigma(j_1)$ switches sign. The energy difference is $H(\omega_2)-H(\omega_1)=2(2-h)$ \cite[Tab.1]{cirillo2003metastability}. We define $\omega_3:=T(\omega_2)$, $\omega_4:=T(\omega_3)=T^2(\omega_2)$ and so on until we obtain a configuration in $\mathcal{I}_\sigma \cap (E_7 \cup E_5)$. So, we have
    \begin{equation}
        V^*_{E_7}=\max\{V^*_{E_6},2(2-h)\}=2(2-h).
    \end{equation}
     Then
     \begin{equation*}
      V^*_E=\max\{V^*_{E_1},V^*_{E_2},V^*_{E_3},V^*_{E_4},V^*_{E_5},V^*_{E_6},V^*_{E_7}\}=V^*_A.
     \end{equation*}
     To conclude the proof, we compare the value of $V^*=\max\{V^*_A,V^*_B,V^*_D,V^*_E\}=2(2-h)$ and $\activationenergy$, and we get
    \begin{equation}
    \activationenergy\equiv -2h\lambda^2+2\lambda(4+h)-2h>2(2-h)=V^*.
    \end{equation}

\end{proof}

\appendix
\section{Appendix}\label{appendixA}
 In this Appendix, we recall some results and give explicit computation that are used in the paper. Equation \eqref{Delta} is obtained as follows,
\begin{align*}
      -\lim_{\beta \rightarrow \infty } \frac{\log p(\sigma,\eta)}{\beta} & =-\lim_{\beta \rightarrow \infty } \frac{\log \Big( \prod_{i \in \Lambda} p_{i,\sigma}(\eta(i)) \Big)}{\beta}= \\
     & =-\lim_{\beta \rightarrow \infty } \frac{\sum_{i \in \Lambda} \log \Big( \frac{1}{1+\exp\{-2\beta \eta (i) |S_{\sigma}(i)+h|\}}\Big)}{\beta}= \\
     & =\lim_{\beta \rightarrow \infty } \frac{\sum_{i \in \Lambda} \log ({1+\exp\{-2\beta \eta (i) |S_{\sigma}(i)+h| \}})}{\beta},
\end{align*}
where we used \eqref{eqmark}, \eqref{eqprob} and logarithm properties,
\begin{align*}
     &\lim_{\beta \rightarrow \infty } \frac{\sum_{i \in \Lambda: \eta(i)(S_{\sigma}(i)+h)>0} \log ({1+\exp\{-2\beta \eta (i) (S_{\sigma}(i)+h)\}})}{\beta}+ \\
     & +\lim_{\beta \rightarrow \infty } \frac{\sum_{i \in \Lambda:\eta(i)(S_{\sigma}(i)+h)<0} \log ({1+\exp\{2\beta \eta (i) (S_{\sigma}(i)+h)\}})}{\beta} = \\
     & = \sum_{i \in \Lambda:\eta(i)(S_{\sigma}(i)+h)<0} \lim_{\beta \rightarrow \infty } \frac{\log ({1+\exp\{2\beta |S_{\sigma}(i)+h|\}})}{\beta}= \\
     & =\sum_{i \in \Lambda:\eta(i)(S_{\sigma}(i)+h)<0} 2|S_{\sigma}(i)+h|.
\end{align*}

 \begin{defn}\cite[Energy landscape Definition 2.1]{cirillo2013relaxation}
 An energy landscape is a quadruplet $(\mathcal{X},Q,H,\Delta)$ where the finite non-empty sets $\mathcal{X}, Q \subset \mathcal{X} \times \mathcal{X}$, and the maps $H : \mathcal{X} \rightarrow \mathbb{R}$, $\Delta : Q \rightarrow \mathbb{R}^+$ are called respectively state space, connectivity relation, energy, and energy cost, and for any $\sigma,\eta\in \mathcal{X}$ there exists an integer $n \ge 2$ and $\omega_1,...,\omega_n \in \mathcal{X}$ such that $\omega_1=\sigma, \omega_n=\eta$, and $(\omega_i, \omega_{i+1})\in Q$ for any $i=1,...,n-1$.
 An energy landscape $(\mathcal{X},Q,H,\Delta)$ is called reversible if and only if $Q$ is symmetric, that is if $(\sigma,\eta)\in Q$ then $(\eta,\sigma)\in Q$, and $H(\sigma)+\Delta(\sigma,\eta)=\Delta(\eta,\sigma)+H(\eta)$ for all $(\sigma,\eta)\in Q$.
\end{defn}

\begin{thm}\cite[Theorem 2.4]{cirillo2013relaxation}
 Consider a reversible energy landscape $(\mathcal{X},Q,H,\Delta)$. Let $\mathcal{X}^s$ be the set of stable states and assume $\mathcal{X}\setminus \mathcal{X}^s \ne \emptyset$. If there exist $A\subset \mathcal{X} \setminus \mathcal{X}^s$ and $a\in \mathbb{R^+}$ such that
\begin{enumerate}
     \item $\Phi(\sigma, \mathcal{X}^s)-H(\sigma)=a$ for all $\sigma\in A$;
     \item either $\mathcal{X}\setminus (A \cup \mathcal{X}^s)=\emptyset$ or $V_\sigma<a$ for all $\sigma\in \mathcal{X}\setminus (A \cup \mathcal{X}^s)$;
\end{enumerate}
then $\activationenergym=a$ and $\mathcal{X}^m=A$.
\end{thm}

\begin{teo}\cite[Proposition 3.1]{cirillo2003metastability}
A configuration $\sigma\in\mathcal{S}_{\underline{-1}}$ is stable
for PCA if and only if $\sigma(x)=+1$ for all sites $x$ inside a
collection of pairwise non-interacting rectangles of minimal side
length $l \ge 2$ and $\sigma(x)=-1$ elsewhere. A configuration
$\sigma\in\mathcal{S}_{\underline{+1}}$ is stable if and only
if $\sigma=\underline{+1}$. There is no stable configuration
$\sigma\in\mathcal{S}_{\underline{c}}$.
\end{teo}

\begin{teo}\cite[Proposition 3.3]{cirillo2003metastability}
\begin{itemize}
     \item[i)] For any $\sigma\in\mathcal{S}_{\underline{+1}} \setminus \{\underline{+1}\}$, the pair $(\sigma,T\sigma)$ is not a stable pair.
     \item[ii)] Given $C\in\{\underline{c}^o,\underline{c}^e\}$ and $\sigma\in\mathcal{S}_{\underline{c}}$ the pair $(\sigma,T\sigma)$ is a stable pair if and only if there exist $k\ge 0$ pairwise non-interacting rectangles $R_{l_{1},m_{1}},\, R_{l_{2},m_{2}},...,R_{l_{k},m_{k}}$ such that $2\le l_{i} \le m_{i} \le L-2$ for any $i=1,...,k$, $\sigma_{\mathcal{R}}=\underline{+1}_\mathcal{R}$ ($\sigma$ coincides with \underline{+1} inside the rectangles) and $\sigma_{\Lambda\setminus \mathcal{R}}=\underline{c}_{\Lambda\setminus \mathcal{R}}$ ($\sigma$ coincides with the chessboard $C$ outside the rectangles), where $\mathcal{R}=\bigcup_{i=1}^k \overline{R}_{l_i,m_i}$.
     \item[iii)] Given $\sigma\in\mathcal{S}_{\underline{-1}}$ the pair $(\sigma,T\sigma)$ is a stable pair if and only if there exist $k\ge1$ rectangles $R_{l_{1},m_{1}},\, R_{l_{2},m_{2}},...,R_{l_{k},m_{k}}$ with $2\le l_{i} \le m_{i} \le L-2$ for any $i=1,...,k$, and there exists an integer $s\in\{1,...,k\}$ such that the following conditions are fulfilled:
    \begin{itemize}
         \item[1.] $\overline{R}_{l_{i},m_{i}} \cap \overline{R}_{l_{j},m_{j}}=\emptyset$ and $l_i \ge 2$ for any $i,j \in \{1,...,k\}$;
         \item[2.] or any $j\in\{1,...,s\}$ the family $R_{l_{j},m_{j}},\, R_{l_{s+1},m_{s+1}}, ..., R_{l_{k},m_{k}}$ is a family of pairwise non-interacting rectangles;
         \item[3.] $\sigma_{\Lambda \setminus \mathcal{R}}= \underline{-1}_{\Lambda \setminus \mathcal{R}}$ ($\sigma$ coincides with
$-1$ outside the rectangles);
        \item[4.] for any $j\in\{s+1,...,k\}$
        \item[5.]
        \begin{itemize}
             \item[5.1.] $\overline{R}'_{l'_{i},m'_{i}} \subset \overline{R}_{l_{j},m_{j}}$ for any $i\in\{1,..,k'\}$;
             \item[5.2.] for any $j=1,...,s$ the family $\{R'_{l'_i,m'_i}: i=1,...,k'\}$ (recall $R'_{l'_i,m'_i}=R'_{l'_i,m'_i}(j)$ for any $i=1,...,k'=k'(j)$ is a family of pairwise non-interacting rectangles;
             \item[5.3.] $\sigma_{\mathcal{R'}}= \underline{+1}_{\mathcal{R'}}$ where $\mathcal{R'}\equiv \mathcal{R'}(j):=\bigcap^k_{i=1} \overline{R}'_{l'_i,m'_i}$;
             \item[5.4.] either $\sigma_{\overline{R}_{l_j,m_j}\setminus\mathcal{R}'=C^o_{l_j,m_j}\setminus\mathcal{R}'}$ or $\sigma_{\overline{R}_{l_j,m_j}\setminus\mathcal{R}'=C^e_{l_j,m_j}\setminus\mathcal{R}'}$;
        \end{itemize}
         \item[6.] or any $i,j\in\{1,...,s\}$ the two rectangles $R_{l_{j},m_{j}}$ $R_{l_{i},m_{i}}$ must be non-interacting if $\sigma_{R_{l_{j},m_{j}} \setminus \mathcal{\overline{R}}'(j)}=\sigma_{R_{l_{j},m_{j}} \setminus \mathcal{\overline{R}}'(i)}$
    \end{itemize}
\end{itemize}
\end{teo}

\section{Appendix}\label{AppendixB}
In this section we prove theorems given in Section~\ref{s:series}.

\medskip
\par\noindent
\textit{Proof of Theorem~\ref{t:ptaset}.\/}
 Recall the equivalence relation above Theorem 3.6 in \cite{cirillo2013relaxation} for $x,y \in \mathcal{X}$
\begin{equation}
 x \sim y \text{ if and only if } \Phi(x,y)-H(x)<\activationenergym \text{ and } \Phi(y,x)-H(y)<\activationenergym.
\end{equation}
 The configurations $\metauno^1,...,\metauno^n$ are in the same equivalence class. Thus, the theorem follows immediately by Condition~\ref{t:series00}, \eqref{series01new}, and \cite[Theorem~3.6]{cirillo2013relaxation}. \\
\qed

 Before given the proof of Theorem~\ref{t:meant}, we state two useful lemmas. In the first of the two lemmas we collect two bounds on the energy cost to go from any state $x\neq \metauno^r$ to $\metauno^r$ or to $\stab$, for $r=1,...,n$. The second lemma is similar.

\begin{lem}
\label{stimePhi}
 Assume Condition~\ref{t:series00} is satisfied. For any $x\in \mathcal{X}$ and $x\neq\metauno^r$, for every $r=1,...,n$. If $H(x)\leq H(\metauno^r)$, we have that
\begin{equation}
\label{claim1}
\Phi(x,\stab)-H(x)
<\Gamma_m
\;\;\textrm{ and }\;\;
\Phi(x,\metauno^r)
-H(\metauno^r)
\ge \Gamma_m, \;\;\textrm{ for every }\;\; r=1,...,n.
\end{equation}
\end{lem}
\medskip
\par\noindent
\textit{Proof.\/}
 Let us prove the first inequality. By Theorem~2.3  in \cite{cirillo2013relaxation} we have that $\Phi(x,\stab)\leq\Gamma_m+H(x)$. If by contradiction $\Phi(x,\stab)=\Gamma_m+H(x)$ then, by the same Theorem~2.3 in \cite{cirillo2013relaxation}, $x\in \mathcal{X}^m$ which is in contradiction with Condition~\ref{t:series00}.
 Next we turn to the proof of the second inequality and we distinguish two cases. If $H(x)<H(\metauno^r)$, then we have that $x\in \mathcal{I}_{\metauno^r}$. By \eqref{detailedbalance} and by \eqref{stablev}, we get
 $$\Phi(\metauno^r,x)\geq \Phi(\metauno^r,\mathcal{I}_{\metauno^r})= \Gamma_m+H(\metauno^r)$$ that proves the inequality. If $H(x)=H(\metauno^r)$, then let us define the set
\begin{displaymath}
 \mathcal{C}:=\{y\in \mathcal{X}: \Phi(y,\metauno^r)<H(\metauno^r)+\Gamma_m\}.
\end{displaymath}
 We will show that $x\not\in \mathcal{C}$. Since $H(x)= H(\metauno^r)$, the identity $\mathcal{I}_{x}=\mathcal{I}_{\metauno^r}$ follows. Furthermore, since $\metauno^r\in \mathcal{X}^m$, we have $\mathcal{C}\cap \mathcal{I}_{\metauno^r}=\emptyset$; hence, $\mathcal{C}\cap \mathcal{I}_{x}=\emptyset$ as well. Moreover, if $x\in\mathcal{C}$ then $V_x=\Phi(x,\mathcal{I}_x)-H(x)\geq H(\metauno^r)+\Gamma_m-H(x)=\Gamma_m$. By \eqref{detailedbalance}, $x$ would be a metastable state, in contradiction with Condition~\ref{t:series00}. Hence, since $x\not\in \mathcal{C}$, we have that
$$
\Phi(x,\metauno^r)\geq \Gamma_m+H(\metauno^r).
$$
This proves the inequality for every $r=1,...,n$.
\qed
\begin{lem}
\label{stimePhi2}
 Assume Condition~\ref{t:series00} is satisfied. For any $x\in \mathcal{X}$ and
$x\notin\{\metadue,\metauno^1,...,\metauno^n,\stab\}$.
If $H(x)\leq H(\metadue)$, then
\begin{equation}
\label{claim3}
\Phi(x,\{\metauno^1,...,\metauno^n,\stab\})-H(x)
<\Gamma_m
\;\;\textrm{ and }\;\;
\Phi(x,\metadue)
-H(\metadue)
\ge \Gamma_m
\end{equation}
\end{lem}
\medskip
\par\noindent
\textit{Proof.\/}
 Let us prove the first inequality. By Theorem~2.3 in \cite{cirillo2013relaxation} we have $\Phi(x,\{x_1^1,...,x_1^n,\stab\})\leq \Phi(x,\stab)\leq\Gamma_m+H(x)$. We proceed by contradiction and assume that $\Phi(x,\stab)=\Gamma_m+H(x)$. By \cite[Theorem 2.3]{cirillo2013relaxation}, $x\in \mathcal{X}^m$ which is in contradiction with Condition~\ref{t:series00}. Next we turn to the proof of the second inequality we distinguish two cases. If $H(x)<H(\metadue)$, then we have that $x\in \mathcal{I}_{\metadue}$. By \eqref{detailedbalance} of metastable state and by \eqref{stablev}, we get $$\Phi(\metadue,x)\geq \Phi(\metadue,\mathcal{I}_{\metadue})= \Gamma_m+H(\metadue).$$ This proves the inequality. If $H(x)=H(\metadue)$, then let us define the set
\begin{displaymath}
 \mathcal{C}:=\{y\in \mathcal{X}: \Phi(y,\metadue)<H(\metadue)+\Gamma_m\}.
\end{displaymath}
 We will show that $x\not\in \mathcal{C}$. Since $H(x)= H(\metadue)$, the identity $\mathcal{I}_{x}=\mathcal{I}_{\metadue}$ follows. Furthermore, since $\metadue\in X_\rr{m}$, we have $\mathcal{C}\cap \mathcal{I}_{\metadue}=\emptyset$; hence, $\mathcal{C}\cap \mathcal{I}_{x}=\emptyset$ as well. Moreover, if $x\in\mathcal{C}$ then $V_x=\Phi(x,\mathcal{I}_x)-H(x)\geq H(\metadue)+\Gamma_m-H(x)=\Gamma$. By \eqref{detailedbalance}, $x$ would be a metastable state, in contradiction with Condition~\ref{t:series00}. Hence, since $x\not\in \mathcal{C}$, we have that $$\Phi(x,\metadue)\geq \Gamma_m+H(\metadue).$$ This proves the inequality.
\qed

\medskip
\par\noindent
\textit{Proof of Theorem~\ref{t:meant}.\/}
We begin by proving Equation \eqref{Egen1}.
 The proof is based on Lemma~\ref{stimePhi} and Lemma~\ref{stimePhi2}. In the proof we only use the representation of the expected mean time in terms of the Green function \cite[Corollary~3.3]{bovier2004metastability}, see also \cite[Eq. (4.29)]{gaudilliere2009condenser}. Indeed, recalling \eqref{cap-prop} above, we rewrite the expected value in terms of the capacity as
\begin{equation}
\label{valatt1}
\mathbb{E}_{\metadue}[\tau_{\{\metauno^1,...,\metauno^n,\stab\}}]
= \frac{1}{\rr{cap}_\beta(\metadue,\{\metauno^1,...,\metauno^n,\stab\})}
\sum_{x\in \mathcal{X}}
\mu_\beta(x) \, h_{\metadue,\{\metauno^1,...,\metauno^n,\stab\}}(x).
\end{equation}
 Since $h_{\metadue,\{\metauno^1,...,\metauno^n,\stab\}}(\metadue)=1$, we get the following lower bound:
\begin{align}
\label{downboundval}
\mathbb{E}_{\metadue}[\tau_{\{\metauno^1,...,\metauno^n,\stab\}}]
 & \ge \frac{1}{\rr{cap}(\metadue,\{\metauno^1,...,\metauno^n,\stab\})} \mu_\beta(\metadue) h_{\metadue,\{\metauno^1,...,\metauno^n,\stab\}}(\metadue) \notag \\
 & = \frac{\mu_\beta(\metadue)}{\rr{cap}(\metadue,\{\metauno^1,...,\metauno^n,\stab\})}.
\end{align}
 In order to give an upper bound, we first use the boundary conditions in \eqref{eqpot} to rewrite \eqref{valatt1} as follows:
\begin{align*}
\mathbb{E}_{\metadue}[\tau_{\{\metauno^1,...,\metauno^n,\stab\}}]
=
\frac{1}{\rr{cap}(\metadue,\{\metauno^1,...,\metauno^n,\stab\})}
&\Big[
 \sum_{{x\in \mathcal{X} \setminus \{\metauno^1,...,\metauno^n,\stab\},}\atop{H(x)\leq H(\metadue)}}
\mu_\beta(x) h_{\metadue,\{\metauno^1,...,\metauno^n,\stab\}}(x) \\
 &+\sum_{{x\in \mathcal{X} \setminus \{\metauno^1,...,\metauno^n,\stab\},}\atop{H(x)> H(\metadue)}}
\mu_\beta(x) h_{\metadue,\{\metauno^1,...,\metauno^n,\stab\}}(x)
\Big].
\end{align*}
 Next we bound $\mu_{\beta}(x)$ as $\mu_\beta(x) \leq \mu_\beta(\metadue)\exp (-\beta \delta)$ for some positive $\delta=\min_x \{H(x) -H(x_2)\}$ and for any $x\in \mathcal{X}$ such that $H(x)>H(\metadue)$. We get
\begin{align}
\label{valmetadue}
\mathbb{E}_{\metadue}[\tau_{\{\metauno^1,...,\metauno^n,\stab\}}]
\simeq
\frac{1}{\rr{cap}(\metadue,\{\metauno^1,...,\metauno^n,\stab\})}
&\Big[
\!\!
\!\!
 \sum_{{x\in \mathcal{X} \setminus (\mathcal{X}^m \cup \mathcal{X}^s),}\atop{H(x)\leq H(\metadue)}}
\!\!
\!\!
 \mu_\beta(x)\,h_{\metadue,\{\metauno^1,...,\metauno^n,\stab\}}(x) \notag \\
&+\mu_\beta(\metadue)[1+o(1)]
\Big].
\end{align}
 Next we upper bound the equilibrium potential $h_{\metadue,\{\metauno^1,...,\metauno^n,\stab\}}(x)$ by applying Proposition \ref{hfun} with $x=x$, $Y=\{x_2\}$, $Z=\{x_1^1,...,x_1^n, \stab\}$, as
\begin{displaymath} h_{\metadue,\{\metauno^1,...,\metauno^n,\stab\}}(x)
\leq\frac{\rr{cap}(x,\metadue)}{\rr{cap}(x,\{\metauno^1,...,\metauno^n,\stab\})}
\;\;.
\end{displaymath}
 Furthermore, if $H(x)\le H(\metadue)$ and $x \notin \mathcal{X}^m \cup \mathcal{X}^s$, then
\begin{eqnarray*}
h_{\metadue,\{\metauno^1,...,\metauno^n,\stab\}}(x)
&\leq&
 C_1
 \frac{e^{-\beta\Phi(x,\metadue)}}{e^{-\beta\Phi(x,\{\metauno^1,...,\metauno^n,\stab\})}}\leq C_1
\frac{e^{-\beta(\Gamma_m+H(\metadue))}}{e^{-\beta(\Gamma_m+H(x)-\delta_1)}}
=
 C_1 e^{-\beta\delta_1}\frac{\mu_\beta(\metadue)}{\mu_\beta(x)},
\end{eqnarray*}
 where $C_1,\delta_1$ are suitable positive constants. In the first inequality we used Proposition \ref{t:apriori}, in the second we used Lemma~\ref{stimePhi} and Lemma \ref{stimePhi2}. By using \eqref{valmetadue} we get
\begin{displaymath}
\mathbb{E}_{\metadue}[\tau_{\{\metauno^1,...,\metauno^n,\stab\}}]
\le
\frac{1}{\rr{cap}(\metadue,\{\metauno^1,...,\metauno^n,\stab\})}
\Big[
\!\!
\!\!
 \sum_{{x\in \mathcal{X} \setminus (\mathcal{X}^m \cup \mathcal{X}^s),}\atop{H(x)\leq H(\metadue)}}
\!\!
\!\!
 C_1\mu_\beta(x) e^{-\beta\delta_1}\frac{\mu_\beta(\metadue)}{\mu_\beta(x)}+ \mu_\beta(\metadue)[1+o(1)]
\Big],
\end{displaymath}
which implies
\begin{equation}
\label{upboundval}
\mathbb{E}_{\metadue}[\tau_{\{\metauno^1,...,\metauno^n,\stab\}}]
\le
\frac{\mu_\beta(\metadue)}{\rr{cap}(\metadue,\{\metauno^1,...,\metauno^n,\stab\})}[1+o(1)],
\end{equation}
where we have used that the configuration space is finite.
 Equation \eqref{Egen1} finally follows by \eqref{downboundval} and \eqref{upboundval}.

Next we prove Equation \eqref{Egen2}.
 Recalling \eqref{cap-prop} above, we rewrite the expected value in terms of the capacity as
\begin{equation}
\label{defvalatt}
\mathbb{E}_{\{\metauno^1,...,\metauno^n\}}[\tau_{\stab}]
= \frac{1}{\rr{cap}_\beta(\{\metauno^1,...,\metauno^n\},\stab)}
\sum_{x\in \mathcal{X}}
\mu_\beta(x) \, h_{\{\metauno^1,...,\metauno^n\},\stab}(x).
\end{equation}
 Considering the contribution of $\metauno^r$ for every $r=1,...,n$ in the sum and observing that $h_{\{\metauno^1,...,\metauno^n\},\stab}(\metauno^q)=1$ for every $q=1,...,n$, we get the following lower bound:
\begin{align}
\label{downval2}
\mathbb{E}_{\{\metauno^1,...,\metauno^n\}}[\tau_{\stab}]
 & \ge \frac{1}{\rr{cap}(\{\metauno^1,...,\metauno^n\},\stab)} \sum_{q=1}^n\mu_\beta(\metauno^q) h_{\{\metauno^1,...,\metauno^n\},\stab}(\metauno^q) \notag \\
 & =\frac{1}{\rr{cap}(\{\metauno^1,...,\metauno^n\},\stab)} \sum_{q=1}^n\mu_\beta(\metauno^q) \notag \\
 & =\frac{\mu_\beta(\{\metauno^1,...,\metauno^n\})}{\rr{cap}(\{\metauno^1,...,\metauno^n\},\stab)},
\end{align}
 where the last equality follows from the definition of Gibbs-measure and $H(\metauno^r)=H(\metauno^q)$ for every $r,q =1,...,n$.
In order to give an upper bound, we first use the
 boundary conditions in \eqref{eqpot} to rewrite \eqref{defvalatt} as follows:
\begin{align}
\mathbb{E}_{\{\metauno^1,...,\metauno^n\}}[\tau_{\stab}]
=
\frac{1}{\rr{cap}(\{\metauno^1,...,\metauno^n\},\stab)}
\Big[ &
\sum_{{x\in \mathcal{X}\setminus \stab,}\atop{H(x)\leq H(\metauno^r)}}
\mu_\beta(x) h_{\{\metauno^1,...,\metauno^n\},\stab}(x) \notag \\ +
& \sum_{{x\in \mathcal{X}\setminus \stab,}\atop{H(x)> H(\metauno^r)}}
\mu_\beta(x) h_{\{\metauno^1,...,\metauno^n\}\stab}(x)
\Big].
\end{align}
 Next we bound $\mu_{\beta}(x)$ as $\mu_\beta(x) \leq \mu_\beta(\metauno^r)\exp (-\beta \delta)$ for some positive $\delta=\min_x \{H(x) -H(\metauno^r)\}$ and for any $x\in \mathcal{X}$ such that $H(x)>H(\metauno^r)$. We get
\begin{equation}
\label{tantival}
\mathbb{E}_{\{\metauno^1,...,\metauno^n\}}[\tau_{\stab}]
=
\frac{1}{\rr{cap}(\{\metauno^1,...,\metauno^n\},\stab)}
\Big[
\!\!
\!\!
 \sum_{{x\in \mathcal{X}\setminus \{ \metauno^1,...,\metauno^n,\stab\},}\atop{H(x)\leq H(\metauno^r)}}
\!\!
\!\!
\mu_\beta(x)\,h_{\{\metauno^1,...,\metauno^n\},\stab}(x)+
n\mu_\beta(\metauno^r)[1+o(1)]
\Big].
\end{equation}
 Next we upper bound the equilibrium potential $h_{\{\metauno^1,...,\metauno^n\},\stab}(x)$ by applying Proposition \ref{hfun} with $x=x$, $Y=\{x_1^1,...,x_1^n\}$ and $Z=\{\stab\}$
\begin{displaymath}
h_{\{\metauno^1,...,\metauno^n\},\stab}(x)\leq\frac{\rr{cap}(x,\{\metauno^1,...,\metauno^n\})}{\rr{cap}(x,\stab)}
\;\;.
\end{displaymath}
 Furthermore, if $H(x)\le H(\metauno^r)$ and $x \notin \{\metauno^1,...,\metauno^n,\stab\}$, then
\begin{eqnarray*}
h_{\{\metauno^1,...,\metauno^n\},\stab}(x)
&\leq&
 C_2
 \frac{e^{-\beta\Phi(x,\{\metauno^1,...,\metauno^n\})}}{e^{-\beta\Phi(x,\stab)}}\leq C_2
\frac{e^{-\beta(\Gamma_m+H(\metauno^r))}}{e^{-\beta(\Gamma_m+H(x)-\delta_2)}}
=
 C_2 e^{-\beta\delta_2}\frac{\mu_\beta(\metauno^r)}{\mu_\beta(x)},
\end{eqnarray*}
 where $C_2,\delta_2$ are suitable positive constants. In the first inequality we used Proposition \ref{t:apriori}, in the second we used Lemma~\ref{stimePhi} and Lemma \ref{stimePhi2}. By using \eqref{tantival} we get
\begin{displaymath}
\mathbb{E}_{\{\metauno^1,...,\metauno^n\}}[\tau_{\stab}]
\le
\frac{1}{\rr{cap}(\{\metauno^1,...,\metauno^n\},\stab)}
\Big[
\!\!
\!\!
 \sum_{{x\in \mathcal{X}\setminus \{\metauno^1,...,\metauno^n,\stab\},}\atop{H(x)\leq H(\metauno^r)}}
\!\!
\!\!
 C_2\mu_\beta(x) e^{-\beta\delta_2}\frac{\mu_\beta(\metauno^r)}{\mu_\beta(x)}+
n\mu_\beta(\metauno^r)[1+o(1)]
\Big],
\end{displaymath}
which implies
\begin{equation}
\label{upval2}
\mathbb{E}_{\{\metauno^1,...,\metauno^n\}}[\tau_{\stab}]
\le
\frac{n\mu_\beta(\metauno^r)}{\rr{cap}(\{\metauno^1,...,\metauno^n\},\stab)}[1+o(1)],
\end{equation}
 where we have used that the configuration space is finite. Equation \eqref{Egen2} finally follows recalling
 $n\mu_\beta(\metauno^r)=\mu_\beta(\{\metauno^1,...\metauno^n\})$ and by \eqref{downval2} and \eqref{upval2}.

Next we prove Equation \eqref{valattsing}.
 Recalling \eqref{cap-prop} above, we rewrite the expected value in terms of the capacity as
\begin{equation}
\label{meant1}
\mathbb{E}_{\metauno^r}[\tau_{\stab}]
= \frac{1}{\rr{cap}_\beta(\metauno^r,\stab)}
\sum_{x\in \mathcal{X}}
 \mu_\beta(x) \, h_{\metauno^r,\stab}(x) \qquad \text{for every } r=1,...,n.
\end{equation}
 Considering the contribution of every $\metauno^r$ in the sum and observing that $h_{\metauno^r,\stab}(\metauno^r)=1$ and $h_{\metauno^r,\stab}(\metauno^q) \simeq 1$ for every $q=1,...,n$ , we get the following lower bound:
\begin{align}
\label{meant_lower}
\mathbb{E}_{\metauno^r}[\tau_{\stab}]
 & \ge \frac{1}{\rr{cap}(\metauno^r,\stab)} \mu_\beta(\metauno^r) h_{\metauno^r,\stab}(\metauno^r)+\sum_{{q=1,}\atop{q \neq r}}^n\frac{1}{\rr{cap}(\metauno^r,\stab)} \mu_\beta(\metauno^q) h_{\metauno^r,\stab}(\metauno^q) \notag \\
 & \simeq \frac{1}{\rr{cap}(\metauno^r,\stab)}\sum_{q=1}^n \mu_\beta(\metauno^q) \notag \\
& =\frac{n\mu_\beta(\metauno^r)}{\rr{cap}(\metauno^r,\stab)},
\end{align}
 where the last equality follows from the definition of Gibbs-measure and $H(\metauno^r)=H(\metauno^q)$ for every $q=1,...,n$. In order to give an upper bound, we first use the boundary conditions in \eqref{eqpot} to rewrite \eqref{meant1} as follows:
\begin{displaymath}
\mathbb{E}_{\metauno^r}[\tau_{\stab}]
=
\frac{1}{\rr{cap}(\metauno^r,\stab)}
\Big[
\sum_{{x\in \mathcal{X} \setminus \stab,}\atop{H(x)\leq H(\metauno^r)}}
\mu_\beta(x) h_{\metauno^r,\stab}(x)+
\sum_{{x\in \mathcal{X} \setminus \stab,}\atop{H(x)> H(\metauno^r)}}
\mu_\beta(x) h_{\metauno^r,\stab}(x)
\Big].
\end{displaymath}
 Next we bound $\mu_{\beta}(x)$ as $\mu_\beta(x) \leq \mu_\beta(\metauno^r)\exp (-\beta \delta)$ for some positive $\delta=\min_x \{H(x) -H(x_1^r)\}$ and for any $x\in \mathcal{X}$ such that $H(x)>H(\metauno^r)$. Recalling that
 $h_{\metauno^r,\stab}(\metauno^r)=1$, $h_{\metauno^r,\stab}(\metauno^q)=1+o(1)$ for every $q=1,...,n$ with $q \neq r$, we get
\begin{equation}
\label{inpiu}
\mathbb{E}_{\metauno^r}[\tau_{\stab}]
\simeq
\frac{1}{\rr{cap}(\metauno^r,\stab)}
\Big[
\!\!
\!\!
 \sum_{{x\in \mathcal{X} \setminus \{\metauno^1,...,\metauno^n,\stab\},}\atop{H(x)\leq H(\metauno^r)}}
\!\!
\!\!
\mu_\beta(x)\,h_{\metauno^r,\stab}(x)+
\sum_{q=1}^n \mu_\beta(\metauno^q)[1+o(1)]
\Big].
\end{equation}
 Next we upper bound the equilibrium potential $h_{\metauno^r,\stab}(x)$ by applying Proposition \ref{hfun} with $x=x$, $Z=\{\stab\}$ and $Y=\{\metauno^r\}$ for every $i=1,...,n$
\begin{displaymath}
h_{\metauno^r,\stab}(x)\leq\frac{\rr{cap}(x,\metauno^r)}{\rr{cap}(x,\stab)}
\;\;.
\end{displaymath}
 Furthermore, if $H(x)\le H(\metauno^r)$ and $x \neq \metauno^q$ for every $q=1,...,n$, then
\begin{eqnarray*}
h_{\metauno^r,\stab}(x)
&\leq&
 C_3
\frac{e^{-\beta\Phi(x,\metauno^r)}}{e^{-\beta\Phi(x,\stab)}}\leq C_3
\frac{e^{-\beta(\Gamma_m+H(\metauno^r))}}{e^{-\beta(\Gamma_m+H(x)-\delta_3)}}
=
 C_3 e^{-\beta\delta_3}\frac{\mu_\beta(\metauno^r)}{\mu_\beta(x)},
\end{eqnarray*}
 where $C_3,\delta_3$ are suitable positive constants. In the first inequality we used Proposition \ref{t:apriori}, in the second we used Lemma~\ref{stimePhi} and Lemma \ref{stimePhi2}. By using \eqref{inpiu} we get

\begin{displaymath}
\mathbb{E}_{\metauno^r}[\tau_{\stab}]
\le
\frac{1}{\rr{cap}(\metauno^r,\stab)}
\Big[
\!\!
\!\!
 \sum_{{x\in \mathcal{X} \setminus \{\metauno^1,...,\metauno^n,\stab\},}\atop{H(x)\leq H(\metauno^r)}}
\!\!
\!\!
 C_3\mu_\beta(x) e^{-\beta\delta_3}\frac{\mu_\beta(\metauno^r)}{\mu_\beta(x)}+
\sum_{q=1}^n \mu_\beta(\metauno^q)[1+o(1)]
\Big],
\end{displaymath}
which implies
\begin{equation}
\label{uppot}
\mathbb{E}_{\metauno^r}[\tau_{\stab}]
\le
\frac{\sum_{q=1}^n\mu_\beta(\metauno^q)}{\rr{cap}(\metauno^r,\stab)}[1+o(1)]=
\frac{n\mu_\beta(\metauno^r)}{\rr{cap}(\metauno^r,\stab)}[1+o(1)],
\end{equation}
 where we have used that the configuration space is finite and $H(\metauno^r)=H(\metauno^q)$ for every $q=1,...,n$.

\medskip
\par\noindent
 \textit{Proof of Theorem~\ref{t:addition01} and Theorem~\ref{t:addition03}.\/}
 The two theorems follow immediately by exploiting Condition~\ref{t:series02} and applying Theorem~\ref{t:meant}. \\
\qed

The proof of Theorem~\ref{t:addition02} is based on the following lemma.

\begin{lem}
\label{t:dimo00}
 Given three or more states $y,w^1,...,w^n,z\in \mathcal{X}$ pairwise mutually different, we have that the following holds
\begin{equation}
\label{dimo00}
\mathbb{E}_y[\tau_z]
=
\mathbb{E}_y[\tau_{\{w^1,...,w^n,z\}}]
+
\mathbb{E}_{\{w^1,...,w^n\}}[\tau_z]\mathbb{P}_y(\tau_{\{w^1,...,w^n\}}<\tau_z).
\end{equation}
\end{lem}

\medskip
\par\noindent
\textit{Proof.\/}
First of all we note that
\begin{eqnarray*}
\mathbb{E}_y(\tau_z)=\mathbb{E}_y[\tau_z\mathbf{1}_{\{\tau_{w^1,...,w^n}\}<\tau_z}]+\mathbb{E}_y[\tau_z\mathbf{1}_{\tau_{\{w^1,...,w^n\}}\geq\tau_z}].
\end{eqnarray*}
We now rewrite the first term as follows
\begin{align*}
 \mathbb{E}_y[\tau_z\mathbf{1}_{\{\tau_{\{w^1,...,w^n\}}<\tau_z\}}]&=\mathbb{E}_y[ \mathbb{E}_y[\tau_z\mathbf{1}_{\{\tau_{\{w^1,...,w^n\}}<\tau_z\}}|\mathcal{F}_{\tau_{\{w^1,...,w^n\}}}]
] \\
&=\mathbb{E}_y[\mathbf{1}_{\{\tau_{\{w^1,...,w^n\}}<\tau_z\}}(\tau_{\{w^1,...,w^n\}}+\mathbb{E}_{\{w^1,...,w^n\}}[\tau_z])]\\
&= \mathbb{E}_y[\tau_{\{w^1,...,w^n\}}\mathbf{1}_{
\{\tau_{\{w^1,...,w^n\}}<\tau_z\}}]+\mathbb{P}_y(\tau_{\{w^1,...,w^n\}}<\tau_z)\mathbb{E}_{\{w^1,...,w^n\}}[\tau_z],
\end{align*}
 where we have used the fact that $\tau_{\{w^1,...,w^n\}}=\min\{\tau_{w^1},...,\tau_{w^n}\}$ is a stopping time, that $\mathbf{1}_{\{\tau_{\{w^1,...,w^n\}}\}}$ is measurable with respect to the pre--$\tau_{\{w^1,...,w^n\}}$--$\sigma$--algebra $\mathcal{F}_{\tau_{\{w^1,...,w^n\}}}$ and the strong Markov property which gives $\mathbb{E}_y[\tau_z|\mathcal{F}_{\tau_{\{w^1,...,w^n\}}}]=\tau_{\{w^1,...,w^n\}}+
 \mathbb{E}_{\{w^1,...,w^n\}}[\tau_z]$ on the event $\{\tau_{\{w^1,...,w^n\}}\leq \tau_z\}$.
 Since $ (\tau_{\{w^1,...,w^n\}}\mathbf{1}_{\{ \tau_{\{w^1,...,w^n\}}<\tau_z \}}+\tau_z\mathbf{1}_{\{ \tau_{\{w^1,...,w^n\}}\geq\tau_z \}})=\tau_{\{w^1,...,w^n,z\}}$, \eqref{dimo00} follows. \\
\qed

\medskip
\par\noindent
\textit{Proof of Theorem~\ref{t:addition02}.\/}
By \eqref{dimo00} we have that
\begin{displaymath}
\mathbb{E}_{\metadue}[\tau_{\stab}]
=
\mathbb{E}_{\metadue}[\tau_{\{\metauno^1,..., \metauno^n\stab\}}]
+
\mathbb{E}_{\{\metauno^1,...,\metauno^n\}}[\tau_{\stab}]\mathbb{P}_{\metadue}(\tau_{\{\metauno^1,...,\metauno^n\}}<\tau_{\stab})
\end{displaymath}
 By Theorem~\ref{t:addition01} and Condition~\ref{t:series01} it follows that
$$
{\mathbb{E}_{\metadue}[\tau_{\stab}]}=
     {e^{\beta\Gamma_m}\left(\frac{1}{k_1}+\frac{1}{k_2}\right)}
[1+o(1)]
$$
which concludes the proof.
\qed

\begin{teo}\label{hfun}
Consider the Markov chain defined in Section \ref{gs}. We have that
\begin{equation}
\label{potbound}
 \mathbb{P}_x(\tau_Y<\tau_Z)\leq \frac{\rr{cap}_\beta(x,Y)}{\rr{cap}_\beta(x,Z)}
\end{equation}
 for any $Y=\{y^1,...,y^t\}\subset \mathcal{X}$ for $t \in \mathbb{N}$, $Z=\{z^1,...,z^{t'}\}\subset \mathcal{X}$ for $t' \in \mathbb{N}$, $Y \cap Z=\emptyset$, $x\in \mathcal{X} \setminus \{Y \cup Z\}$.
\end{teo}
 \proof Given $Y,Z \subset \mathcal{X}$ such that $Y \cap Z=\emptyset$ and $x\in \mathcal{X} \setminus \{Y \cup Z\}$, a renewal argument and the strong Markov property yield
\begin{eqnarray*}
 \mathbb{P}_x(\tau_Y<\tau_Z)&=&\mathbb{P}_x(\tau_Y<\tau_Z, \tau_{Y \cup Z}>\tau_{x})+\mathbb{P}_x(\tau_Y<\tau_Z, \tau_{Y \cup Z}<\tau_{x})\\
 &=&\mathbb{P}_x(\tau_Y<\tau_Z| \tau_{Y \cup Z}>\tau_{x})\mathbb{P}_x(\tau_{Y \cup Z}> \tau_{x})\\
&+&\mathbb{P}_x(\tau_Y<\tau_Z, \tau_{Y \cup Z}<\tau_{x})\\
 &=&\mathbb{P}_x(\tau_Y<\tau_Z)\mathbb{P}_x(\tau_{Y \cup Z}> \tau_{x})+\mathbb{P}_x(\tau_Y<\tau_Z, \tau_Y<\tau_{x})\\
 &=&\mathbb{P}_x(\tau_Y<\tau_Z)\mathbb{P}_x(\tau_{Y \cup Z}> \tau_{x})+\mathbb{P}_x(\tau_Y<\tau_{Z \cup\{x\}}).
\end{eqnarray*}
 Therefore
\begin{displaymath}
\mathbb{P}_x(\tau_Y<\tau_Z)
\!=\!
\frac{\mathbb{P}_x(\tau_Y<\tau_{Z \cup \{x\}})}
{1-\mathbb{P}_x(\tau_{Y \cup Z}>\tau_{x})}=
\frac{\mathbb{P}_x(\tau_Y<\tau_{Z \cup \{x\}})}
{\mathbb{P}_x(\tau_{Y \cup Z} \leq \tau_{x})}
\leq \frac{\mathbb{P}_x(\tau_Y<\tau_{x})}
{\mathbb{P}_x(\tau_Z <\tau_{x})}.
\end{displaymath}
 Recalling \eqref{cap-prop}, we can rewrite the ratio in terms of ratio of capacities:
\begin{displaymath}
\label{cri02:023}
\frac{\mathbb{P}_x(\tau_Y<\tau_{x})}
 {\mathbb{P}_x(\tau_Z <\tau_{x})}=\frac{\rr{cap}_\beta(x,Y)}{\rr{cap}_\beta(x,Z)}.
\end{displaymath}
Hence, we get Equation \eqref{potbound}.
\begin{teo}{\cite[Lemma 3.1.1]{bovier2006sharp}}
\label{t:apriori}
Consider the Markov chain defined in Section \ref{gs}.
 For every not empty disjoint sets $Y,Z\subset X$ there exist constants $0<C_1<C_2<\infty$ such that
\begin{equation}
\label{r:apriori}
C_1\le e^{\beta\Phi(Y,Z)}\, Z_\beta\, \rr{cap}_\beta(Y,Z)\le C_2,
\end{equation}
for all $\beta$ large enough.
\end{teo}

\bibliographystyle{abbrv}
\bibliography{BJNArchive2}

\begin{thebibliography}{10}

\bibitem{arous1996metastability}
G.~B. Arous and R.~Cerf.
\newblock Metastability of the three dimensional {Ising} model on a torus at
  very low temperatures.
\newblock {\em Electronic Journal of Probability}, 1, 1996.

\bibitem{bashiri2017note}
K.~Bashiri.
\newblock A note on the metastability in three modifications of the standard
  {Ising} model.
\newblock {\em arXiv preprint arXiv:1705.07012}, 2017.

\bibitem{beltran2010tunneling}
J.~Beltran and C.~Landim.
\newblock Tunneling and metastability of continuous time markov chains.
\newblock {\em Journal of Statistical Physics}, 140(6):1065--1114, 2010.

\bibitem{beltran2012tunneling}
J.~Beltr{\'a}n and C.~Landim.
\newblock Tunneling and metastability of continuous time markov chains ii, the
  nonreversible case.
\newblock {\em Journal of Statistical Physics}, 149(4):598--618, 2012.

\bibitem{bianchi2016metastable}
A.~Bianchi and A.~Gaudilliere.
\newblock Metastable states, quasi-stationary distributions and soft measures.
\newblock {\em Stochastic Processes and their Applications}, 126(6):1622--1680,
  2016.

\bibitem{bigelis1999critical}
S.~Bigelis, E.~N. Cirillo, J.~L. Lebowitz, and E.~R. Speer.
\newblock Critical droplets in metastable states of probabilistic cellular
  automata.
\newblock {\em Physical Review E}, 59(4):3935, 1999.

\bibitem{bovier2016metastability}
A.~Bovier and F.~{Den}~Hollander.
\newblock {\em Metastability: a potential-theoretic approach}, volume 351.
\newblock Springer, 2016.

\bibitem{bovier2006sharp}
A.~Bovier, F.~{Den}~Hollander, and F.~R. Nardi.
\newblock Sharp asymptotics for {Kawasaki} dynamics on a finite box with open
  boundary.
\newblock {\em Probability theory and related fields}, 135(2):265--310, 2006.

\bibitem{bovier2010homogeneous}
A.~Bovier, F.~{Den}~Hollander, C.~Spitoni, et~al.
\newblock Homogeneous nucleation for {Glauber} and {Kawasaki} dynamics in large
  volumes at low temperatures.
\newblock {\em The Annals of Probability}, 38(2):661--713, 2010.

\bibitem{bovier2002metastability}
A.~Bovier, M.~Eckhoff, V.~Gayrard, and M.~Klein.
\newblock Metastability and low lying spectra in reversible {Markov} chains.
\newblock {\em Communications in mathematical physics}, 228(2):219--255, 2002.

\bibitem{bovier2004metastability}
A.~Bovier, M.~Eckhoff, V.~Gayrard, and M.~Klein.
\newblock Metastability in reversible diffusion processes {I}. {Sharp}
  asymptotics for capacities and exit times.
\newblock 2004.

\bibitem{boviermanzo2002metastability}
A.~Bovier and F.~Manzo.
\newblock Metastability in {Glauber} dynamics in the low-temperature limit:
  beyond exponential asymptotics.
\newblock {\em Journal of Statistical Physics}, 107(3-4):757--779, 2002.

\bibitem{cassandro1984metastable}
M.~Cassandro, A.~Galves, E.~Olivieri, and M.~E. Vares.
\newblock Metastable behavior of stochastic dynamics: a pathwise approach.
\newblock {\em Journal of Statistical Physics}, 35(5-6):603--634, 1984.

\bibitem{catoni1999simulated}
O.~Catoni.
\newblock Simulated annealing algorithms and {Markov} chains with rare
  transitions.
\newblock In {\em S{\'e}minaire de probabilit{\'e}s XXXIII}, pages 69--119.
  Springer, 1999.

\bibitem{catoni1997exit}
O.~Catoni and R.~Cerf.
\newblock The exit path of a {Markov} chain with rare transitions.
\newblock {\em ESAIM: Probability and Statistics}, 1:95--144, 1997.

\bibitem{catoni1992parallel}
O.~Catoni and A.~Trouv{\'e}.
\newblock Parallel annealing by multiple trials: a mathematical study.
\newblock {\em Simulated annealing}, pages 129--143, 1992.

\bibitem{cerf2013nucleation}
R.~Cerf and F.~Manzo.
\newblock Nucleation and growth for the {Ising} model in $ d $ dimensions at
  very low temperatures.
\newblock {\em The Annals of Probability}, 41(6):3697--3785, 2013.

\bibitem{cirillo1998metastability}
E.~N.~M. Cirillo and J.~L. Lebowitz.
\newblock Metastability in the two-dimensional {Ising} model with free boundary
  conditions.
\newblock {\em Journal of Statistical Physics}, 90(1-2):211--226, 1998.

\bibitem{cirillo2003metastability}
E.~N.~M. Cirillo and F.~R. Nardi.
\newblock Metastability for a stochastic dynamics with a parallel heat bath
  updating rule.
\newblock {\em Journal of Statistical Physics}, 110(1-2):183--217, 2003.

\bibitem{cirillo2013relaxation}
E.~N.~M. Cirillo and F.~R. Nardi.
\newblock Relaxation height in energy landscapes: an application to multiple
  metastable states.
\newblock {\em Journal of Statistical Physics}, 150(6):1080--1114, 2013.

\bibitem{cirillo2015metastability}
E.~N.~M. Cirillo, F.~R. Nardi, and J.~Sohier.
\newblock Metastability for general dynamics with rare transitions: escape time
  and critical configurations.
\newblock {\em Journal of Statistical Physics}, 161(2):365--403, 2015.

\bibitem{cirillo2008competitive}
E.~N.~M. Cirillo, F.~R. Nardi, and C.~Spitoni.
\newblock Competitive nucleation in reversible {Probabilistic} {Cellular}
  {Automata}.
\newblock {\em Physical Review E}, 78(4):040601, 2008.

\bibitem{cirillo2008metastability}
E.~N.~M. Cirillo, F.~R. Nardi, and C.~Spitoni.
\newblock Metastability for reversible {Probabilistic} {Cellular} {Automata}
  with self-interaction.
\newblock {\em Journal of Statistical Physics}, 132(3):431--471, 2008.

\bibitem{cirillo2016sum}
E.~N.~M. Cirillo, F.~R. Nardi, and C.~Spitoni.
\newblock Sum of exit times in series of metastable states in {Probabilistic}
  {Cellular} {Automata}.
\newblock In {\em International Workshop on Cellular Automata and Discrete
  Complex Systems}, pages 105--119. Springer, 2016.

\bibitem{cirillo2017sum}
E.~N.~M. Cirillo, F.~R. Nardi, and C.~Spitoni.
\newblock Sum of exit times in a series of two metastable states.
\newblock {\em The European Physical Journal Special Topics},
  226(10):2421--2438, 2017.

\bibitem{cirillo1996metastability}
E.~N.~M. Cirillo and E.~Olivieri.
\newblock Metastability and nucleation for the {Blume}-{Capel} model. different
  mechanisms of transition.
\newblock {\em Journal of Statistical Physics}, 83(3-4):473--554, 1996.

\bibitem{dehghanpour1997metropolis}
P.~Dehghanpour and R.~H. Schonmann.
\newblock Metropolis dynamics relaxation via nucleation and growth.
\newblock {\em Communications in mathematical physics}, 188(1):89--119, 1997.

\bibitem{den2003droplet}
F.~{Den}~Hollander, F.~R. Nardi, E.~Olivieri, and E.~Scoppola.
\newblock Droplet growth for three-dimensional {Kawasaki} dynamics.
\newblock {\em Probability theory and related fields}, 125(2):153--194, 2003.

\bibitem{den2012metastability}
F.~{Den}~Hollander, F.~R. Nardi, and A.~Troiani.
\newblock Metastability for low--temperature {Kawasaki} dynamics with two types
  of particles.
\newblock {\em Electronic Journ. of Probability}, 17:1--26, 2012.

\bibitem{derrida1989dynamical}
B.~Derrida.
\newblock Dynamical phase transitions in spin models and automata.
\newblock Technical report, CEA Centre d'Etudes Nucleaires de Saclay, 1989.

\bibitem{gaudilliere2009condenser}
A.~Gaudilli{\`e}re.
\newblock Condenser physics applied to {Markov} chains.
\newblock {\em Lecture Notes for the 12th Brazilian School of Probability},
  2009.

\bibitem{gaudilliere2009ideal}
A.~Gaudilli{\`e}re, F.~{Den}~Hollander, F.~R. Nardi, E.~Olivieri, and
  E.~Scoppola.
\newblock Ideal gas approximation for a two-dimensional rarefied gas under
  {Kawasaki} dynamics.
\newblock {\em Stochastic Processes and their Applications}, 119(3):737--774,
  2009.

\bibitem{gaudillierelandim2014}
A.~Gaudilliere and C.~Landim.
\newblock {A Dirichlet principle for non reversible Markov chains and some
  recurrence theorems}.
\newblock {\em {Probability Theory and Related Fields}}, 158:55--89, 2014.

\bibitem{gaudilliere2020asymptotic}
A.~Gaudilli{\`e}re, P.~Milanesi, and M.~E. Vares.
\newblock Asymptotic exponential law for the transition time to equilibrium of
  the metastable kinetic {Ising} model with vanishing magnetic field.
\newblock {\em Journal of Statistical Physics}, pages 1--46, 2020.

\bibitem{gaudilliere2010upper}
A.~Gaudilliere and F.~R. Nardi.
\newblock An upper bound for front propagation velocities inside moving
  populations.
\newblock {\em Brazilian Journal of Probability and Statistics},
  24(2):256--278, 2010.

\bibitem{gaudilliere2005nucleation}
A.~Gaudilliere, E.~Olivieri, and E.~Scoppola.
\newblock Nucleation pattern at low temperature for local {Kawasaki} dynamics
  in two dimensions.
\newblock {\em Markov Processes Relat. Fields}.

\bibitem{hollander2000metastability}
F.~D. Hollander, E.~Olivieri, and E.~Scoppola.
\newblock Metastability and nucleation for conservative dynamics.
\newblock {\em Journal of Mathematical Physics}, 41(3):1424--1498, 2000.

\bibitem{holley1988simulated}
R.~Holley and D.~Stroock.
\newblock Simulated annealing via {Sobolev} inequalities.
\newblock {\em Communications in Mathematical Physics}, 115(4):553--569, 1988.

\bibitem{kotecky1994shapes}
R.~Koteck{\`y} and E.~Olivieri.
\newblock Shapes of growing droplets—a model of escape from a metastable
  phase.
\newblock {\em Journal of Statistical Physics}, 75(3-4):409--506, 1994.

\bibitem{manzo2004essential}
F.~Manzo, F.~R. Nardi, E.~Olivieri, and E.~Scoppola.
\newblock On the essential features of metastability: tunnelling time and
  critical configurations.
\newblock {\em Journal of Statistical Physics}, 115(1-2):591--642, 2004.

\bibitem{manzo1998relaxation}
F.~Manzo and E.~Olivieri.
\newblock Relaxation patterns for competing metastable states: a nucleation and
  growth model.
\newblock In {\em Markov Proc. Relat. Fields}, volume~4, pages 549--570, 1998.

\bibitem{manzo2001dynamical}
F.~Manzo and E.~Olivieri.
\newblock Dynamical {Blume}--{Capel} model: competing metastable states at
  infinite volume.
\newblock {\em Journal of Statistical Physics}, 104(5-6):1029--1090, 2001.

\bibitem{nardi1996low}
F.~R. Nardi and E.~Olivieri.
\newblock Low temperature stochastic dynamics for an {Ising} model with
  alternating field.
\newblock In {\em Markov Proc. Relat. Fields}, volume~2, pages 117--166, 1996.

\bibitem{nardi2012sharp}
F.~R. Nardi and C.~Spitoni.
\newblock Sharp asymptotics for stochastic dynamics with parallel updating
  rule.
\newblock {\em Journal of Statistical Physics}, 146(4):701--718, 2012.

\bibitem{nardi2016hitting}
F.~R. Nardi, A.~Zocca, and S.~C. Borst.
\newblock Hitting time asymptotics for hard-core interactions on grids.
\newblock {\em Journal of Statistical Physics}, 162(2):522--576, 2016.

\bibitem{neves1991critical}
E.~J. Neves and R.~H. Schonmann.
\newblock Critical droplets and metastability for a {Glauber} dynamics at very
  low temperatures.
\newblock {\em Communications in Mathematical Physics}, 137(2):209--230, 1991.

\bibitem{neves1992behavior}
E.~J. Neves and R.~H. Schonmann.
\newblock Behavior of droplets for a class of {Glauber} dynamics at very low
  temperature.
\newblock {\em Probability theory and related fields}, 91(3-4):331--354, 1992.

\bibitem{olivieri1995markov}
E.~Olivieri and E.~Scoppola.
\newblock Markov chains with exponentially small transition probabilities:
  first exit problem from a general domain {I}. {The} reversible case.
\newblock {\em Journal of Statistical Physics}, 79(3-4):613--647, 1995.

\bibitem{olivieri1996markov}
E.~Olivieri and E.~Scoppola.
\newblock Markov chains with exponentially small transition probabilities:
  first exit problem from a general domain. {II}. {The} general case.
\newblock {\em Journal of Statistical Physics}, 84(5-6):987--1041, 1996.

\bibitem{olivieri2005large}
E.~Olivieri and M.~E. Vares.
\newblock {\em Large deviations and metastability}, volume 100.
\newblock Cambridge University Press, 2005.

\bibitem{penrose1971rigorous}
O.~Penrose and J.~L. Lebowitz.
\newblock Rigorous treatment of metastable states in the {Van} der
  {Waals}-{Maxwell} theory.
\newblock {\em Journal of Statistical Physics}, 3(2):211--236, 1971.

\bibitem{schonmann1994slow}
R.~H. Schonmann.
\newblock Slow droplet-driven relaxation of stochastic {Ising} models in the
  vicinity of the phase coexistence region.
\newblock {\em Communications in Mathematical Physics}, 161(1):1--49, 1994.

\bibitem{schonmann1998wulff}
R.~H. Schonmann and S.~B. Shlosman.
\newblock Wulff droplets and the metastable relaxation of kinetic {Ising}
  models.
\newblock {\em Communications in mathematical physics}, 194(2):389--462, 1998.

\bibitem{scoppola1994metastability}
E.~Scoppola.
\newblock Metastability for {Markov} chains: a general procedure based on
  renormalization group ideas.
\newblock In {\em Probability and Phase Transition}, pages 303--322. Springer,
  1994.

\bibitem{trouve1996rough}
A.~Trouv{\'e}.
\newblock Rough large deviation estimates for the optimal convergence speed
  exponent of generalized simulated annealing algorithms.
\newblock In {\em Annales de l'IHP Probabilit{\'e}s et statistiques},
  volume~32, pages 299--348, 1996.

\end{thebibliography}
\end{document}